\documentclass[12pt]{amsart}
\usepackage{comment}
\usepackage[utf8]{inputenc} 
\usepackage[T1]{fontenc} 
\usepackage{geometry} 
\usepackage{upgreek}
\usepackage{amsfonts}
\usepackage{amssymb}
\usepackage{amsthm}
\usepackage{amsmath}
\usepackage{amscd}
\usepackage[shortlabels]{enumitem}
\usepackage{mathrsfs}
\usepackage{tikz}
\usetikzlibrary{calc,arrows,decorations.pathreplacing}
\usepackage{nicefrac, xfrac}
\usepackage{mathtools}
\usepackage[bottom]{footmisc}
\usepackage[mathcal]{euscript}
\usepackage{float}

\usepackage[pagebackref, colorlinks = true,
linkcolor = red,
urlcolor  = blue,
citecolor = blue,
anchorcolor = blue]{hyperref}
\hypersetup{pdftitle={\@title},pdfauthor={\@author}}

\theoremstyle{plain}

\newtheorem{theorem}{Theorem}[section]
\newtheorem{corollary}[theorem]{Corollary}
\newtheorem{lemma}[theorem]{Lemma}

\newtheorem{proposition}[theorem]{Proposition}

\newtheorem{claim}{Claim}[theorem]

\theoremstyle{definition}
\newtheorem{definition}[theorem]{Definition}
\newtheorem{observation}[theorem]{Observation}
\newtheorem{remark}[theorem]{Remark}

\newtheorem*{theorem*}{Theorem}

\newenvironment{subproof}[1][\proofname]{%
  \begin{proof}[#1]%
}{%
  \end{proof}%
}

\renewcommand{\phi}{\varphi}
\renewcommand{\epsilon}{\varepsilon}

\newcommand{\sub}{\ensuremath{ \subseteq}} 
\newcommand{\bus}{\ensuremath{ \supseteq}} 
\newcommand{\union}{\ensuremath{\bigcup}} 
\newcommand{\set}[1]{\ensuremath{ \left\{#1\right\} }} 
\newcommand{\norm}[1]{\ensuremath{ \left\lVert#1\right\rVert }} 
\newcommand{\absval}[1]{\ensuremath{ \left\lvert#1\right\rvert }} 
\newcommand{\spot}{\ensuremath{ \makebox[1ex]{\textbf{$\cdot$}} }}

\newcommand{\ol}[1]{\overline{#1}}
 
\newcommand{\bs}{\ensuremath{\backslash}}

\newcommand{\lt}{\ensuremath{\left}}
\newcommand{\rt}{\ensuremath{\right}}

\newcommand{\Om}{\ensuremath{\Omega}}
\newcommand{\Lm}{\ensuremath{\Lambda}}
\newcommand{\Gm}{\ensuremath{\Gamma}}

\newcommand{\Sg}{\ensuremath{\Sigma}}

\newcommand{\Dl}{\ensuremath{\Delta}}

\newcommand{\Ta}{\ensuremath{\Theta}}

\newcommand{\om}{\ensuremath{\omega}}
\newcommand{\lm}{\ensuremath{\lambda}}
\newcommand{\gm}{\ensuremath{\gamma}} \newcommand{\hgm}{\ensuremath{\hat\gamma}} 
\newcommand{\al}{\ensuremath{\alpha}} \newcommand{\hal}{\ensuremath{\hat\alpha}}
\newcommand{\bt}{\ensuremath{\beta}}  
\newcommand{\sg}{\ensuremath{\sigma}}
\newcommand{\ep}{\ensuremath{\epsilon}}
\newcommand{\dl}{\ensuremath{\delta}}
\newcommand{\kp}{\ensuremath{\kappa}}
\newcommand{\zt}{\ensuremath{\zeta}}

\newcommand{\vta}{\ensuremath{\vartheta}}
\newcommand{\vkp}{\ensuremath{\varkappa}}
\newcommand{\vrho}{\ensuremath{\varrho}}

\DeclareSymbolFont{bbold}{U}{bbold}{m}{n}
\DeclareSymbolFontAlphabet{\mathbbold}{bbold}

\newcommand{\ind}{\ensuremath{\mathbbold{1}}}
\DeclareMathOperator*{\esssup}{ess\sup}
\DeclareMathOperator*{\essinf}{ess\inf}

\DeclareMathOperator{\diam}{diam}

\DeclareMathOperator{\intr}{Int}

\newcommand{\cA}{\ensuremath{\mathcal{A}}}
\newcommand{\cB}{\ensuremath{\mathcal{B}}}
\newcommand{\cC}{\ensuremath{\mathcal{C}}}
\newcommand{\cD}{\ensuremath{\mathcal{D}}}
\newcommand{\cE}{\ensuremath{\mathcal{E}}}
\newcommand{\cF}{\ensuremath{\mathcal{F}}}

\newcommand{\cH}{\ensuremath{\mathcal{H}}}
\newcommand{\cI}{\ensuremath{\mathcal{I}}}
\newcommand{\cJ}{\ensuremath{\mathcal{J}}}

\newcommand{\cL}{\ensuremath{\mathcal{L}}}
\newcommand{\cM}{\ensuremath{\mathcal{M}}}

\newcommand{\cP}{\ensuremath{\mathcal{P}}}
\newcommand{\cQ}{\ensuremath{\mathcal{Q}}}

\newcommand{\cS}{\ensuremath{\mathcal{S}}}

\newcommand{\cU}{\ensuremath{\mathcal{U}}}
\newcommand{\cV}{\ensuremath{\mathcal{V}}}

\newcommand{\cZ}{\ensuremath{\mathcal{Z}}}

\newcommand{\sB}{\ensuremath{\mathscr{B}}}
\newcommand{\sC}{\ensuremath{\mathscr{C}}} \newcommand{\fC}{\ensuremath{\mathfrak{C}}}

\newcommand{\sF}{\ensuremath{\mathscr{F}}}

\newcommand{\NN}{\ensuremath{\mathbb N}}

\newcommand{\RR}{\ensuremath{\mathbb R}}

\newcommand{\ZZ}{\ensuremath{\mathbb Z}} 

\newcommand{\ceil}[1]{\ensuremath{\left\lceil#1 \right\rceil}}
\newcommand{\floor}[1]{\ensuremath{\left\lfloor#1 \right\rfloor}}

\newcommand{\maeom}{$m$ a.e. $\om\in\Om$}
\newcommand{\flag}[1]{\textbf{***[#1]***}}

\def\lra{\longrightarrow}
\def\var{\text{{\rm var}}}
\def\BV{\text{{\rm BV}}}

\def\lt{\left}
\def\rt{\right}

\providecommand{\phantomsection}{}
\AtBeginDocument{\let\textlabel\label}
\makeatletter
\newcommand{\mylabel}[2]{\raisebox{.7\normalbaselineskip}{\phantomsection}(#1)%
	\def\@currentlabel{#1}\textlabel{#2}}
\makeatother

\renewcommand\~{\tilde}
\newcommand{\hcL}{\ensuremath{\hat\cL}} 

\allowdisplaybreaks

\numberwithin{equation}{section}
\title[]{Thermodynamic Formalism for Random Weighted Covering Systems}

\author{Jason Atnip}
\address{School of Mathematics and Statistics, University of New South Wales, Sydney, NSW 2052, Australia}
\email{\href{j.atnip@unsw.edu.au}{j.atnip@unsw.edu.au} }

\author{Gary Froyland}
\address{School of Mathematics and Statistics, University of New South Wales, Sydney, NSW 2052, Australia}
\email{\href{g.froyland@unsw.edu.au}{g.froyland@unsw.edu.au} }

\author{Cecilia Gonz\'alez-Tokman}
\address{School of Mathematics and Physics, The University of Queensland, St Lucia, QLD 4072, Australia}
\email{\href{cecilia.gt@uq.edu.au}{cecilia.gt@uq.edu.au} }

\author{Sandro Vaienti}
\address{Aix Marseille Université, Université de Toulon, CNRS, CPT, 13009 Marseille, France}
\email{\href{vaienti@cpt.univ-mrs.fr}{vaienti@cpt.univ-mrs.fr} }

\begin{document}
\begin{abstract}
	We develop for the first time a quenched thermodynamic formalism for random dynamical systems generated by countably branched, piecewise-monotone mappings of the interval that satisfy a random covering condition.
	Given a random contracting potential $\varphi$ (in the sense of Liverani-Saussol-Vaienti), we prove there exists a unique random conformal measure $\nu_\varphi$ and unique random equilibrium state $\mu_\varphi$.
	Further, we prove quasi-compactness of the associated transfer operator cocycle and exponential decay of correlations for $\mu_\varphi$.
	Our random driving is generated by an invertible, ergodic, measure-preserving transformation $\sigma$ on a probability space $(\Omega,\mathscr{F},m)$;  for each $\omega\in\Omega$ we associate a piecewise-monotone, surjective map $T_\omega:I\to I$.
	We consider general potentials $\varphi_\omega:I\to\mathbb R\cup\{-\infty\}$ such that the weight function $g_\omega=e^{\varphi_\omega}$ is of bounded variation.
	We provide several examples of our general theory. 
	In particular, our results apply to new examples of linear and non-linear systems including random $\beta$-transformations, randomly translated random $\beta$-transformations, countably branched random Gauss-Renyi maps, random non-uniformly expanding maps (such as intermittent maps and maps with contracting branches) composed with expanding maps, and a large class of random Lasota-Yorke maps.
\end{abstract}


\maketitle
\tableofcontents

\section{Introduction}

Deterministic transitive expansive dynamics $T:[0,1]\to [0,1]$ or $T:S^1\to S^1$ have a well-developed ``thermodynamic formalism'', whereby the transfer operator with potential $\varphi$ of sufficient regularity has a unique invariant probability measure $\mu$ that is absolutely continuous with respect to a conformal measure, and $\mu$ arises as a maximizer of the sum of the metric entropy and the integral of the potential \cite{bowen75,Ruelle_Thermodynamicformalism_1978} (shifts of finite type, smooth dynamics), \cite{denker1990,liverani_conformal_1998,buzzi-sarig03} (piecewise smooth dynamics), \cite{Ruelle_Thermodynamicformalism_1978,przytycki_conformal_2010} (distance expanding mappings), \cite{mauldin_graph_2003,buzzi-sarig03,arbieto-matheus06,varandas-viana10,castro-varandas13} (countable shifts, non-uniformly expanding dynamics).

Thermodynamic formalism for random dynamical systems has been exclusively restricted to the case of continuous random dynamics and maps that can be coded as random shifts.
Kifer \cite{kifer92}
proved the existence and uniqueness of equilibrium states for smooth expanding dynamics;  Khanin and Kifer \cite{khanin-kifer96} generalized this to smooth random dynamics that are expanding on average.
In parallel work, Bogensch\"utz \cite{Bogenschutz1992} proved a variational principle for topological random dynamical systems (the domain $X$ is compact and each $T_\omega:X\to X$ is a homeomorphism), extending a result of Ledrappier and Walters \cite{LW77} to allow random potentials with measurable dependence $\omega\mapsto \varphi_\omega$.
For random subshifts of finite type, Bogensch\"utz and Gundlach \cite{Bogenschutz_RuelleTransferOperator_1995a} and Gundlach \cite{gundlach96} proved uniqueness of (random) relative equilibrium states.
Thermodynamic formalism for countable random Markov shifts has been investigated in various settings by Denker, Kifer, and Stadlbauer \cite{denker-etal08}, Stadlbauer \cite{stadlbauer_random_2010,stadlbauer_coupling_2017}, Roy and Urba\'nski \cite{roy_random_2011}, and Mayer and Urba\'nski \cite{mayer_countable_2015}.
Mayer, Skorulski, and Urba\'nski developed distance expanding random mappings in \cite{mayer_distance_2011}, which generalize the works of \cite{kifer92} and \cite{Bogenschutz_RuelleTransferOperator_1995a}. Simmons and Urba\'nski established a variational principle and the existence of a unique relative equilibrium state for these maps in \cite{simmons_relative_2013}. 
In the recent articles \cite{stadlbauer_quenched_2020} and \cite{stadlbauer_thermodynamic_2020}, the authors considered sequential dynamics given by composition of transfer operators for continuous surjective expanding and non-uniform expanding maps; this allowed them to construct random equilibrium states, and to establish statistical laws.


In the present paper, for the first time, we develop a thermodynamic formalism in the quenched random setting for the difficult situation of discontinuous fiber maps. Furthermore, also for the first time we consider the case of countably branched random maps which cannot be coded by a random shift.
We work in the random setting in which one has a complete probability space $(\Omega,\sF,m)$ and base dynamical system $\sigma:\Omega\to\Omega$, which is merely assumed to be invertible, $m$-preserving, and ergodic.
For each base configuration $\omega\in\Omega$, we associate a corresponding surjective piecewise monotonic map (possibly with countably many branches) $T_\omega:[0,1]\to[0,1]$. 
Our random dynamical system on the interval is given by a map cocycle $T^n_\omega:=T_{\sigma^{n-1}(\omega)}\circ\cdots\circ T_{\sigma(\omega)}\circ T_\omega$, which satisfies a random covering condition;  we have no Markovian assumptions. 
Indeed, our maps may have discontinuities and do not necessarily have full branches.
The corresponding transfer operator cocycles 
$\cL^n_\omega:=\cL_{\sigma^{n-1}(\omega)}\circ\cdots\circ \cL_{\sigma(\omega)}\circ \cL_\omega$ 
are weighted using random versions of the contracting potentials introduced in \cite{liverani_conformal_1998}.

A first step in producing a thermodynamic formalism in this random setting is the construction of a conformal measure.
In the deterministic setting Liverani, Saussol, and Vaienti \cite{liverani_conformal_1998} constructed the conformal measure of a piecewise monotonic interval map with a contracting potential $\phi$ via forward iteration of the transfer operator.
This forward iteration is convenient for our random dynamics, and many of the key contributions of our work are the extension of the methodology of \cite{liverani_conformal_1998} to general random piecewise expanding dynamics.
In Theorem~\ref{main thm: summary existence of measures and density} we obtain existence of a fiber-wise collection of conformal measures $(\nu_\omega)_{\omega\in\Omega}$ satisfying $\mathcal{L}_\omega^*\nu_{\sigma(\omega)}=\lambda_{\omega}\nu_{\omega}$ from following the approach of  \cite{liverani_conformal_1998} to extend the functional we obtain via fixed point arguments, similar to those of \cite{hofbauer_equilibrium_1982}  and \cite[Chapter 3]{mayer_distance_2011}, to a random non-atomic Borel probability measure. 
Existence and uniqueness of a family of invariant functions $(q_\omega)_{\omega\in\Omega}$ satisfying $\mathcal{L}_\omega q_{\omega}=\lambda_{\omega}q_{\sigma(\omega)}$, the corresponding invariant measures $\mu_\omega=q_\omega\nu_\omega$, and the uniqueness of the $\nu_\omega$, will follow from random cone contraction arguments, which in turn rely on random Lasota-Yorke inequalities.
Because our fiber maps $T_\omega$ are nonsmooth and may have discontinuities, care is needed when developing the random Lasota-Yorke inequalities;  we follow the approach of Buzzi \cite{buzzi_exponential_1999}, whereby we construct a large measure set of fibers for which we obtain cone contraction in a fixed, uniformly bounded time and then show that the ``bad'' fibers do not occur often enough to distort the obtained cone contraction too much.
We note that if all the maps preserve a common cone, one could attempt to use the recent approach of \cite{Horan2019}.



From the random cone contraction arguments we immediately obtain, in Theorems~\ref{main thm: summary quasicompactness} and \ref{main thm: decay of corr}, an exponential convergence of the sequence $(\lm_{\sg^{-n}(\om)}^n)^{-1}\cL_{\sg^{-n}(\om)}^n\ind$ to the invariant density $q_\om$ as well as a fiberwise exponential decay of correlations with respect to the invariant measure $\mu_\om$ and the ergodicity of the global random measure $\mu$.
Following \cite{simmons_relative_2013}, we define the expected topological pressure of the random contracting potential $\phi=(\phi_\om)_{\om\in\Om}$ by $\cE P(\phi)=\int_\Om \log \lm_\om \, dm(\om)$, see also \cite{atnip_critically_2020,mayer_distance_2011,urbanski_random_2018}. We then prove, in Theorem~\ref{main thm: eq states}, a variational principle for the expected pressure, a new result for infinitely branched random maps, and then show that the $T$-invariant random measure $\mu$ is in fact the unique relative equilibrium state for the random potential $\phi$, where $T:\Om\times I\to \Om\times I$ is the induced skew product map.





Our main results are summarized by the following simplified theorem. For a complete statement of our results see Section~\ref{sec:main results}.
\begin{theorem*}
		Given a random weighted covering system $(\Om,\sg,m,I,(T_\om)_{\om\in\Om}, (\phi_\om)_{\om\in\Om})$ (see Definition~\ref{def: random weighted covering system}), the following hold. 
		\begin{enumerate}
			\item There exists a unique random probability measure $\nu=(\nu_\om)_{\om\in\Om}$ such that for each function $f$ of bounded variation 
			\begin{align*}
			\nu_{\sg(\om)}(\cL_\om f_\om)=\lm_\om\nu_\om(f_\om),
			\end{align*}
			where 
			\begin{align*}
			\lm_\om:=\nu_{\sg(\om)}(\cL_\om\ind_\om).
			\end{align*} 	
			\item There exists a unique, modulo $\nu$, strictly positive random function $q=(q_\om)_{\om\in\Om}$ of bounded variation such that $\nu_\om(q_\om)=1$ and 
			\begin{align*}
			\cL_\om q_\om=\lm_\om q_{\sg(\om)}
			\end{align*}
			for $m$-a.e. $\om\in\Om$.
			\item The random measure $\mu$, given by $\mu_\om=q_\om\nu_\om$, is a $T$-invariant and ergodic random probability measure.
			Furthermore, the random measure $\mu=(\mu_\om)_{\om\in\Om}$ exhibits a quenched (forward and backward) exponential decay of correlations and is the unique relative equilibrium state for the random contracting potential $\phi=(\phi_\om)_{\om\in\Om}$. 
		\end{enumerate}
\end{theorem*}
Our results can treat new classes of random dynamical systems, including random $\bt$-transformations (with no upper bound on the number of branches), randomly translated $\bt$-transformations, infinitely branched random Gauss-Renyi maps, random non-uniformly expanding maps which contain an indifferent fixed point or contracting branches composed with expanding maps, and a very broad class of random Lasota-Yorke maps.

An outline of the paper is as follows.
Formal definitions, assumptions, and other properties of our random systems are described in Section \ref{sec:setup}, where we also state our main results.
In Section~\ref{sec:conformal} we prove the existence of a non-atomic random conformal measure.
Section \ref{sec:LY} develops random Lasota-Yorke inequalities for our random contracting potentials in terms of variation and the random conformal measure and in Section \ref{sec:cones} we construct the corresponding random cones.
Section~\ref{sec:good} proves that for a large ``Good'' subset of $\Omega_G\subset\Omega$, a constant number of iterates of the transfer operator beginning at $\omega\in\Omega_G$ yields a cone contraction.
The remaining ``Bad'' $\omega$-fibers are dealt with in Section \ref{sec:bad}.
Using Hilbert metric contraction arguments, Section~\ref{sec:invariant} brings together the estimates from Sections~\ref{sec:LY}, \ref{sec:cones}, \ref{sec:good}, and \ref{sec:bad} to prove our main technical lemma (Lemma~\ref{lem: exp conv in C+ cone}), which is then used to prove
(i) the existence of a random invariant density and (ii) that convergence to this invariant density occurs at an exponential rate along a subsequence.
Section~\ref{sec:randommeasures} verifies that the random families of conformal measures, invariant densities, and invariant measures are indeed measurable with respect to $\omega$. We also prove the uniqueness of the invariant density found in Section~\ref{sec:invariant}. 
In Section~\ref{sec:exp pres} we utilize the measurability produced in Section~\ref{sec:randommeasures} to introduce the expected pressure and then give more refined versions of the exponential convergence results of Section~\ref{sec:invariant} as well as prove the uniqueness of the random conformal and invariant measures.
In Section~\ref{sec:dec cor} we use the results of Section~\ref{sec:exp pres} to prove an exponential decay of correlations for the random invariant measures from Section~\ref{sec:randommeasures} in both the forward and backward directions. In Section~\ref{sec:equil} we prove a variational principle and show that the unique invariant measure established in Section~\ref{sec:randommeasures} is in fact a relative equilibrium state. 
Finally, in Section~\ref{sec:examples}, we apply our results to new classes of random dynamical systems with detailed calculations.

\section{Setup}\label{sec:setup}
We begin with a complete probability space $(\Om,\sF,m)$ and suppose that $\sg:\Om\to\Om$ is an ergodic, invertible map, which preserves the measure $m$, i.e. 
\begin{align*}
m\circ\sg^{-1}=m.
\end{align*}
Let $I$ denote a compact interval in $\RR$,
\footnote{Our theory works equally well if $I$ is taken to be an uncountable, totally ordered, order-complete set as in \cite{liverani_conformal_1998}. In this setting the order topology makes $I$ compact. We choose to work under the assumption that $I$ is a compact interval in $\RR$ for ease of exposition.}
and for each $\om\in \Om$ we consider a map 
\begin{align*}
	T_\om:I \lra I
\end{align*}
such that there exists a countable partition $\cZ_\om^*$ of $I$ such that each $Z\in\cZ_\om^*$ satisfies the following:
\begin{enumerate}[(i)]
	\item $T_\om:I \to I$ is onto, 
	\item $T_\om(Z)$ is an interval,
	\item $T_\om\rvert_Z$ is continuous and strictly monotone. 
\end{enumerate}
Let $D_\om$ be the set of discontinuities of $T_\om$, i.e. the collection of all endpoints of intervals $Z\in\cZ_\om^*$. Denote by $\cS_\om$ the singular set of $T_\om$ which is defined by 
\begin{align}\label{eq: def of sing set}
	\cS_\om:=\union_{k\geq 0}T_\om^{-k}(D_{\sg^k(\om)})
\end{align}
Then $\cS_\om$ is clearly backward $T_\om$ invariant in the sense that 
\begin{align}\label{def: cS_om}
	T_\om^{-1}(\cS_{\sg(\om)})\sub\cS_\om.
\end{align}
Notice that for each $\om\in\Om$, $\cS_\om$ is countable. 
For each $\om\in\Om$ let 
$$
	X_\om:=I\bs \cS_\om,
$$
which naturally inherits the subspace topology of $I$. 
In the sequel, by an interval $J\sub X_\om$, we of course mean $J=\~J\cap X_\om$ for some interval $\~J\sub I$. 
By virtue of \eqref{def: cS_om}, we have that 
\begin{align*}
	T_\om(X_\om)\sub X_{\sg(\om)}. 
\end{align*}
In what follows we will consider the dynamics of the maps
$$
	T_\om\rvert_{X_\om}: X_\om\lra X_{\sg(\om)}.
$$ 
We note that the maps $T_\om$ are continuous with respect to the topology on $X_\om$. 
We consider the global dynamics of the map $T:\Om\times I \to \Om \times I$ given by
$$
	T(\om,x)=(\sg(\om),T_\om(x)).
$$
Letting
$$
	\cJ:=\union_{\om\in\Om}\{\om\}\times X_\om,
$$
we may also consider the global dynamics $T\rvert_\cJ:\cJ\to\cJ$.
For each $\om\in\Om$ we let 
\begin{align*}
	\cZ_\om^{(1)}:=\set{Z\cap X_\om: Z\in\cZ_\om^*}.
\end{align*}
Now, for $ n>1$ let 
\begin{align*}
	\cZ_\om^{(n)}:=\bigvee_{j=0}^{n-1}T_\om^{-j}\lt(\cZ_{\sg^j(\om)}^{(1)}\rt).
\end{align*}
Then $\cZ_\om^{(n)}$ consists of countably many non-degenerate subintervals which are open with respect to $X_\om$. Now, given $Z\in\cZ_\om^{(n)}$, we denote by
$$	
	T_{\om,Z}^{-n}:T_\om^n(Z)\lra Z
$$ 
the inverse branch of $T_\om^n$ which takes $T_\om^n(x)$ to $x$ for each $x\in Z$.
We will assume that the partitions $\cZ_\om^*$ are generating, i.e. 
\begin{align}
	\bigvee_{n=1}^\infty \cZ_\om^{(n)}=\sB,
	\label{cond GP}\tag{GP}
\end{align}
where $\sB=\sB(I)$ denotes the Borel $\sg$-algebra of $I$. 

Let $\cB(I)$ denote the set of all bounded real-valued functions on $I$ and for each $f\in\cB(I)$ and each $A\sub I$ let
\begin{align*}
	\var_A(f):=\sup\set{\sum_{j=0}^{k-1}\absval{f(x_{j+1})-f(x_j)}: x_0<x_1<\dots x_k, \, x_j\in A \text{ for all } 0\leq j\leq k},
\end{align*}  
where the supremum is taken over all finite sets $\set{x_j}\sub A$,
denote the variation of $f$ over $A$. We let 
\begin{align*}
	\BV(I):=\set{f\in\cB(I): \var_I(f)<\infty}
\end{align*}
denote the set of functions of bounded variation on $I$. Let $\norm{f}_\infty$ and 
$$
\norm{f}_\BV:=\var(f)+\norm{f}_\infty
$$ 
be norms on the respective Banach spaces $\cB(I)$ and $\BV(I)$. 

Let $\cB(X_\om)$ be the set of all bounded real-valued functions on $X_\om$ and let $\BV(X_\om)$ be the set of all functions of bounded variation on $X_\om$, i.e. all functions $f\in\cB(X_\om)$ such that $\var_{X_\om}(f)<\infty$. 
For a function $f:\Om\times I\to\RR$ we let $f_\om:=f\rvert_{\set{\om}\times I}:I\to\RR$ for each $\om\in\Om$. In particular, we have that 
$$
	\ind_\om:=\ind_{X_\om}.	
$$
Note that if $f\in\cB(I)$, then $f_\om\in\cB(X_\om)$ and if $f\in\cB(I)$ then $f_\om\in\BV(X_\om)$ with $\var_{X_\om}(f)\leq \var_I(f)$. For this reason we will generally write $\var(f_\om)$, suppressing the set dependence, when dealing with the variation, either on the entire interval $I$ or fiber space $X_\om$, unless more care is needed.
\begin{definition}\label{def: random bounded and BV}
	We say that a function $f:\Om\times I\to\RR$ is \textit{random bounded} if 
\begin{enumerate}[(i)]
	\item $f_\om\in\cB(I)$ for each $\om\in\Om$, 
	\item for each $x\in I$ the function $\Om\ni\om\mapsto f_\om(x)$ is measurable, 
	\item the function $\Om\ni\om\mapsto\norm{f_\om}_\infty$ is measurable. 
\end{enumerate}
Let $\cB_\Om(I)$ denote the collection of all random bounded functions on $\Om\times I$, and let 
$$
	\cB_\Om^1(I)=\set{f\in\cB_\Om: \norm{f_\om}_\infty\in L^1_m(\Om)}.
$$ 
The set $\cB_\Om^1(I)$ becomes a Banach space when taken together with norm  
$$
	|f|_\infty:=
	\int_\Om \norm{f_\om}_\infty\, dm(\om).
$$
We say that a function $f\in\cB_\Om(I)$ is of \textit{random bounded variation} if $f_\om\in\BV(I)$ for each $\om\in\Om$.
We let $\BV_\Om(I)$ denote the set of all random bounded variation functions and let 
\begin{align}\label{eq: def of B_Om^l}
	\BV_\Om^l(I):=\set{f\in\BV_\Om(I): \norm{f_\om}_\infty, \log\var(f_\om), \log\inf|f_\om|\in L^1_m(\Om)}.
\end{align} 
\end{definition}
For functions $f:\Om\times I\to\RR$ and $F:\Om\to\RR$ we let 
\begin{align*}
	S_{n,T}(f_\om):=\sum_{j=0}^{n-1}f_{\sg^j(\om)}\circ T_\om^j
	\quad\text{ and }\quad
	S_{n,\sg}(F):=\sum_{j=0}^{n-1}F\circ\sg^j
\end{align*}
denote the Birkhoff sums of $f$ and $F$ with respect to $T$ and $\sg$ respectively. 
We will consider a potential of the form 
$\phi:\Om\times I\to\RR\cup\set{-\infty}$ and for each $n\geq 1$ we consider the weight $g^{(n)}:\Om\times I\to\RR$ whose disintegrations are given by
\begin{align}\label{def: formula for g_om^n}
	g_\om^{(n)}:=\exp(S_{n,T}(\phi_\om))
\end{align} 
for each $\om\in\Om$. 
We define the (Perron-Frobenius) transfer operator, $\cL_{\phi,\om}:\cB(X_\om)\to \cB(X_{\sg(\om)})$, with respect to the potential $\phi:\Om\times I\to\RR\cup\set{-\infty}$, by 
\begin{align*}
	\cL_{\phi,\om}(f)(x)=\cL_\om (f)(x):=\sum_{y\in T_\om^{-1}(x)}g_\om^{(1)}(y)f(y); \quad 
	f\in \cB(X_\om), \
	x\in X_{\sg(\om)}.
\end{align*} 
We define the global operator $\cL:\cB_\Om(I) \to\cB_\Om(I) $ by 
\begin{align*}
	(\cL f)_\om: =\cL_{\sg^{-1}(\om)} f_{\sg^{-1}(\om)}.
\end{align*}
Inducting on $n$ gives that the iterates $\cL_\om^n:\cB(X_\om)\to \cB(X_{\sg^{n}(\om)})$ are given by 
\begin{align}\label{def: formula for L_om^n}
	\cL_{\phi,\om}^n(f)(x)
	&=(\cL_{\sg^{n-1}(\om)} \circ\dots\circ\cL_\om)(f)(x)
	\nonumber\\
	&=
	\sum_{y\in T_\om^{-n}(x)}g_\om^{(n)}(y)f(y); \quad 
	f\in \cB(X_\om), \,
	x\in X_{\sg^{n}(\om)}.
\end{align}
For each $\om\in\Om$ we let $\cB^*(X_\om)$ and $\BV^*(X_\om)$ denote the respective dual spaces of $\cB(X_\om)$ and $\BV(X_\om)$. We let $\cL_\om^*:\cB^*(X_{\sg(\om)})\to\cB^*(X_\om)$ be the dual transfer operator.

To ensure that these operators $\cL_\om$ are well defined and that the sum in \eqref{def: formula for L_om^n} does in fact converge, we will need some additional assumptions on the potential $\phi$. 
\begin{definition}\label{def: summable potential}
	We will say that a potential $\phi:\Om\times I\to\RR\cup\set{-\infty}$ is \textit{summable} if for $m$-a.e. $\om\in\Om$ we have 
	\begin{flalign} 
		& \inf g_\om^{(1)}\rvert_Z>0 \text{ for all } Z\in\cZ_\om^{(1)},
		\label{cond SP1} \tag{SP1} 	
		&\\
		& g_\om^{(1)}\in\BV(I), 
		\label{cond SP2} \tag{SP2} 
		\\
		& S_\om^{(1)}:=\sum_{Z\in\cZ_\om^*}\sup_Z g_\om^{(1)}<\infty.
		\label{cond SP3} \tag{SP3}
	\end{flalign}
\end{definition}
\begin{remark}\label{rem: SP1 - SP3 hold when finite}
	Note that if $\#\cZ_\om^{(1)}<\infty$ and $\phi_\om\in\BV(I)$ for each $\om\in\Om$ then \eqref{cond SP1}-\eqref{cond SP3} are immediate. Furthermore, as $\phi_\om\in\cB(I)$ we would also have that $\inf g_\om^{(n)}>0$ for each $\om\in\Om$ and $n\in\NN$.
\end{remark}
Note that \eqref{cond SP3} implies that the transfer operator $\cL_\om:\cB(X_\om)\to\cB(X_{\sg(\om)})$ is well defined. The following lemma shows that the operator $\cL_\om^n:\cB(X_\om)\to\cB(X_{\sg^n(\om)})$ is well defined for all $n\geq 1$.

\begin{lemma}\label{lem: LY setup}
	Given a summable potential $\phi:\Om\times I\to\RR\cup\set{-\infty}$, for all $n\in\NN$ and $m$-a.e. $\om\in\Om$ we have 
	\begin{flalign} 
	& g_\om^{(n)}\in\BV(I), 
	\label{(i) lem: LY setup} \tag{i} 
	&\\
	& S_\om^{(n)}:=\sum_{Z\in\cZ_\om^{(n)}}\sup_Z g_\om^{(n)}<\infty.
	\label{(ii) lem: LY setup} \tag{ii} 
	\end{flalign}
\end{lemma}
\begin{proof}
	We will prove each of the claims via induction. To prove \eqref{(i) lem: LY setup}, we first note that by \eqref{cond SP2} $g_\om^{(1)}\in\BV(I)$. Now suppose that $g_\om^{(j)}\in\BV(I)$ for each $\om\in\Om$ and all $1\leq j\leq n$. Now, note that we may write 
	\begin{align*}
		g_\om^{(n+1)}=g_\om^{(1)}\cdot g_{\sg(\om)}^{(n)}\circ T_\om.
	\end{align*}
	Then, 
	\begin{align*}
	\var(g_\om^{(n+1)})
	&\leq 
	\sum_{Z\in\cZ_\om^{(1)}}
	\lt(
	\var_Z(g_\om^{(n+1)})+2\sup_Zg_\om^{(n+1)}
	\rt)
	\\
	&\leq 
	\sum_{Z\in\cZ_\om^{(1)}}
	\lt(
	\var_Z(g_\om^{(1)})\sup_{T_\om(Z)}g_{\sg(\om)}^{(n)}
	+\sup_Z g_\om^{(1)}\var_{T_\om(Z)}(g_{\sg(\om)}^{(n)})
	+2\sup_Zg_\om^{(1)}\sup_{T_\om(Z)}g_{\sg(\om)}^{(n)}
	\rt)
	\\
	&\leq 
	\norm{g_{\sg(\om)}^{(n)}}_{\BV}
	\sum_{Z\in\cZ_\om^{(1)}}	
	\lt(
 	\var_Z g_\om^{(1)}+\sup_Z g_\om^{(1)}+2\sup_Z g_\om^{(1)} 
	\rt)
	\\
	&\leq \lt(3 S_\om^{(1)}+\var(g_\om^{(1)})\rt)\norm{g_{\sg(\om)}^{(n)}}_{\BV}.
	\end{align*}
	To see \eqref{(ii) lem: LY setup}, we begin by noting that \eqref{cond SP3} implies that $S_\om^{(1)}<\infty$. Now, suppose that $S_\om^{(j)}<\infty$ for each $\om\in\Om$ and all $1\leq j\leq n$. Now, note that we may also rewrite $g_\om^{(n+1)}$ as 
	\begin{align*}
	g_\om^{(n+1)}=g_\om^{(n)}\cdot g_{\sg^n(\om)}^{(1)}\circ T_\om^n.
	\end{align*}
	Then, 
	\begin{align*}
	S_\om^{(n+1)}
	&=
	\sum_{Z\in\cZ_\om^{(n+1)}}\sup_Z g_\om^{(n+1)}
	\leq\sum_{Z\in\cZ_\om^{(n+1)}}\sup_Z g_\om^{(n)}\sup_{T_\om^n(Z)}g_{\sg^n(\om)}^{(1)}
	\\
	&\leq
	\sum_{Z'\in\cZ_\om^{(n)}}\sum_{\substack{Z\in\cZ_\om^{(n+1)}\\ Z\sub Z'}}\sup_{Z'}g_\om^{(n)} \sup_{T_\om^n(Z)}g_{\sg^n(\om)}^{(1)}
	\\
	&\leq 
	\sum_{Z'\in\cZ_\om^{(n)}}\sup_{Z'}g_\om^{(n)}\cdot\sum_{Z''\in\cZ_{\sg^n(\om)}^{(1)}}\sup_{Z''}g_{\sg^n(\om)}^{(1)}
	\\
	&\leq S_\om^{(n)}\cdot S_{\sg^n(\om)}^{(1)}<\infty,
	\end{align*}
	and thus we are done. 
\end{proof}

\subsection{Preliminaries on Random Measures}
Given a measurable space $Y$, we let $\cP(Y)$ denote the collection of all Borel probability measures on $Y$.
Recall that $\sB$ denotes the Borel $\sg$-algebra of $I$, and for each $\om\in\Om$, let $\sB_\om$ denote the Borel $\sg$-algebra of $X_\om$, i.e. the $\sg$-algebra $\sB$ restricted to the set $X_\om$.
Now, let $\sF\otimes\sB$ denote the product $\sg$-algebra of $\sB$ and $\sF$ on $\Om\times I$. 

Let $\cP_m(\Om\times I)$ denote the set of all probability measures $\mu$ on $\Om\times I$ that have marginal $m$, i.e.  
$$
\cP_m(\Om\times I):=\set{\mu\in\cP(\Om\times I):\mu\circ\pi^{-1}_\Om=m},
$$
where $\pi_\Om:\Om\times X\to\Om$ is the projection onto the first coordinate. 
\begin{definition}\label{def: random prob measures}
A map $\mu:\Om\times\sB\to[0,1]$ with $\Om\times\sB\ni(\om,B)\mapsto \mu_\om(B)$ is said to be a \textit{random probability measure} on $I$ if
\begin{enumerate}
	\item for every $B\in\sB$, the map $\Om\ni\om\mapsto\mu_\om(B)\in [0,1]$ is measurable, 
	\item for $m$-a.e. $\om\in\Om$, the map $\sB\ni B\mapsto\mu_\om(B)\in [0,1]$ is a Borel probability measure.
\end{enumerate}
We let $\cP_\Om(I)$ denote the set of all random probability measures on $I$. We will frequently denote a random measure $\mu$ by $(\mu_\om)_{\om\in\Om}$.
\end{definition}

The following proposition, which summarizes results of Crauel \cite{crauel_random_2002}, shows that the collection $\cP_m(\Om\times I)$ can be canonically identified with the collection $\cP_\Om(I)$ of all random probability measures on $I$.
\begin{proposition}[\cite{crauel_random_2002}, Propositions 3.3, 3.6]\label{prop: random measure equiv}
	For each $\mu\in\cP_m(\Om\times I)$ there exists a unique random probability measure $(\mu_\om)_{\om\in\Om}\in\cP_\Om(I)$ such that 
	\begin{align*}
	\int_{\Om\times I} f(\om,x)\,d\mu(\om,x)
	=
	\int_\Om \int_I f(\om,x) \, d\mu_\om(x)\, dm(\om)
	\end{align*}
	for every bounded measurable function $f:\Om\times I\to\RR$.
	
	Conversely, if $(\mu_\om)_{\om\in\Om}\in\cP_\Om(I)$ is a random probability measure on $I$, then for every bounded measurable function $f:\Om\times I\to\RR$ the function 
	$$
	\Om\ni\om\longmapsto \int_I f(\om,x) \, d\mu_\om(x)
	$$ 
	is measurable and 
	$$
	\sF\otimes\sB\ni A\longmapsto\int_\Om \int_I \ind_A(\om,x) \, d\mu_\om(x)\, dm(\om)
	$$
	defines a probability measure in $\cP_m(\Om\times I)$.
\end{proposition}
\subsection{Conditions on the Dynamics}

We now give a definition of random covering similar to the covering properties of \cite{buzzi_exponential_1999} and \cite{liverani_conformal_1998}.
\begin{definition}\label{def: covering}
	We will say that our system has \textit{random covering} if for each $\om\in\Om$ and each $J\sub I$ there exists $M_\om=M_\om(J)\in\NN$ 
	such that for all $n\geq M_\om(J)$ there exists a constant $C_{\om,n}(J)>0$ such that 
	\begin{equation}\label{cond RC} \tag{RC}
	\inf_{X_{\sg^n(\om)}}\cL_\om^{n}\ind_{J}\geq C_{\om,n}(J).
	\end{equation} 	
	
\end{definition}
\begin{remark}\label{rmk:coveringWFinitePart}
	Note that random covering implies that for all intervals $J\sub I$ there exists $M_\om=M_\om(J)\in\NN$ such that 
	\begin{align}\label{eq: finite partition cond 1}
	T_\om^{M_\om}(J)\bus X_{\sg^{M_\om}(\om)}, 
	\end{align}
	and also that for all intervals $J\sub X_\om$ there exists $M_\om\in\NN$ such that 
	\begin{align}\label{eq: finite partition cond 2}
	T_\om^{M_\om}(J)=X_{\sg^{M_\om}(\om)}.
	\end{align}
	Furthermore, note that if the partition $\cZ_\om^*$ is finite then \eqref{eq: finite partition cond 1} and \eqref{eq: finite partition cond 2} are equivalent to \eqref{cond RC}.
\end{remark}

We now concern ourselves with special finite partitions of $I$. For each $n\in\NN$, each $\om\in\Om$, $\hat \al\geq 0$, and $\hgm\geq 1$ we let $\cP_{\om,n}(\hat\al,\hgm)$ be a finite partition of $I$ such that for each $P\in\cP_{\om,n}(\hat\al,\hgm)$ the following hold: 
\begin{flalign}
&\var_P(g_\om^{(n)})\leq \hat\al\norm{g_\om^{(n)}}_\infty,
\label{cond P1}\tag{P1}
&\\
&\sum_{\substack{Z\in\cZ_\om^{(n)}\\ Z\cap P\neq \emptyset}}\sup_{Z\cap P}g_\om^{(n)}\leq \hgm\norm{g_\om^{(n)}}_\infty.
\label{cond P2}\tag{P2}
\end{flalign}
\begin{remark}\label{rem: when partitions Z_om^n=P_om,n}
	Note that the for certain potentials we may be able to satisfy \eqref{cond P1} and \eqref{cond P2} for $\hal= 0$ and $\hgm= 1$. Indeed, if $\cZ_\om^{(n)}$ is finite and if the partition $\cZ_\om^{(n)}$ satisfies \eqref{cond P1} for $\hal\geq 0$, then we may take $\hgm=1$ and  $\cP_{\om,n}(\hal,1)=\cZ_\om^{(n)}$, as \eqref{cond P2} holds trivially for $\cZ_\om^{(n)}$.
	Furthermore, if the weight functions $g_\om$ are constant then we may take $\hal=0$.
\end{remark}

The following lemma gives conditions for such partitions to exist.  
\begin{lemma}\label{lem: P1 and P2 hold}
	If $\hal>1$ and $\hgm\geq 1$ then for each $n\in\NN$ and $\om\in\Om$ there exists a partition $\cP_{\om,n}(\hal,\hgm)$ which satisfies conditions \eqref{cond P1} and \eqref{cond P2}.
\end{lemma} 
\begin{proof}
	Following Lemma 6 of \cite{rychlik_bounded_1983}, for any $\ep>0$ we can find a finite partition $\cP$ such that 
	$$
	\var_P(g_\om^{(n)})\leq \norm{g_\om^{(n)}}_\infty+\ep
	$$ 
	for each $P\in\cP$. Thus for any $\hat\al>1$ we can find a finite partition $\cP_{\om,n}(\hat\al)$ such that 
	\begin{align}
	\var_P(g_\om^{(n)})\leq \hat\al\norm{g_\om^{(n)}}_\infty
	\label{eq: proto cond P1'}
	\end{align}
	for each $P\in\cP_{\om,n}(\hat\al)$.
	Now, following the argument in the proof of Sub-Lemma 4.1.1 of \cite{liverani_conformal_1998}, for $\bt\geq 1$ we can then refine the partition $\cP_{\om,n}(\hat\al)$ with $\cZ_\om^{(n)}$ to obtain a finite partition $\cP_{\om,n}(\hat\al,\hgm)$ such that 
	\begin{align}
	\sum_{\substack{Z\in\cZ_\om^{(n)}\\ Z\cap P\neq \emptyset}}\sup_{Z\cap P}g_\om^{(n)}\leq \hgm\norm{g_\om^{(n)}}_\infty
	\label{eq: proto cond P2'}
	\end{align}
	for each $P\in\cP_{\om,n}(\hat\al,\hgm)$ since $S_\om^{(n)}<\infty$. 	
\end{proof}

\begin{remark}
	For our general theory we are only concerned that there exist $\hal\geq 0$ and $\hgm\geq 1$ and partitions $\cP_{\om,n}(\hal,\hgm)$ 
	\footnote{In \cite{liverani_conformal_1998} they consider the case $\hal=4$ and $\hgm=2$.} 
	such that \eqref{cond P1} and \eqref{cond P2} are satisfied. However, as we will see in Section~\ref{sec:examples}, the choice of $\hal$ and $\hgm$ directly impacts the amount of work necessary to verify our hypotheses for examples. So, finding (nearly) optimal values for $\hal$ and $\hgm$ may be necessary in practice. 
\end{remark}
Now, fixing $\hal\geq 0$ and $\hgm\geq1$ such that \eqref{cond P1} and \eqref{cond P2} hold for the partitions $\cP_{\om,n}(\hal,\hgm)$, since $\cZ_\om^*$ is generating, i.e. assumption \eqref{cond GP}, for each $P\in\cP_{\om,n}(\hal,\hgm)$ there exists a least number $N_{\om,n}(P)\in\NN$ and a measurable choice 
\begin{align}\label{eq: def of J(P)}
	J(P)\in\cZ_\om^{(N_{\om,n}(P))}
\end{align}
with $J(P)\sub P$.  

We will assume the following measurability and integrability conditions on our system.
\begin{flalign}
& \text{The map }T: \Om\times I\to\Om\times I \text{ is measurable}. 
\tag{M1}\label{cond M1} 
&\\
& \text{For }f\in\BV_\Om(I)\text{ we have }\cL f\in\BV_\Om(I). 
\tag{M2}\label{cond M2}
\\
& \log S_\om^{(1)}\in L^1_m(\Om),
\tag{M3}\label{cond M3} 
\\
& \log\inf\cL_\om\ind_\om\in L^1_m(\Om),
\tag{M4}\label{cond M4}
\\
& \min_{P\in\cP_{\om,n}(\hal,\hgm)}\log\inf_{J(P)}g_\om^{M_\om(J(P))}\in L^1_m(\Om) \text{ for each } n\in\NN, 
\tag{M5}\label{cond M5} 
\\		
& \max_{P\in\cP_{\om,n}(\hal,\hgm)}\log\norm{\cL_\om^{M_\om(J(P))}\ind_\om}_\infty\in L^1_m(\Om) \text{ for each } n\in\NN.
\tag{M6}\label{cond M6}  	
\end{flalign}
\begin{remark}\label{rem: inf L^n 1 log int}
	Note that condition \eqref{cond M4} implies that $\inf\cL_\om\ind_\om>0$ for $m$-a.e. $\om\in\Om$. Furthermore,although condition \eqref{cond M4} may seem restrictive, the following examples show that it can be easily checked under mild assumptions.  
	\begin{enumerate}[(a)]
		\item  If $\cZ_\om^{*}$ is finite for each $\om\in\Om$ with $\phi_\om\in\BV(I)$ and $\log\inf g_\om\in L^1_m(\Om)$, then \eqref{cond M4} holds.
		\item \eqref{cond M4} holds if for each $\om\in\Om$ there exists a finite collection $Z_{\om,1}, \dots, Z_{\om,n_\om}\in\cZ_\om^{*}$ with $T_\om(\cup_{j=1}^{n_\om} Z_{\om,j})=I$ such that $\log\inf g_\om\rvert_{\cup_{j=1}^{n_\om}Z_{\om,j}}\in L^1_m(\Om)$. In particular, \eqref{cond M4} holds if for each $\om$ there exists some $Z_\om\in\cZ_\om^{(1)}$ which supports a full branch such that $\log\inf g_\om\rvert_{Z_\om}\in L^1_m(\Om)$.
	\end{enumerate}
\end{remark}
Given a a real-valued function $f$, we let $f^+, f^-$ denote the positive and negative parts of $f$ respectively with $f^+,f^-\geq 0$ and $f=f^+-f^-$. The following lemma now follows from the previous assumptions.
\begin{lemma}\label{lem: log integrability}
	For each $n\in\NN$ we have the following: 
	\begin{flalign} 
	& \log \inf\cL_\om^n\ind_\om\in L^1_m(\Om),
	\tag{i}\label{lem: log integrability item i}  
	&\\
	& \log S_\om^{(n)}\in L^1_m(\Om),
	\tag{ii}\label{lem: log integrability item ii}
	\\
	& \log\norm{\cL_\om^n\ind_\om}_\infty\in L^1_m(\Om),
	\tag{iii}\label{lem: log integrability item iii}
	\\ 
	& \log^+ \big\|g_\om^{(n)}\big\|_\infty\in L^1_m(\Om),
	\tag{iv}\label{lem: log integrability item iv}
	\\
	& (\sup \phi_\om)^+\in L^1_m(\Om).
	\tag{v}\label{lem: log integrability item v}
	\end{flalign}
\end{lemma}
\begin{proof}
	We begin by noting that \eqref{cond M2} implies that the functions $\om\mapsto \inf\cL_\om\ind_\om$, $\om\mapsto \sup\phi_\om$ (and consequently $\om\mapsto \big\|g_\om^{(1)}\big\|_\infty$), and $\om\mapsto \norm{\cL_\om\ind_\om}_\infty$ are all $m$-measurable. Now we note that the submultiplicativity of $\big\|g_\om^{(n)}\big\|_\infty$ gives 
	\begin{align} 
	0
	\leq
	\log^+ \norm{g_\om^{(n)}}_\infty
	\leq 
	\sum_{j=0}^{n-1}\log^+\norm{g_{\sg^j(\om)}^{(1)}}_\infty
	= 
	\sum_{j=0}^{n-1}(\sup\phi_{\sg^j(\om)})^+
	\leq 
	\sum_{j=0}^{n-1} \log S_{\sg^j(\om)}^{(1)};
	\label{eq: log integrability 1}
	\end{align}
	using \eqref{cond M3} we obtain \eqref{lem: log integrability item iv} and \eqref{lem: log integrability item v}.
	Similarly, the submultiplicativity of $S_\om^{(n)}$ implied by the proof of Lemma~\ref{lem: LY setup} item \eqref{(ii) lem: LY setup} and supermultiplicativity of $\inf\cL_\om^n\ind_\om$ yields
	\begin{align}
	\sum_{j=0}^{n-1}\log\inf\cL_{\sg^j(\om)}\ind_{\sg^j(\om)}
	\leq
	\log\inf\cL_\om^n\ind_\om
	\leq
	\log \norm{\cL_\om^n\ind_\om}_\infty
	\leq
	\log S_\om^{(n)}
	\leq 
	\sum_{j=0}^{n-1} \log S_{\sg^j(\om)}^{(1)},
	\label{eq: log integrability 2}
	\end{align}
	and thus we obtain \eqref{lem: log integrability item i}--\eqref{lem: log integrability item iii} by \eqref{cond M3} and \eqref{cond M4}.
	As both the right and left-hand sides of \eqref{eq: log integrability 1} and \eqref{eq: log integrability 2} are $m$-integrable, we are done. 
\end{proof}
\begin{remark}
	Note that since the covering times $M_\om(J(P))$ may change with $\om$, our assumption \eqref{cond M6} does not necessarily follow from Lemma~\ref{lem: log integrability}.  
\end{remark}
The next definition, which is an extension of \cite[Definition~3.4]{liverani_conformal_1998} into the random setting, will be essential in the sequel.
\begin{definition}\label{def: contracting potential}
	A summable potential $\phi:\Om\times I\to\RR\cup\set{-\infty}$ is said to be \textit{contracting} if for $m$-a.e. $\om\in\Om$ we have
	\begin{equation} 
	\lim_{n\to\infty}\frac{1}{n}\log\norm{g_\om^{(n)}}_\infty  
	<
	\lim_{n\to\infty}\frac{1}{n}\log\inf \cL_\om^n\ind_\om,
	\footnote{Note that these limits exist by the subadditive ergodic theorem due to the submultiplicativity and supermultiplicativity of the respective sequences $\{\om \mapsto \|g_\om^{(n)}\|_\infty\}_{n\in\NN}$ and $\set{\om \mapsto \inf \cL_\om^n \ind_\om}_{n\in\NN}$. Furthermore, the limits are constant $m$-a.e. and the left-hand limit may be equal to $-\infty$.}
	\label{cond CP1} \tag{CP1}
	\end{equation}
	and there exists $N\in\NN$ such that 
	\begin{align}
		-\infty< \int_\Om\log\frac{\norm{g_\om^{(N)}}_\infty}{\inf\cL_\om^N\ind_\om}\, dm(\om) <0.
		\label{cond CP2} \tag{CP2}
	\end{align} 
\end{definition}

\begin{remark}\label{rmk:checkCP}
Since the 
 sequences $\{\om \mapsto \|g_\om^{(n)}\|_\infty\}_{n\in\NN}$ and $\set{\om \mapsto \inf \cL_\om^n \ind_\om}_{n\in\NN}$  are submultiplicative and supermultiplicative, respectively, in view of Lemma~\ref{lem: log integrability} the subadditive ergodic theorem implies that 
\eqref{cond CP1} is equivalent to the assumption that there exist $N_2,N_1 \in \NN$ such that
\begin{align}\label{eq: on avg CPN1N2}
	\frac{1}{N_2}\int_\Om\log \norm{g_\om^{(N_2)}}_\infty\, dm(\om) < \frac{1}{N_1} \int_\Om\log \inf \cL_\om^{N_1} \ind_\om\, dm(\om).
\end{align}
If $\log\big\|g_\om^{(n)}\|_\infty\in L^1_m(\Om)$, then \eqref{cond CP1} is equivalent to \eqref{cond CP2} is equivalent to \eqref{eq: on avg CPN1N2}.

\end{remark}
\begin{remark}
Note that in the deterministic setting our contracting potential assumption \eqref{cond CP1} implies the contracting potential assumption of \cite{liverani_conformal_1998}, and is implied by the usual assumption that $\sup \phi<P(\phi)$, where $P(\phi)$ denotes the topological pressure, see, for example, \cite{denker1990}.
\end{remark}

\begin{definition}\label{def: random weighted covering system}
	
	We will call the tuple $(\Om,\sg,m,I,T, \phi)$ a \textit{random weighted covering system} if it satisfies our assumptions \eqref{cond GP}, \eqref{cond SP1}--\eqref{cond SP3}, \eqref{cond RC}, \eqref{cond M1}--\eqref{cond M6}, and \eqref{cond CP1} and \eqref{cond CP2}. 
\end{definition}

 \subsection{Main Results}\label{sec:main results}
 Our main results are the following.
  \begin{theorem}\label{main thm: summary existence of measures and density}
  	Given a random weighted covering system $(\Om,\sg,m,I,T,\phi)$, the following hold. 
  	\begin{enumerate}
  		\item There exists a unique random probability measure $\nu\in\cP_\Om(I)$ such that 
  		\begin{align*}
  			\nu_{\sg(\om)}(\cL_\om f)=\lm_\om\nu_\om(f), 
  		\end{align*}
  		 for each $f\in\BV(I)$, where 
  		\begin{align*}
  		\lm_\om:=\nu_{\sg(\om)}(\cL_\om\ind_\om).
  		\end{align*} 	
  		Furthermore, we have that $\log\lm_\om\in L^1_m(\Om)$.
  		\item There exists a function $q\in\BV_\Om(I)$ such that $\nu(q)=1$ and for $m$-a.e. $\om\in\Om$ we have 
  		\begin{align*}
  			\cL_\om q_\om=\lm_\om q_{\sg(\om)}.
  		\end{align*}
  		Moreover, $q$ is unique modulo $\nu$, 
		and 
		\begin{align*}
			\lim_{|n|\to \infty}\frac{1}{|n|}\log \inf q_{\sg^n(\om)}
			=
			\lim_{|n|\to \infty}\frac{1}{|n|}\log \norm{q_{\sg^n(\om)}}_\infty
			=0.
		\end{align*}
  		\item The measure $\mu:=q\nu$ is a $T$-invariant and ergodic random probability measure supported on $\cJ$.
  	\end{enumerate}
  \end{theorem}
 We also show that the operator cocycle is quasi-compact.
 \begin{theorem}\label{main thm: summary quasicompactness}
 	With the same hypotheses as Theorem~\ref{main thm: summary existence of measures and density},
 	for each $f\in\BV_\Om^l(I)$ there exists a measurable function $\Om\ni\om\mapsto B_f(\om)\in(0,\infty)$ and $\kp\in(0,1)$ such that for $m$-a.e. $\om\in\Om$ and all $n\in\NN$ we have
 	\begin{align*}
 		\norm{\lt(\lm_\om^n\rt)^{-1}\cL_{\om}^n f_\om - \nu_{\om}(f_\om)q_{\sg^n(\om)}}_\infty\leq B_f(\om)\norm{f_\om}_\BV\kp^n
 	\end{align*}
 	and 
 	\begin{align*}
 	\norm{\lt(\lm_{\sg^{-n}(\om)}^n\rt)^{-1}\cL_{\sg^{-n}(\om)}^n f_{\sg^{-n}(\om)} - \nu_{\sg^{-n}(\om)}(f_{\sg^{-n}(\om)})q_{\om}}_\infty
 	\leq B_f(\om)\norm{f_{\sg^{-n}(\om)}}_\BV\kp^n.
 	\end{align*}
 \end{theorem}
From quasi-compactness we easily deduce the exponential decay of correlations. 
\begin{theorem}\label{main thm: decay of corr}
	With the same hypotheses as Theorem~\ref{main thm: summary existence of measures and density}, for every $h\in \BV_\Om^l(I)$ and every $\vkp\in(\kp,1)$, with $\kp$ as in Theorem~\ref{main thm: summary quasicompactness}, there exists a measurable function $\Om\ni\om\mapsto C_h(\om)\in(0,\infty)$ such that 
	for $m$-a.e. $\om\in\Om$,  every $n\in\NN$, and every $f\in L^1_{\mu}(\Om\times I)$  we have 
	\begin{align*}
		\absval{
		\mu_{\om}
		\lt(\lt(f_{\sg^{n}(\om)}\circ T_{\om}^n\rt)h_{\om} \rt)
		-
		\mu_{\sg^{n}(\om)}(f_{\sg^{n}(\om)})\mu_{\om}(h_{\om})
		}
		\leq C_h(\om)
		\norm{f_{\sg^n(\om)}}_{L^1_{\mu_{\sg^n(\om)}}}\norm{h_\om}_\BV\vkp^n,
	\end{align*} 
	and 
	\begin{align*}
		\absval{
		\mu_{\sg^{-n}(\om)}
		((f_\om\circ T_{\sg^{-n}(\om)}^n)h_{\sg^{-n}(\om)} )
		-
		\mu_\om(f_\om)\mu_{\sg^{-n}(\om)}(h_{\sg^{-n}(\om)})
		}
		\leq C_h(\om)
		\norm{f_\om}_{L^1_{\mu_\om}}\norm{h_{\sg^{-n}(\om)}}_\BV\vkp^n.
	\end{align*}	
\end{theorem}
Theorem~\ref{main thm: summary quasicompactness} is proven in Section~\ref{sec:exp pres}, while Theorems~\ref{main thm: summary existence of measures and density} and \ref{main thm: decay of corr} are proven in Section~\ref{sec:dec cor}.

Let $\cP_T(\Om\times I)\sub\cP_\Om(I)$ denote the set of all $T$-invariant random probability measures on $I$. 
For $\eta\in\cP_T(\Om\times I)$ we denote the (fiberwise) conditional information of the partition $\cZ_\om^*$ given $T_\om^{-1}\sB$, with respect to $\eta_\om$, by $I_{\eta_\om}$, which is given by 
\begin{align*}
	I_{\eta_\om}=I_{\eta_\om}[\cZ_{\om}^*|T_\om^{-1}\sB]:=-\log g_{\eta,\om},
\end{align*} 
where 
\begin{align*}
	g_{\eta,\om}:=\sum_{Z\in\cZ_\om^*}\ind_Z E_{\eta_\om}\lt(\ind_Z \rvert T_\om^{-1}\sB \rt)
\end{align*}
and $E_{\eta_\om}\lt(\ind_Z \rvert T_\om^{-1}\sB \rt)$ denotes the conditional expectation
of $\ind_Z$ with respect to $\eta_\om$ given the $\sg$-algebra $T_\om^{-1}\sB$.
\begin{definition}
	Given a contracting potential $\phi:\Om\times I\to\RR\cup\set{-\infty}$ we define the expected pressure of $\phi$ to be 
	\begin{align*}
		\cE P(\phi):=\int_\Om\log\lm_\om\, dm(\om),
	\end{align*} 
	and we say that a random measure $\eta\in\cP_T(\Om\times I)$ is a relative equilibrium state if 
	\begin{align*}
		\cE P(\phi)=\int_\Om
				\lt(
				\int_{I}I_{\eta_\om}+\phi_\om \, d\eta_\om
				\rt)
				 \, dm(\om).
	\end{align*}
\end{definition}
For countable partitions, it is possible that both the entropy and the integral of the potential could be infinite, which is why we state the following variational principle in terms of the conditional information as in \cite{liverani_conformal_1998}.
\begin{theorem}\label{main thm: eq states}
	The $T$-invariant, ergodic random probability measure $\mu$ produced in Theorem~\ref{main thm: summary existence of measures and density} is the unique relative equilibrium state for $\phi$, i.e. 
	\begin{align*}
		\cE P(\phi)
		=
		\int_\Om
		\lt(
		\int_{I}I_{\mu_\om}+\phi_\om \, d\mu_\om
		\rt)
		\, dm(\om)
		\geq
		\sup_{\eta\in \cP_T(\Om\times I)}\int_\Om
		\lt(
		\int_{I}I_{\eta_\om}+\phi_\om \, d\eta_\om
		\rt)
		\, dm(\om),
	\end{align*} 
	where equality holds if and only if $\eta_\om = \mu_\om$ for $m$-a.e. $\om\in\Om$.
\end{theorem}
The proof of Theorem~\ref{main thm: eq states} is presented in Section~\ref{sec:equil}.
\begin{remark}
	For $\eta\in \cP_T(\Om\times I)$ let $H_{\eta_\om}(\cZ_\om^*)$ denote the entropy of the partition $\cZ_\om^*$, which is given by 
	\begin{align*} 
	H_{\eta_\om}(\cZ_\om^*)
	:=-\sum_{Z\in\cZ_\om^*}\eta_\om(Z)\log \eta_\om(Z), 
	\end{align*}
	and let $h_\eta(T)$ denote the fiber entropy of $T$ with respect to $\eta$. Since $\cZ_\om^*$ is assumed to be generating, if $m(H_{\eta_\om}(\cZ_\om^*))<\infty$, then we have 
	\begin{align*}
	h_\eta(T)
	=\int_\Om H_{\eta_\om}(\cZ_\om^*|T_\om^{-1}\sB) \, dm(\om)
	=\int_\Om \int_I I_{\eta_\om} \, d\eta_\om\, dm(\om).
	\end{align*}
	Thus, we note that in the case that $m(H_{\mu_\om}(\cZ_\om^*))<\infty$, Theorem~\ref{main thm: eq states} reduces to the usual variational principle, that is 
	\begin{align*}
		\cE P(\phi)
		=
		h_{\mu}(T)+
		\int_\Om\int_{I}\phi_\om \, d\mu_\om\, dm(\om)
		\geq
		\sup_{\substack{\eta\in \cP_T(\Om\times I) \\ m(H_{\eta_\om}(\cZ_\om^*))<\infty}}
		h_{\eta}(T)+
		\int_\Om\int_{I}\phi_\om \, d\eta_\om\, dm(\om).
	\end{align*}
	See \cite{bogenschutz_eqilibrium_1993} for details.
	In particular, this holds if $\cZ_\om^*$ is finite for $m$-a.e. $\om\in\Om$.
\end{remark}
In what follows, we will assume throughout that we are working with a random weighted covering system $(\Om,\sg,m,I,T, \phi)$.
\section{Existence of a Conformal Family of Measures}\label{sec:conformal}

In this section we prove the existence of a conformal family of measures. We obtain these measures, as is often the case, via fixed point methods, see for example \cite{hofbauer_equilibrium_1982} or \cite[Chapter 3]{mayer_distance_2011}, for the deterministic and random settings respectively. In the continuous random setting, one considers the transfer operator $\cL_\om$ acting on the space of continuous bounded real-valued functions, and then applies the Schauder-Tichonov Fixed Point Theorem to the map $\nu_{\sg(\om)}\mapsto\cL_\om^*\nu_{\sg(\om)}/\cL_\om^*\nu(\ind_{\sg(\om)})$, where $\cL_\om^*$ is the dual operator,  to obtain a functional $\nu_\om$ such that $\cL_\om^*\nu_{\sg(\om)}=\lm_\om\nu_{\om}$ 
where $\lm_\om=\nu_{\sg(\om)}(\cL_\om\ind_\om)$. In this setting, the functional $\nu_\om$ can be uniquely identified with a Borel probability measure via the Riesz Representation Theorem. However, to deal with discontinuities we work with BV functions rather than continuous functions, so our strategy differs from that of \cite{hofbauer_equilibrium_1982} or \cite{mayer_distance_2011} in that our functional $\Lm_\om$, which we obtain via the Schauder-Tichonov Fixed Point Theorem, must first be extended to the space of continuous functions so that we may identify $\Lm_\om$ with a measure $\nu_\om$ via Riesz's Theorem. We must then show that the measure $\nu_\om$ is equal to the functional $\Lm_\om$ when restricted to BV functions. After obtaining such a family of measures $(\nu_\om)_{\om\in\Om}$ it is natural to then ask whether this family is a random probability measure as in Definition~\ref{def: random prob measures}. As we do not currently have the tools to deal with the question of $m$-measurability, we will leave this task for later, specifically Section~\ref{sec:randommeasures}. For now, we simply prove that a conformal family of measures exists. 

The main result of this section is the following.
\begin{proposition}\label{prop: existence of conformal family}
	There exists a family $(\nu_\om)_{\om\in\Om}$ of non-atomic Borel probability measures, i.e. $\nu_\om\in\cP(I)$, with 
	\begin{align}\label{eq: conformal measure property}
		\int_I\cL_\om f\, d\nu_{\sg(\om)}=\lm_\om\int_I f\,d\nu_\om
	\end{align}
	for each $\om\in\Om$ and $f\in L^1_{\nu_\om}(I)$, where 
	\begin{align*}
		\lm_\om:=\nu_{\sg(\om)}(\cL_\om\ind_\om), 
	\end{align*}
	such that 
	\begin{enumerate}[(i)]
		\item $\nu_\om(X_\om)=1$,
		\item $\nu_\om(J)>0$	
	\end{enumerate}
	for each non-degenerate interval $J\sub I$.
\end{proposition}
\begin{proof}
	For each $\om\in\Om$ let $\cM^1_\om$ denote the set of all positive linear functionals $\Gm\in\BV^*(X_\om)$ such that $\Gm(\ind_\om)=1$. We begin by noting that $\cM^1_\om$ is a convex and weak*-closed subset of the norm unit ball in $\BV^*(X_\om)$, which is weak*-compact by the Banach-Alaoglu Theorem. Thus we have that the set   
	$$
	\cM^1=\set{\Gm=(\Gm_\om)_{\om\in\Om}: \Gm_\om\in\cM_\om^1}.
	$$
	is a compact and convex subset of a locally convex topological vector space, namely the product  $\prod_{\om\in\Om}\BV^*(X_\om)$. Define the map $\Psi=(\Psi_\om)_{\om\in\Om}:\cM^1\to\cM^1$ by 
	\begin{align*}
		\Psi_\om(\Gm_{\sg(\om)})=\frac{\cL_\om^*\Gm_{\sg(\om)}}{\cL_\om^*\Gm_{\sg(\om)}(\ind_\om)}.
	\end{align*}
	Clearly $\Psi_\om$ is weak*-continuous, and thus $\Psi$ is continuous with respect to coordinate convergence, and thus $\Psi$ is continuous with respect to the product topology. Applying the Schauder-Tichonov Theorem 
	produces a $\Lm\in\cM^1$ which is a fixed point of $\Psi$. In other words, we have that there exists $\Lm=(\Lm_\om)_{\om\in\Om}$ such that 
	\begin{align}\label{eq: equivariance of functional}
		\cL_\om^*\Lm_{\sg(\om)}(f)=\lm_\om\Lm_\om(f)
	\end{align}
	for all $f\in\BV(X_\om)$, where 
	\begin{equation*}
		\lm_\om:=\Lm_{\sg(\om)}(\cL_\om\ind_\om).
	\end{equation*}
	In particular, iteration of \eqref{eq: equivariance of functional} gives that
	\begin{align}\label{eq: equivariance of functional for all n}
			\Lm_{\sg^n(\om)}(\cL_\om^nf)=\lm_\om^n\Lm_\om(f)
	\end{align}
	for each $n\in\NN$ and all $f\in\BV(X_\om)$, where 
	\begin{align}\label{eq: lm^n = prod lm}
		\lm_\om^n:=\Lm_{\sg^n(\om)}(\cL_\om^n\ind_\om)=\prod_{j=0}^{n-1}\lm_{\sg^j(\om)}.
	\end{align}
	Note that we must have 
	\begin{align}\label{eq: lm>0}
		\lm_\om>0.
	\end{align} 
	To see this we note that the random covering hypothesis \eqref{cond RC} implies the existence of some $M_\om(I)$ such that for all $n\geq M_\om(I)$, $\inf\cL_\om^n\ind_\om>0$, which implies that $\lm_\om^n>0$, which, in light of \eqref{eq: lm^n = prod lm}, in turn implies that $\lm_\om>0$.
	
	We can then extend $\Lm_\om$ to $\BV(I)$ by 
	$$
		\Lm_\om(h):=\Lm_\om(h\rvert_{X_\om}), \qquad h\in\BV(I).
	$$	
	By the compactness of $I$, each continuous function $f\in\cC(I)$ can be approximated in $\cC(I)$ by a sequence of functions in $\BV(I)$. Thus, each random continuous function can be approximated uniformly by random $\BV$ functions, which means that for each $\om\in\Om$ we can extend $\Lm_\om$ to $\cC(I)$. 
	By the Riesz Theorem we can then identify the functional $\Lm_\om$ with a random Borel probability measure $\nu_\om$ on $I$ which agrees with $\Lm_\om$ on $\cC(I)$.	
	
	Working towards showing that $\nu_\om$ agrees with $\Lm_\om$ on $\BV(X_\om)$, we first show that $\nu_\om(J)=\Lm_\om(\ind_J)>0$ for all non-degenerate intervals $J\sub I$. 
	\begin{claim}\label{prop: nu(J) positive}\label{clm: nu(J) positive}
		For all non-degenerate intervals $J\sub I$ we have 
		\begin{align*}
			\nu_\om(J)=\Lm_\om(\ind_J)>0.
		\end{align*}
	\end{claim}
	\begin{subproof}
	We first prove that $\Lm_\om(\ind_J)>0$ for all non-degenerate intervals $J\sub I$. 
	To that end, we let $M=M_\om(J)$ and $C=C_{\om,M_\om(J)}(J)$ be the constants coming from our random covering assumption \eqref{cond RC}.
	Then we note that \eqref{eq: equivariance of functional for all n} gives 
	\begin{align*}
		\lm_\om^{M}\Lm_\om(\ind_J)=\Lm_{\sg^{M}(\om)}(\cL_\om^{M}\ind_J)\geq C>0, 
	\end{align*}
	and hence, 
	\begin{align}
	\Lm_\om(\ind_J)\geq \frac{C}{\lm_\om^M}>0.
	\end{align}
	Now suppose that $\ol J=[a,b]$ and let $\ep>0$. In light of our contracting potential assumption \eqref{cond CP1}, choose $n\in\NN$ so large that 
	\begin{align}\label{eq: choice n large sup g/inf L1 small}
		\frac{\norm{g_\om^{(n)}}_\infty}{\inf\cL_\om^n\ind_\om}<\frac{\ep}{4}.
	\end{align}
	Note that since $T_\om^n$ is injective on each element of $\cZ_\om^{(n)}$, we must have that 
	\begin{align}\label{eq: L 1_Z leq g/inf L 1}
		\cL_\om^n \ind_Z\leq \norm{g_\om^{(n)}}_\infty
	\end{align} 
	for each $Z\in\cZ_\om^{(n)}$. 
	Thus, in light of \eqref{eq: equivariance of functional for all n}, \eqref{eq: choice n large sup g/inf L1 small}, and \eqref{eq: L 1_Z leq g/inf L 1},
	for all $n\in\NN$ and all $Z\in\cZ_\om^{(n)}$ we have that 
	\begin{align}\label{eq: Lm(1_Z)leq ep/4}
		\Lm_\om(\ind_Z)=(\lm_\om^n)^{-1}\Lm_{\sg^n(\om)}(\cL_\om^n\ind_Z)\leq \frac{\norm{g_\om^{(n)}}_\infty}{\inf\cL_\om^n\ind_\om}<\frac{\ep}{4}.
	\end{align}
	We now claim that there exist open intervals $V_a,V_b\sub I$ containing $a$ and $b$, respectively, such that $\Lm_\om(V_a)+\Lm_\om(V_b)<\ep$. To see this, we first construct $V_a$ by noting that \eqref{eq: Lm(1_Z)leq ep/4} taken together with the fact that $\sum_{Z\in\cZ_\om^{(n)}}\Lm_\om(\ind_Z)=1$ allows us to find a (not necessarily finite) sub-collection $\cQ_a\sub\cZ_\om^{(n)}$ such that
		\begin{enumerate}
			\item $\sum_{Z\in\cQ_a}\Lm_\om(\ind_Z)<\frac{\ep}{2}$,
			\item there exist $\al_a,\bt_a\in X_\om$
				\footnote{If $a\in X_\om$ we may take $\al_a=a=\bt_a$.}
			with $\al_a\leq a\leq \bt_a$ such that 
			\begin{equation*}
				\al_a,\bt_a\in\intr\lt(\bigcup_{Z\in\cQ_a} Z\rt).
			\end{equation*}
		\end{enumerate}
	Finally, take $V_a=\intr\lt(\bigcup_{Z\in\cQ_a} Z\rt)$. The construction of $V_b$ proceeds in a similar manner. Let $H=J\cup V_a\cup V_b$ and let  $f\in\BV(I)\cap \cC(I)$ such that 
	\begin{align*}
		\ind_J\leq f\rvert_H\leq 1 
	\end{align*}
	and $f\rvert_{H^c}=0$. 
	\footnote{One could prove the existence of such a function in a similar manner to Urysohn's Lemma to show that for $x<y$ there exists a continuous, increasing function (hence BV) from $[x,y]$ onto $[0,1]$. }
	Thus, we have 
	\begin{align*}
		\nu_\om(\ind_J)\leq \Lm_\om(f)\leq \Lm_\om(\ind_J)+\Lm_\om(\ind_{V_a})+\Lm_\om(\ind_{V_b}) < \Lm_\om(\ind_J)+\ep.
	\end{align*}
	As $\ep>0$ was arbitrary, we must have that $\nu_\om(J)\leq \Lm_\om(\ind_J)$ for all intervals $J\sub I$. To see the opposite inequality we simply note that for any disjoint intervals $A,B\sub I$ with $A\cup B=I$ we have 
	\begin{align*}
		1=\nu_\om(A)+\nu_\om(B)\leq \Lm_\om(\ind_A)+\Lm_\om(\ind_B)=1.
	\end{align*}
	Thus, taking $A=J$ and $B=J^c$ in the above equation, we have must in fact have $\Lm_\om(\ind_J)=1-\nu_\om(J^c)=\nu_\om(J)$ for all intervals $J\sub I$. 
	\end{subproof}
	Note that we have just shown in the proof of Claim~\ref{clm: nu(J) positive} that given any $a\in X_\om$, any $\ep>0$, and sufficiently large $n\in\NN$ we can find a collection $\cQ\sub\cZ_\om^{(n)}$ such that 
	\begin{enumerate}
		\item $a\in \cup_{Z\in\cQ}Z=:V$,
		\item $0<\nu_\om(V)<\ep$.
	\end{enumerate}
	In conjunction with the fact that $\nu_\om(X_\om)=\Lm_\om(\ind_\om)=1$, we see that $\nu_\om$ has no atoms, and in particular that the singular set has measure zero, i.e. 
	$$
		\nu_\om(\cS_\om)=\nu_\om(I\bs X_\om)=0.
	$$ 
	We are now able to show that $\nu_\om$ agrees with $\Lm_\om$ on $\BV(I)$. 
	\begin{claim}\label{clm: nu agrees Lm on BV} 
		For all $f\in\BV(I)$ we have 
		\begin{align*}
			\nu_\om(f)=\Lm_\om(f).
		\end{align*}
	\end{claim}
	\begin{subproof}
		Let $\ep>0$. Since the jump discontinuities of $f$ can be larger than $\ep$ on only a finite set $D_{\om,\ep}$, we can find a piecewise continuous function $f_\ep:I\to\RR$ such that $\norm{f-f_\ep}_\infty<\ep$. Since $f_\ep$ is piecewise continuous we have that $\nu_\om(f_\ep)=\Lm_\om(f_\ep)$, and hence
		\begin{align*}
			|\nu_\om(f)-\Lm_\om(f)|
			&\leq 
			|\nu_\om(f-f_\ep)|+|\nu_\om(f_\ep)-\Lm_\om(f_\ep)|+|\Lm_\om(f-f_\ep)|
			\\
			&\leq 2\norm{f-f_\ep}_\infty<2\ep.			
		\end{align*}  
		As $\ep>0$ is arbitrary, we must have that $\nu_\om(f)=\Lm_\om(f)$.
	\end{subproof}
	An immediate consequence of Claim~\ref{clm: nu agrees Lm on BV} is that 
	\begin{align}\label{eq: equivariance prop of nu}
		\nu_{\sg(\om)}(\cL_\om f)=\Lm_{\sg(\om)}(\cL_\om f)=\lm_\om\Lm_\om(f)=\lm_\om\nu_\om(f)
	\end{align}
	for all $f\in L^1_{\nu_\om}(I)$, which finishes the proof of Proposition~\ref{prop: existence of conformal family}. 
	\end{proof}

We can immediately see, cf. \cite{denker_existence_1991}, that the conformality of the family $(\nu_\om)_{\om\in\Om}$ produced in Proposition~\ref{prop: existence of conformal family} can be characterized equivalently as: for each $n\geq 1$ and each set $A$ on which $T_\om^n\rvert_A$ is one-to-one we have 
\begin{align*}
	\nu_{\sg^n(\om)}(T_\om^n(A))=\lm_\om^n\int_A e^{-S_{n,T}(\phi_\om)} \, d\nu_\om
\end{align*}
where 
\begin{align}\label{eq: lm = integral and product}
	\lm_\om^n=\nu_{\sg^n(\om)}(\cL_\om^n\ind_\om)=\prod_{j=0}^{n-1}\lm_{\sg^j(\om)}.
\end{align}
In particular, this gives that for each $n\geq 1$ and each $Z\in\cZ_\om^{(n)}$ we have 
\begin{align*}
	\nu_{\sg^n(\om)}(T_\om^n(Z))=\lm_\om^n\int_Z e^{-S_{n,T}(\phi_\om)} \, d\nu_\om.
\end{align*}

\begin{definition}
	We define the normalized operator (for any $n\geq 1$) $\~\cL_\om^n:\cB(X_\om)\to\cB(X_{\sg^n(\om)})$ by 
	$$
		\~\cL_\om^n(f):=(\lm_\om^n)^{-1}\cL_\om^n f, \qquad f\in\cB(X_\om).
	$$
\end{definition}
In light of \eqref{eq: equivariance prop of nu}
 we immediately see that this normalized operator satisfies the following:
\begin{align}\label{eq: equivariance prop of nu norm op}
	\nu_{\sg^n(\om)}(\~\cL_\om^nf)=\nu_\om(f), 
	\qquad 
	f\in L^1_{\nu_\om}(I).
\end{align}

\section{Lasota-Yorke Inequalities}\label{sec:LY}
In this section we prove Lasota-Yorke inequalities. We first prove a random version of the Lasota-Yorke inequality appearing in  \cite{liverani_conformal_1998}. However, the coefficients produced in this inequality are highly $\om$-dependent and may grow too much with further iteration. To rectify this, we then follow \cite{buzzi_exponential_1999} to prove a more refined Lasota-Yorke inequality, Proposition~\ref{lem: LY ineq 2}, with more manageable $\om$-dependent coefficients.

The following Lasota-Yorke type inequality gives a useful bound on the variation of the image of the transfer operator, which will be crucial in the sequel, and in fact implies that 
\begin{align*}
	\cL_\om^n(\BV(X_\om))\sub \BV(X_{\sg^n(\om)})
\end{align*}
for all $n\in\NN$ and $\om\in\Om$.
\begin{lemma}\label{lem: LY ineq}
	Suppose that $\hal\geq 0$ and $\hgm\geq1$ and that for each $\om\in\Om$ and $n\in\NN$ there exists a partition $\cP_{\om,n}(\hal,\hgm)$ which satisfies \eqref{cond P1} and \eqref{cond P2}. Then for all $\om\in\Om$, all $f\in\BV(X_\om)$, and all $n\in\NN$ there exist positive, measurable constants $A_\om^{(n)}$ and $B_\om^{(n)}$ such that 
	\begin{align*}
		\var(\cL_\om^nf)\leq A_\om^{(n)}\lt(\var(f)+B_\om^{(n)}\nu_\om(|f|)\rt), 
	\end{align*}
	where
	\begin{align*}
		A_\om^{(n)}:=(\hal+2\hgm+1)\norm{g_\om^{(n)}}_\infty 
		\quad\text{ and }\quad
		B_\om^{(n)}:=\frac{(\hal+2\hgm)\norm{\cL_\om^{M_{\om,n}} \ind_\om}_\infty}{\inf_{J(P_{\om,n})}g_\om^{(M_{\om,n})}}
	\end{align*}
	where $P_{\om,n}$ is some element of the finite partition $\cP_{\om,n}(\hal,\hgm)$, $J(P_{\om,n})$ is the monotonicity partition element coming from \eqref{eq: def of J(P)}, and $M_{\om,n}=M_\om(J(P_{\om,n}))$.
\end{lemma}
\begin{proof}
	Recall that Lemma~\ref{lem: LY setup} gives that $g_\om^{(n)}\in\BV(X_\om)$ for each $n\in\NN$ and $\om\in\Om$. Thus, for $f\in\BV(X_\om)$ and $n\geq 0$ we have  
	\begin{align}
		\var(\cL_\om^n f)
		&\leq 
		\var\lt(\sum_{Z\in\cZ_\om^{(n)}}\ind_{T_\om^n(Z)}\lt((g_\om^{(n)}f)\circ T_{\om,Z}^{-n}\rt)\rt)
		\nonumber\\
		&\leq 
		\sum_{Z\in\cZ_\om^{(n)}}\var\lt(\ind_{T_\om^n(Z)}\lt((g_\om^{(n)}f)\circ T_{\om,Z}^{-n}\rt)\rt)
		\nonumber\\
		&\leq 
		\sum_{Z\in\cZ_\om^{(n)}}\var\lt(g_\om^{(n)}f\ind_Z\rt)
		\nonumber\\
		&\leq 
		\sum_{Z\in\cZ_\om^{(n)}} \lt( \var_Z(g_\om^{(n)}f) +2\absval{\sup_Z g_\om^{(n)}f} \rt)
		\nonumber\\
		&\leq 
		\var(g_\om^{(n)}f)+2\sum_{Z\in\cZ_\om^{(n)}}\absval{\sup_Z g_\om^{(n)}f}.\label{LY ineq part 1}
	\end{align} 
	Thus, using \eqref{LY ineq part 1} and \eqref{cond P1}--\eqref{cond P2}, we get 
	\begin{align}
		\var(\cL_\om^nf)
		&\leq 
		\sum_{P\in\cP_{\om,n}(\hal,\hgm)}\lt(\var_{P}(g_\om^{(n)}f)+2\sum_{\substack{Z\in\cZ_\om^{(n)}\\ Z\cap P\neq \emptyset}}\absval{\sup_{Z\cap P}g_\om^{(n)}f}\rt)
		\nonumber\\
		&\leq 
		\sum_{P\in\cP_{\om,n}(\hal,\hgm)}\lt(\norm{g_\om^{(n)}}_\infty\var_{P}(f)+\norm{f\ind_{P}}_\infty\lt(\var_{P}(g_\om^{(n)})+2\sum_{\substack{Z\in\cZ_\om^{(n)}\\ Z\cap P\neq \emptyset}}\sup_{Z\cap P}g_\om^{(n)}\rt)\rt)
		\nonumber\\
		&\leq 
		\sum_{P\in\cP_{\om,n}(\hal,\hgm)}\lt(\norm{g_\om^{(n)}}_\infty\var_{P}(f)+(\hal+2\hgm)\norm{g_\om^{(n)}}_\infty\norm{f\rvert_{P}}_\infty\rt)
		\nonumber\\
		&\leq 
		\sum_{P\in\cP_{\om,n}(\hal,\hgm)}\lt((\hal+2\hgm+1)\norm{g_\om^{(n)}}_\infty\var_{P}(f)+(\hal+2\hgm)\norm{g_\om^{(n)}}_\infty \inf\absval{ f\rvert_{P}}\rt)
		\nonumber\\
		&\leq 
		\sum_{P\in\cP_{\om,n}(\hal,\hgm)}\lt(
		(\hal+2\hgm+1)\norm{g_\om^{(n)}}_\infty\var_{P}(f)+(\hal+2\hgm)\norm{g_\om^{(n)}}_\infty 
		\frac{\nu_\om(|f\rvert_{P}|)}{\nu_\om(P)}
		\rt).
		\label{final LY calc}
	\end{align}
	Note that Claim~\ref{clm: nu(J) positive} from the proof of Proposition~\ref{prop: existence of conformal family} gives that $\nu_\om(P)>0$ for all $P\in\cP_{\om,n}(\hal,\hgm)$.
	Recall that since $\cZ_\om^{(1)}$ is generating, see \eqref{cond GP}, for each $P\in\cP_{\om,n}(\hal,\hgm)$ there exists a least number $N_{\om,n}(P)\in\NN$ and a measurable choice 
	\begin{align*}
	J(P)\in\cZ_\om^{(N_{\om,n}(P))}
	\end{align*}
	with $J(P)\sub P$. 
	In light of \eqref{cond SP1} we have 
	\begin{align*}
	\inf \cL_\om^{M_{\om}(J(P))}\ind_{P}\geq \inf_{J(P)} g_\om^{M_{\om}(J(P))}>0
	\end{align*} 	
	for each $P\in\cP_{\om,n}(\hal,\hgm)$, and thus for each $n\in\NN$ and each $P\in\cP_{\om,n}(\hal,\hgm)$ we denote
	\begin{align*}
		\hat B_\om^{(n)}(P):=
		\frac{\norm{\cL_\om^{M_{\om}(J(P))} \ind_\om}_\infty}
		{\inf_{J(P)}g_\om^{(M_{\om}(J(P)))}}
	\end{align*}
	and let $P_{\om,n}$ be the element of $\cP_{\om,n}(\hal,\hgm)$ which maximizes this quantity, i.e. 
	\begin{align*}
		\hat B_\om^{(n)}(P_{\om,n}):=\max_{P\in\cP_{\om,n}}\hat B_\om^{(n)}(P).
	\end{align*}
	Then for each $P\in\cP_{\om,n}(\hal,\hgm)$ we have 
	\begin{align}\label{eq: lower bdd nu meas finite part elem}
	\nu_\om(\ind_{P})
	&=
	\nu_{\sg^{M_{\om}(J(P))}(\om)}(\~\cL_\om^{M_{\om}(J(P))} \ind_{P})
	\geq
	(\lm_\om^{M_{\om}(J(P))})^{-1} \inf_{J(P)}g_\om^{(M_{\om}(J(P)))}
	\nonumber\\
	&\geq 
	\left(\hat B_\om^{(n)}(P)\right)^{-1}
	\geq 
	\left(\hat B_\om^{(n)}(P_{\om,n})\right)^{-1}.
	\end{align}
	Inserting \eqref{eq: lower bdd nu meas finite part elem} into \eqref{final LY calc}, we have 
	\begin{align*}
		\var(\cL_\om^n f)
		\leq 
		A_\om^{(n)}\lt(\var(f)+B_\om^{(n)}\nu_\om(|f|)\rt),
	\end{align*}
	where 	
	$$
	M_{\om,n}:=M_\om(J(P_{\om,n}))
	$$ 
	denotes the covering time of $J(P_{\om,n})$ and
	\begin{align}\label{eq: def of A and B constants}
		A_\om^{(n)}:=(\hal+2\hgm+1)\norm{g_\om^{(n)}}_\infty 
		\quad\text{ and }\quad
		B_\om^{(n)}:=\frac{(\hal+2\hgm)\norm{\cL_\om^{M_{\om,n}} \ind_\om}_\infty}{\inf_{J(P_{\om,n})}g_\om^{(M_{\om,n})}},
	\end{align}
	which finishes the proof.
\end{proof}
Define the random constants
\begin{align}\label{eq: def of Q and K}
	Q_\om^{(n)}:=\frac{A_\om^{(n)}}{\inf\cL_\om^n\ind_\om}
	\quad \text{ and }\quad
	K_\om^{(n)}:=Q_\om^{(n)}B_\om^{(n)}. 
\end{align}
In light of our contracting potential assumption \eqref{cond CP1} we see that $Q_\om^{(n)}\to 0$ exponentially quickly for each $\om\in\Om$. 
\begin{remark}\label{rem: optimal P1 and P2 constants}
	In order to achieve the exponential decay of the random constants $Q_\om^{(n)}$ it would suffice to replace the fixed numbers $\hal\geq 0$ and $\hgm\geq 1$ with increasing polynomials $\hal(n),\hgm(n)$. 	
\end{remark}
The following proposition now follows from Lemma~\ref{lem: log integrability} and assumptions \eqref{cond M5} and \eqref{cond M6}. \begin{proposition}\label{prop: log integr of Q and K}
	For each $n\in\NN$, $\log^+ Q_\om^{(n)}, \log K_\om^{(n)}\in L^1_m(\Om)$ 
\end{proposition}
Now, considering the normalized operator, we arrive at the following immediate corollary.
\begin{corollary}\label{cor: normalized LY ineq 1}
	For all $\om\in\Om$, all $f\in\BV(X_\om)$, and all $n\in\NN$ we have 
	\begin{align*}
		\var(\~\cL_\om^nf)\leq Q_\om^{(n)}\var(f)+K_\om^{(n)}\nu_\om(|f|).
	\end{align*}
\end{corollary}
\begin{definition}
	In view of our contracting potential assumption \eqref{cond CP2} and the submultiplicativity of $\big\|g_\om^{(n)}\big\|_{\infty}/\inf\cL_\om^n\ind_\om$, we let $N_*\in\NN$ be the minimum integer $n\geq 1$ such that
	\begin{align}\label{eq: def of N}
		-\infty< \int_\Om \log Q_\om^{(n)} dm(\om) <0, 
	\end{align}
	and we define the number
	\begin{align}\label{eq: def of xi}
		\xi:=-\frac{1}{N_*}\int_\Om \log Q_\om^{(N_*)} dm(\om).
	\end{align}
\end{definition}

In light of Corollary~\ref{cor: normalized LY ineq 1} we may now find an appropriate upper bound for the BV norm of the normalized transfer operator. 
\begin{lemma}\label{lem: buzzi LY1}
	
		There exists a measurable function $\om\mapsto L_{\om}\in(0,\infty)$ with $\log L_{\om}\in L^1_m(\Om)$ such that for all $f\in\BV(I)$ and each $1\leq n\leq N_*$ we have 
	\begin{align}\label{eq: BV norm bound using L}
		\norm{\~\cL_{\om}^n f}_\BV \leq L_{\om}^n\lt(\var(f)+\nu_\om\lt(|f|\rt)\rt),
	\end{align}
	where 
	\begin{align*}
		L_\om^n=L_\om L_{\sg(\om)}\cdots L_{\sg^{n-1}(\om)}\geq 55^n.
	\end{align*}
\end{lemma}
\begin{proof}	
	Corollary~\ref{cor: normalized LY ineq 1} and \eqref{eq: equivariance prop of nu} give 
	\begin{align*}
		\norm{\~\cL_\om^n f}_\BV
		&=
		\var(\~\cL_\om^n f)+\norm{\~\cL_\om^n f}_\infty
		\leq 
		2\var(\~\cL_\om^n f)+\nu_{\sg^n(\om)}(\~\cL_\om^n f)
		\\
		&\leq
		2\lt(Q_\om^{(n)}\var(f)+K_\om^{(n)}\nu_{\om}(|f|)\rt)+\nu_\om(|f|)
		\\
		&\leq
		2Q_{\om}^{(n)}\var(f)+\lt(2K_{\om}^{(n)}+1\rt)\nu_\om(|f|).
	\end{align*}
	Now, set
	\begin{equation}\label{eq: L geq 6}
		\~L_{\om}^{(n)}:=\max\set{55, 2Q_{\om}^{(n)}, 2K_{\om}^{(n)}+1}
		\footnote{$55$ is chosen because $\sfrac{1}{\log 55}<\sfrac{1}{4}$, which will be useful in Section~\ref{sec:bad}.}.
	\end{equation}
	Finally, setting 
	\begin{align}\label{eq: defn of L_om^n}
		L_\om:=\max\set{\~L_{\om}^{(j)}: 1\leq j\leq N_*} 
	\end{align}
	and 
	\begin{align*}
		L_\om^n:=\prod_{j=0}^{n-1}L_{\sg^j(\om)}
	\end{align*}
	for all $n\geq 1$ suffices. The $\log$-integrability of $L_\om^n$ follows from Proposition~\ref{prop: log integr of Q and K}.
\end{proof}
\begin{definition}
In light of Lemma~\ref{lem: buzzi LY1} we now define the number
\begin{align}\label{eq: def of rho}
	\rho:=\frac{1}{N_*}\int_\Om \log L_\om^{N_*} dm(\om). 
\end{align}
\end{definition}
\begin{remark}\label{rem: cond M5' and M6' and tilde N*}
	We would like to note that the hypotheses \eqref{cond M5} and \eqref{cond M6} may be difficult to check for examples. In the sequel we will only need that these conditions hold for $n=N_*$. Thus, we will actually prove our results under the following weaker assumptions: 
	\begin{flalign} 
	&\min_{P\in\cP_{\om,N_*}}\log\inf_{J(P)}g_\om^{M_\om(J(P))}\in L^1_m(\Om), 
	\tag{M5'}\label{cond M5'} 
	&\\		
	& \max_{P\in\cP_{\om,N_*}}\log\norm{\cL_\om^{M_\om(J(P))}\ind_\om}_\infty\in L^1_m(\Om).
	\tag{M6'}\label{cond M6'} 
	\end{flalign}
	Of course this means that we need only to consider \eqref{cond P1} and \eqref{cond P2} for $n=N_*$ as well. 
\end{remark}
The constants $K_\om^{(n)}$ from Corollary~\ref{cor: normalized LY ineq 1} grow to infinity with $n$, making them difficult to use. We therefore now follow Buzzi \cite[Lemma 2.1]{buzzi_exponential_1999} to  prove the main result of this section, a similar, but more useful Lasota-Yorke type inequality. 

\begin{proposition}\label{lem: LY ineq 2}
	For each $\ep>0$ there exists a measurable, $m$-a.e. finite function $C_\ep(\om)>0$ such that for $m$-a.e. $\om\in\Om$, each $f\in\BV(X_{\sg^{-n}(\om)})$, and all $n\in\NN$ we have 
	\begin{align*}
		\var(\~\cL_{\sg^{-n}(\om)}^nf)\leq C_\ep(\om)e^{-(\xi-\ep)n}\var(f)+C_\ep(\om)\nu_{\sg^{-n}(\om)}(|f|).
	\end{align*}
\end{proposition} 

\begin{proof}	
	We begin by noting that since $\sg$ is ergodic and writing 
	\begin{align*}
		\sum_{k=0}^{nN_*-1}\psi\circ \sg^k=\sum_{j=0}^{N_*-1}\sum_{k=0}^{n-1}\psi\circ \sg^{kN_*+j},
	\end{align*}
	where $\psi:\Om\to\RR$, there must exist an integer $0\leq r_\om<N_*$ such that 
	\begin{align}\label{eq: limit Q N*}
		\lim_{n\to \infty}\frac{1}{n}\sum_{k=0}^{n-1}\log Q_{\sg^{-r_\om-kN_*}(\om)}^{(N_*)}\leq \xi.
	\end{align}
	Now for each $n\in\NN$, we may write $n=d_\om+s_\om N_*+r_\om$ with $s_\om\geq 0$, $0\leq d_\om<N_*$.
	Note that $s_\om$ and $d_\om$ depend on $\om$ through their dependence on $r_\om$. For the remainder of the proof we will drop the dependence on $\om$ from our notation and simply refer to $r$, $s$, and $d$ for $r_\om$, $s_\om$, and $d_\om$ respectively.
	Thus, for $f\in\BV(X_{\sg^{-n}(\om)})$, we can write
	\begin{align}\label{eq: inner/outer compositions}
		\~\cL_{\sg^{-n}(\om)}^n f = \~\cL_{\sg^{-r}(\om)}^r \circ \~\cL_{\sg^{-n+d}(\om)}^{sN_*} \circ \~\cL_{\sg^{-n}(\om)}^d (f). 
	\end{align}
	We begin by looking at the outer composition in \eqref{eq: inner/outer compositions}. Consider some function $h\in\BV(X_{\sg^{-r}(\om)})$. Recursively applying Lemma~\ref{lem: buzzi LY1} and \eqref{eq: equivariance prop of nu} gives
	\begin{align*}
		\var(\~\cL_{\sg^{-r}(\om)}^r (h))
		&=
		\var(\~\cL_{\sg^{-1}(\om)}(\~\cL_{\sg^{-r}(\om)}^{r-1} h))
		\nonumber\\
		&\leq 
		L_{\sg^{-1}(\om)}\lt(\var(\~\cL_{\sg^{-r}(\om)}^{r-1} h)+ \nu_{\sg^{-1}(\om)}(\big|\~\cL_{\sg^{-r}(\om)}^{r-1} h\big|)\rt)
		\nonumber\\
		&\leq
		L_{\sg^{-1}(\om)}\lt(\var(\~\cL_{\sg^{-r}(\om)}^{r-1} h)+ \nu_{\sg^{-r}(\om)}(|h|)\rt)
		\nonumber\\
		&\leq 
		L_{\sg^{-1}(\om)}L_{\sg^{-2}(\om)}\lt(\var(\~\cL_{\sg^{-r}(\om)}^{r-2} h)+ \nu_{\sg^{-r}(\om)}(|h|)\rt) +L_{\sg^{-1}(\om)}\nu_{\sg^{-r}(\om)}(h)
		\nonumber\\
		&\leq 
		\prod_{j=1}^r L_{\sg^{-j}(\om)}\var(h) + \nu_{\sg^{-r}(\om)}(|h|)\cdot \sum_{j=1}^r \prod_{k=1}^j L_{\sg^{-k}(\om)}.
	\end{align*}
	As $r<N_*$, we can write 
	\begin{align}\label{eq: LY2 inner comp ineq}
		\var(\~\cL_{\sg^{-r}(\om)}^r (h))
		\leq 
		C^{(1)}(\om)\lt(\var(h)+\nu_{\sg^{-r}(\om)}(|h|)\rt),
	\end{align}
	where 
	\begin{align*}
		C^{(1)}(\om):= N_*\cdot \prod_{j=1}^{N_*} L_{\sg^{-j}(\om)}.
	\end{align*}
	To deal with the innermost composition in \eqref{eq: inner/outer compositions} we will use similar techniques as when dealing with the outer composition. However, we first remark that since $\log L_\om\in L^1_m(\Om)$, Birkhoff's Ergodic Theorem implies that 
	\begin{align*}
		\lim_{k\to\pm\infty} \frac{1}{|k|}\log L_{\sg^k(\om)}=0,
	\end{align*}
	and thus, we must also have that for each $\dl>0$ there exists some measurable constant $C_\dl^{(2)}(\om)\geq 1$ such that we have
	\begin{align}\label{eq: L(om) subexp bound}
		L_{\sg^k(\om)}\leq C_\dl^{(2)}(\om)e^{\dl|k|}
	\end{align}
	for each $k\in\ZZ$.
	Continuing as in \eqref{eq: LY2 inner comp ineq} together with \eqref{eq: L(om) subexp bound}, we see
	\begin{align}
	\var(\~\cL_{\sg^{-n}(\om)}^d (f))
	&\leq 
	\prod_{j=n-d+1}^n L_{\sg^{-j}(\om)}\var(f) + \sum_{j=n-d+1}^n \prod_{k=n-d+1}^j L_{\sg^{-k}(\om)}\nu_{\sg^{-n}(\om)}(|f|)
	\nonumber\\
	&\leq 
	\lt(C_\dl^{(2)}(\om)\rt)^d\exp\lt(\dl\sum_{j=n-d+1}^n j\rt) \var(f)
	\nonumber\\
	&\qquad\qquad+ 
	\sum_{j=n-d+1}^n \lt(C_\dl^{(2)}(\om)\rt)^{j-n+d} \exp\lt(\dl\sum_{k=n-d+1}^j k\rt)\nu_{\sg^{-n}(\om)}(|f|)
	\nonumber\\
	&\leq 
	d\lt(C_\dl^{(2)}(\om)\rt)^{N_*}e^{nd\dl}\lt(\var(f)+\nu_{\sg^{-n}(\om)}(|f|)\rt)
	\nonumber\\
	&\leq 
	C^{(2)}(\om)e^{n\frac{\ep}{2}}\lt(\var(f)+\nu_{\sg^{-n}(\om)}(|f|)\rt)
	\label{eq: LY2 outer comp ineq}
	\end{align}
	where $\dl$ has been chosen to be smaller than $\sfrac{\ep}{2N_*}$ and 
	\begin{align*}
		C^{(2)}(\om):=N_*\lt(C_\dl^{(2)}(\om)\rt)^{N_*}.
	\end{align*}
	To deal with the middle composition of \eqref{eq: inner/outer compositions} we let $\tau=\sg^{-n+d}(\om)$ and consider some function $h\in\BV(X_{\tau})$. Applying Corollary~\ref{cor: normalized LY ineq 1} and \eqref{eq: equivariance prop of nu} repeatedly yields
	\begin{align}
	\var(\~\cL_{\tau}^{sN_*} (h))
	&=
	\var(\~\cL_{\sg^{(s-1)N_*}(\tau)}^{N_*}(\~\cL_{\tau}^{(s-1)N_*} h))
	\nonumber\\
	&\leq 
	Q_{\sg^{(s-1)N_*}(\tau)}^{(N_*)} \var(\~\cL_{\tau}^{(s-1)N_*} h)+ K_{\sg^{(s-1)N_*}(\tau)}^{(N_*)}\nu_\tau(|h|)
	\nonumber\\
	&\leq  
	\prod_{j=1}^s Q_{\sg^{(s-j)N_*}(\tau)}^{(N_*)}\var(h) + \sum_{j=1}^s \prod_{k=1}^j \lt(Q_{\sg^{(s-k)N_*}(\tau)}^{(N_*)}K_{\sg^{(s-j)N_*}(\tau)}^{(N_*)}\rt)
	\nu_{\tau}(|h|).
	\label{eq: LY2 middle comp ineq part 1}
	\end{align}
	In light of the definitions of $N_*$ \eqref{eq: def of N} and $\xi$ \eqref{eq: def of xi} and using \eqref{eq: limit Q N*}, we can find a measurable constant $C^{(3)}(\om)>0$ such that for any $j\geq 1$ we have 
	\begin{align}\label{eq: prod Q^N exp bound}
		\prod_{k=1}^j Q_{\sg^{(s-k)N_*}(\tau)}^{(N_*)} \leq C^{(3)}(\om)e^{-(\xi-\frac{\ep}{2})jN_*}.
	\end{align}
	Furthermore, since $\log K_\om^{(N_*)}\in L^1_m(\Om)$ we can find a measurable constant $C^{(4)}(\om)\geq C^{(3)}(\om)$ such that for any $j\geq 1$ we have 
	\begin{align}\label{eq: prod K^N exp bound}
		K_{\sg^{(s-j)N_*}(\tau)}^{(N_*)} \leq C^{(4)}(\om)e^{\frac{\ep}{2}jN_*}.
	\end{align}
	Inserting \eqref{eq: prod Q^N exp bound} and \eqref{eq: prod K^N exp bound} into \eqref{eq: LY2 middle comp ineq part 1} gives 
	\begin{align}
		\var(\~\cL_{\tau}^{sN_*} (h))
		&\leq  
		\prod_{j=1}^s Q_{\sg^{(s-j)N_*}(\tau)}^{(N_*)}\var(h) + \sum_{j=1}^s \prod_{k=1}^j Q_{\sg^{(s-k)N_*}(\tau)}^{(N_*)}K_{\sg^{(s-j)N_*}(\tau)}^{(N_*)}\nu_{\tau}(|h|)
		\nonumber\\
		&\leq 
		C^{(4)}(\om)e^{-(\xi-\frac{\ep}{2})sN_*} \var(h) 
		\nonumber\\
		&\qquad\qquad
		+ 
		\sum_{j=1}^s C^{(4)}(\om)e^{-(\xi-\frac{\ep}{2})(s-j)N_*}\cdot C^{(4)}(\om)e^{\frac{\ep}{2}(s-j)N_*}\nu_{\tau}(|h|)
		\nonumber\\
		&\leq 
		C^{(4)}(\om)e^{-(\xi-\frac{\ep}{2})sN_*} \var(h) 
		+
		\frac{\lt(C^{(4)}(\om)\rt)^2}{1-e^{-(\xi-\ep)N_*}}\nu_{\tau}(|h|)
		\nonumber\\
		&\leq 
		C^{(5)}(\om)e^{-(\xi-\frac{\ep}{2})sN_*} \var(h) 
		+
		C^{(5)}(\om)\nu_{\tau}(|h|),
		\label{eq: LY2 middle comp ineq part 2}
	\end{align}
	where the measurable constant $C^{(5)}(\om)$ is given by  
	\begin{align*}
		C^{(5)}(\om):=\max\set{C^{(4)}(\om), \frac{\lt(C^{(4)}(\om)\rt)^2}{1-e^{-(\xi-\ep)N_*}}}.
	\end{align*}
	For the remainder of the proof we will ignore the $\om$ dependence of the constants $C^{(j)}(\om)$ and simply write $C^{(j)}$ instead for $j=1,\dots, 5$.
	Note that since $sN_*=n-d-r$ and since $0\leq d,r\leq N_*$, we can write 
	\begin{align}\label{eq: LY2 comb note}
		-\lt(\xi-\frac{\ep}{2}\rt)sN_*+\frac{n\ep}{2}=-(\xi-\ep)n +\lt(\xi-\frac{\ep}{2}\rt)(d+r) 
		\leq 
		-(\xi-\ep)n+2(\xi-\frac{\ep}{2})N_*	
	\end{align}
	Finally, assembling our estimates \eqref{eq: LY2 inner comp ineq}, \eqref{eq: LY2 outer comp ineq}, and \eqref{eq: LY2 middle comp ineq part 2}, in view of \eqref{eq: LY2 comb note} we get 
	\begin{align*}
		&\var(\~\cL_{\sg^{-n}(\om)}^n f)
		= 
		\var(\~\cL_{\sg^{-r}(\om)}^r \circ \~\cL_{\sg^{-n+d}(\om)}^{sN_*} \circ \~\cL_{\sg^{-n}(\om)}^d (f))
		\\
		&\quad
		\leq 
		C^{(1)}\lt(\var(\~\cL_{\sg^{-n+d}(\om)}^{sN_*} \circ \~\cL_{\sg^{-n}(\om)}^d (f))+\nu_{\sg^{-n}(\om)}(|f|)\rt)
		\\
		&\quad
		\leq  
		C^{(1)}\lt(C^{(5)}e^{-(\xi-\frac{\ep}{2})sN_*}\var(\~\cL_{\sg^{-n}(\om)}^d (f))+(C^{(5)}+1)\nu_{\sg^{-n}(\om)}(|f|)\rt)
		\\
		&\quad
		\leq  
		C^{(1)}\lt(C^{(5)}e^{-(\xi-\frac{\ep}{2})sN_*}C^{(2)}e^{\frac{n\ep}{2}}
		\lt(\var(f)+\nu_{\sg^{-n}(\om)}(|f|)\rt)
		+(C^{(5)}+1)\nu_{\sg^{-n}(\om)}(|f|)\rt)
		\\
		&\quad
		\leq
		C^{(1)}C^{(5)}C^{(2)}e^{2(\xi-\ep)N_*}e^{-(\xi-\ep)n}
		\lt(\var(f)+(C^{(5)}+1)\nu_{\sg^{-n}(\om)}(|f|)\rt)
		\\
		&\quad
		\leq 
		C_\ep(\om)e^{-(\xi-\ep)n}\var(f)+C_\ep(\om)\nu_{\sg^{-n}(\om)}(|f|), 
	\end{align*}
	where the measurable constant $C_\ep(\om)>0$ is taken so large that 
	\begin{align*}
		C_\ep(\om)\geq (C^{(1)}C^{(5)}C^{(2)})(C^{(5)}+1)e^{2(\xi-\ep)N_*}, 
	\end{align*}
	which completes the proof.
	
\end{proof}

\section{Random Birkhoff Cones and Hilbert Metrics}\label{sec:cones}
In this section we first recall the theory of convex cones first used by Birkhoff in \cite{birkhoff_lattice_1940}, and then present the random cones on which our operator $\~\cL_\om$ will act as a contraction. The use of convex cones in dynamical systems was first investigated in \cite{mayer_approach_1984,ferrero_produits_1988,liverani_decay_1995}.
We begin with a definition. 
\begin{definition}
	Given a vector space $\cV$, we call a subset $\cC\sub\cV$ a \textit{convex cone} if $\cC$ satisfies the following:
	\begin{enumerate}
		\item $\cC\cap-\cC=\emptyset$,
		\item for all $\al>0$, $\al\cC=\cC$, 
		\item $\cC$ is convex,
		\item for all $f, h\in\cC$ and all $\al_n\in\RR$ with $\al_n\to\al$ as $n\to\infty$, if $h-\al_nf\in\cC$ for each $n\in\NN$, then $h-\al f\in\cC\cup\set{0}$.
	\end{enumerate}
\end{definition}
\begin{lemma}[Lemma 2.1 \cite{liverani_conformal_1998}]
	The relation $\leq $ defined on $\cV$ by 
	$$
		f\leq h \text{ if and only if } h-f\in\cC\cup\set{0}
	$$
	is a partial order satisfying the following:
	\begin{flalign*}
		& f\leq 0\leq f  \implies f=0,
		\tag{i}
		&\\
		& \lm>0 \text{ and } f\geq 0 \iff \lm f\geq 0,
		\tag{ii} \\
		& f\leq h \iff 0\leq h-f,
		\tag{iii}\\
		& \text{for all } \alpha_n\in\RR \text{ with } \al_n\to\al, \, \al_nf\leq h \implies \al f\leq h,
		\tag{iv} \\
		& f\geq 0 \text{ and } h\geq 0 \implies f+h\geq 0.
		\tag{v}
	\end{flalign*}
\end{lemma}
The Hilbert metric on $\cC$ is given by the following definition.
\begin{definition}
	Define a distance $\Ta(f,h)$ by 
	\begin{align*}
		\Ta(f,h):=\log\frac{\bt(f,h)}{\al(f,h)},
	\end{align*}
	where 
	\begin{align*}
		\al(f,h):=\sup\set{a>0: af\leq h} 
		\quad\text{ and }\quad
		\bt(f,h):=\inf\set{b>0: bf\geq h}.
	\end{align*}
\end{definition}
Note that $\Ta$ is a pseudo-metric as two elements in the cone may be at an infinite distance from each other. Furthermore, $\Ta$ is a projective metric because any two proportional elements must be zero distance from each other. The next theorem, which is due to Birkhoff \cite{birkhoff_lattice_1940}, shows that every positive linear operator that preserves the cone is a contraction provided that the diameter of the image is finite. 
\begin{theorem}[\cite{birkhoff_lattice_1940}]\label{thm: cone distance contraction}
	Let $\cV_1$ and $\cV_2$ be vector spaces with convex cones $\cC_1\sub\cV_1$ and $\cC_2\sub\cV_2$ and a positive linear operator $\cL:\cV_1\to\cV_2$ such that $\cL(\cC_1)\sub \cC_2$. If $\Ta_i$ denotes the Hilbert metric on the cone $\cC_i$ and if 
	$$
		\Dl=\sup_{f,h\in\cC_1}\Ta_2(\cL f,\cL h),
	$$
	then 
	$$
		\Ta_2(\cL f,\cL h)\leq \tanh\lt(\frac{\Dl}{4}\rt)\Ta_1(f,h).
	$$
	for all $f,h\in\cC_1$.
	\end{theorem}
Note that it is not clear whether $(\cC,\Ta)$ is complete.
The following lemma of \cite{liverani_conformal_1998} addresses this problem by linking the metric $\Ta$ with a suitable norm $\norm{\spot}$ on $\cV$.
\begin{lemma}[\cite{liverani_conformal_1998}, Lemma 2.2]\label{lem: birkhoff cone contraction}
	Let $\norm{\spot}$ be a norm on $\cV$ such that for all $f,h\in\cV$ if $-f\leq h\leq f$, then $\norm{h}\leq \norm{f}$, and let  $\vrho:\cC\to (0,\infty)$ be a homogeneous and order-preserving function, which means that for all $f,h\in\cC$ with $f\leq h$ and all $\lm>0$ we have 
	$$
		\vrho(\lm f)=\lm\vrho(f) \qquad\text{ and }\qquad \vrho(f)\leq \vrho(h).
	$$
	Then, for all $f,h\in\cC$ $\vrho(f)=\vrho(h)>0$ implies that 
	$$
		\norm{f-h}\leq \lt(e^{\Ta(f,h)}-1\rt)\min\set{\norm{f},\norm{h} }.
	$$
\end{lemma}
\begin{remark}
	Note that the choice $\vrho(\spot)=\norm{\spot}$ satisfies the hypothesis, however in the sequel we shall be interested in the choice of $\vrho=\nu_\om$ in Section~\ref{sec:invariant}.
\end{remark}
\begin{definition}\label{def: + and a cones}
	
For each $a>0$ and $\om\in\Om$ let 
\begin{align*}
\sC_{\om,a}:=\set{f\in\BV(X_\om): f\geq 0,\, \var(f)\leq a\nu_\om(f)}.
\end{align*}
To see that this cone is non-empty, we note that the function $f+c\in\sC_{\om,a}$ for $f\in\BV$ and $c\geq a^{-1}\var(f)-\inf_{X_\om}f$. We also define the cone 
\begin{align*}
\sC_{\om,+}:=\set{f\in\BV(X_\om): f\geq 0}.
\end{align*}

\end{definition}
Let $\Ta_{\om,a}$ and $\Ta_{\om,+}$ denote the Hilbert metrics induced on the respective cones $\sC_{\om,a}$ and $\sC_{\om,+}$. The following lemma collects together the main properties of these metrics. 
\begin{lemma}[\cite{liverani_conformal_1998}, Lemmas 4.2, 4.3, 4.5]\label{lem: summary of cone dist prop}
	For $f, h\in\sC_{\om,+}$ the $\Ta_{\om,+}$ distance between $f,h$ is given by 
	\begin{align*}
	\Ta_{\om,+}(f,h)=\log\sup_{x,y\in X_\om}\frac{f(y)h(x)}{f(x)h(y)}
	\end{align*}
	If $f, h\in\sC_{\om,a}$, then
	\begin{align}\label{eq: Ta+ leq Ta}
	\Ta_{\om,+}(f,h)\leq \Ta_{\om,a}(f,h), 
	\end{align}
	and if $f\in\sC_{\om,\sfrac{a}{2}}$ we then have
	\begin{align*}
	\Ta_{\om,a}(\ind,f)\leq \log\frac{\norm{f}_\infty+\frac{1}{2}\nu_\om(f)}{\min\set{\inf_{X_\om} f, \frac{1}{2}\nu_\om(f)}}.
	\end{align*}
\end{lemma}
As a consequence of Lemma~\ref{lem: LY ineq} we have that 
\begin{align}\label{eq: L_om is a weak contraction on C_+}
	\~\cL_\om\lt(\sC_{\om,+}\rt)\sub \sC_{\sg(\om),+},
\end{align}
and thus $\~\cL_\om$ is a weak contraction on $\sC_{\om,+}$. Our eventual goal is to show a similar statement for the cones $\sC_{\om,a}$. We begin that journey with the following lemma, a version of which first appeared in \cite[Lemma 3.2]{liverani_decay_1995-1}. 
\begin{lemma}\label{lem: LSV lemma 4.6}
	Given $a>0$, for each $\om\in\Om$ there exists a finite partition $\cU_{\om,a}$ of $X_\om$ and a positive integer $t_{\om,a}$ such that 
	\begin{align*}
		b_\om:=\sup_{U\in\cU_{\om,a}}\norm{\frac{\cL_\om^{t_{\om,a}}\ind_U}{\cL_\om^{t_{\om,a}}\ind_\om}}_\infty
		<
		\frac{1}{2a},
	\end{align*}
	and for all $f\in\sC_{\om,a}$ there exists $U_f\in\cU_{\om,a}$ such that $\inf f\rvert_{U_f}\geq \frac{1}{2}\nu_\om(f)$.
\end{lemma}
\begin{proof}
	First, we let $t_{\om,a}\in\NN$ be given by  
	\begin{align*}
		t_{\om,a}:=\min\set{n_0\in\NN: 
		\frac{\big\|g_\om^{(n)}\big\|_\infty}{\inf\cL_\om^n\ind_\om}
		<
		\frac{1}{2a} \text{ for all } n\geq n_0}.
	\end{align*} 
	Then for each $U\in\cZ_\om^{(t_{\om,a})}$ we have 
	\begin{align*}
		\norm{\frac{\cL_\om^{t_{\om,a}}\ind_U}{\cL_\om^{t_{\om,a}}\ind_\om}}_\infty
		\leq 
		\frac{\big\|g_\om^{(t_{\om,a})}\big\|_\infty}{\inf\cL_\om^{t_{\om,a}}\ind_\om}
		<
		\frac{1}{2a}.
	\end{align*}
	If $\cZ_\om^{(t_{\om,a})}$ is finite then we may set $\cU_{\om,a}:=\cZ_\om^{(t_{\om,a})}$. If, on the other hand, $\cZ_\om^{(t_{\om,a})}$ is infinite, noting that Lemma~\ref{lem: LY setup} gives
	\begin{align*}
		\cL_\om^{t_{\om,a}}\ind_\om 
		= 
		\sum_{U\in\cZ_\om^{(t_{\om,a})}}\cL_\om^{t_{\om,a}}\ind_U
		\leq 
		\sum_{U\in\cZ_\om^{(t_{\om,a})}}\norm{\cL_\om^{t_{\om,a}}\ind_U}_\infty
		\leq 
		S_\om^{(t_{\om,a})}
		<\infty,
	\end{align*} 
	we can choose intervals $U_1,\dots,U_\ell\in \cZ_\om^{(t_{\om,a})}$ 
		such that 
		\begin{align*}
			\norm{\frac{\cL_\om^{t_{\om,a}}\ind_{X_\om\bs\cup_{j=1}^\ell U_j}}{\cL_\om^{t_{\om,a}}\ind_\om}}_\infty 
			<
			\frac{1}{2a}.
		\end{align*}
	Now the set $X_\om\bs\cup_{j=1}^\ell U_j$ must consist of at most $\ell+1$ many intervals (which are not necessarily elements of $\cZ_\om^{(t_{\om,a})}$, though they are unions of elements of $\cZ_\om^{(t_{\om,a})}$), say $U_{\ell+1},\dots, U_L$. Setting $\cU_{\om,a}=\set{U_j: 1\leq j\leq L}$ finishes the first claim.
	Now fix	$f\in\sC_{\om,a}$ and set 
	$$
	\cA=\set{U\in\cU_{\om,a}: \text{ there exists } x_U\in U \text{ such that } f(x_U)<\frac{1}{2}\nu_\om(f) }.
	$$
	To prove the second part of Lemma~\ref{lem: LSV lemma 4.6} we will show that $\cA\neq\cU_{\om,a}$. Towards contradiction, suppose that for each $U\in\cU_{\om,a}$ there exists $x_U\in U$ such that $f(x_U)<\sfrac{1}{2}\nu_\om(f)$. Now, for $n\geq t_{\om,a}$ we have 
	\begin{align*}
	\cL_\om^n(f\ind_U)
	\leq
	\cL_\om^n(\ind_U)\lt(f(x_U)+\var_U(f)\rt)
	<
	\cL_\om^n(\ind_U)\frac{\nu_\om(f)}{2}+b_\om\cL_\om^n(\ind_\om)\var_U(f).
	\end{align*}
	Summing over all $U\in\cU_{\om,a}$ gives
	\begin{align*}
	\cL_\om^n f
	<
	\cL_\om^n\ind_\om \cdot \frac{\nu_\om(f)}{2}+b_\om \cL_\om^n\ind_\om\cdot \var(f).
	\end{align*}
	Now, integrating both sides with respect to $\nu_{\sg^n(\om)}$ and then dividing by $\lm_\om^n$ gives
	\begin{align*}
	\frac{\nu_{\sg^n(\om)}(\cL_\om^n f)}{\lm_\om^n}
	<
	\frac{\nu_\om(f)}{2}+b_\om \var(f) 
	\leq 
	\lt(\frac{1}{2}+ab_\om\rt)\nu_\om(f).
	\end{align*}
	As the left hand side is equal to $\nu_\om(f)$ and noting that $ab_\om<\sfrac{1}{2}$, we have arrived at our desired contradiction, and thus, we are done. 
\end{proof}

As an immediate consequence of the previous lemma and its proof we may define the measurable function $N_{\om,a}(\cU_{\om,a}):\Om\to\NN$ by 
\begin{align}\label{eq: def of N_om P_om}
	N_{\om,a}(\cU_{\om,a}):=\max\set{M_\om(U): U\in\cU_{\om,a}}\geq t_{\om,a}.
\end{align}


\section{Cone Contraction on Good Fibers}\label{Sec: good fibers}\label{sec:good}

In this section we follow Buzzi's approach \cite{buzzi_exponential_1999}, and describe the good behavior across a large measure set of fibers. In particular, we will show that, for sufficiently many iterates $R_*$, the normalized transfer operator $\~\cL_\om^{R_*}$ uniformly contracts the cone $\sC_{\om,a}$ on ``good'' fibers $\om$. However, first we must describe what exactly makes a fiber ``good''.  

Recall that the numbers $\xi$ and $\rho$ are given by 
\begin{align*}
	\xi:=-\frac{1}{N_*}\int_\Om \log Q_\om^{(N_*)} dm(\om)
	\quad \text{ and } \quad
	\rho:=\frac{1}{N_*}\int_\Om \log L_\om^{N_*} dm(\om). 
\end{align*}
Note that \eqref{eq: L geq 6}, \eqref{eq: defn of L_om^n}, and the ergodic theorem imply that 
\begin{align}\label{eq: rho geq log 6}
	\log 55\leq \rho = \lim_{n\to\infty}\frac{1}{nN_*}\sum_{k=0}^{n-1}\log L_{\sg^{kN_*}(\om)}^{N_*}.
\end{align}
The following definition is adapted from \cite[Definition~2.4]{buzzi_exponential_1999}.
\begin{definition}
	We will say that $\omega$ is \textit{good} with respect to the numbers $\ep$, $a$, $B_*,\,R_a=q_aN_*,\,C_*\geq 1$, and  $0<\al_*\leq C_*$ if the following: 
	\begin{flalign} 
	& B_*q_ae^{-\frac{\xi}{2}R_a}\leq \frac{1}{3}.
	\tag{G1}\label{G1} 
	&\\
	& \frac{1}{R_a}\sum_{k=0}^{\sfrac{R_a}{N_*}-1}\log L_{\sg^{kN_*}(\om)}^{N_*} 
	\in[\rho-\ep,\rho+\ep], 
	\tag{G2}\label{G2}
	\\
	& R_a\geq N_{\om,a}(\cU_{\om,a}), 
	\tag{G3}\label{G3} 
	\\
	& C_*^{-1}\leq \inf\cL_\om^{R_a}\ind_\om\leq \lm_\om^{R_a}\leq \norm{\cL_\om^{R_a}\ind_\om}_\infty\leq  C_*, 
	\tag{G4}\label{G4} 
	\\
	& \inf\cL_\om^{R_a}\ind_U \geq \al_* \text{ for all } U\in\cU_{\om,a}. 
	\tag{G5}\label{G5} 
	\end{flalign}
\end{definition}
Now, we denote
\begin{align}\label{eq: def of ep_0}
\ep_0:=\min\set{1, \frac{\xi}{2}}.
\end{align}
The following lemma describes the prevalence of the good fibers as well as how to find them. 
\begin{lemma}\label{lem: constr of Om_G}
	Given $\ep<\ep_0$ and $a>0$, there exists parameters $B_*$, $R_a$, $\al_*$, $C_*$ (all of which depend on $\ep$) such that there is a set $\Om_G\sub \Om$ of good $\om$ with $m(\Om_G)\geq 1-\sfrac{\ep}{4}$. 
\end{lemma}
\begin{proof}
	We begin by letting 
	\begin{align}\label{eq: def of Om'}
	\Om'=\Om'(B_*):=\set{\om\in\Om: C_\ep(\om)\leq B_*},
	\end{align}
	where $C_\ep(\om)>0$ is the $m$-a.e. finite measurable constant coming from Proposition~\ref{lem: LY ineq 2}.
	Choose $B_*$ sufficiently large such that $m(\Om')\geq 1-\sfrac{\ep}{8}$. Then for any $\om\in\Om$ with $\sg^{qN_*}(\om)\in\Om'$ for any $q\in\NN$, Proposition~\ref{lem: LY ineq 2} gives 
	\begin{align}\label{eq: LY ineq on good fibers}
	\var(\~\cL_\om^{qN_*}f)\leq B_*e^{-(\xi-\ep)qN_*}\var(f)+B_*\nu_\om(f)
	\end{align}
	for any $f\in\BV(X_\om)$. Noting that $\ep<\xi/2$ by \eqref{eq: def of ep_0}, we set $R_0=q_0N_*$ and choose $q_0$ sufficiently large such that 
	\begin{align*}
	B_*q_0e^{-(\xi-\ep)R_0}\leq B_*q_0e^{-\frac{\xi}{2}R_0}\leq\frac{1}{3}.
	\end{align*}
	Now for $q_1\geq q_0$ and define the set 
	\begin{align*}
	\Om''=\Om''(q_1, C_*, \al_*):=\set{\om\in\Om: \eqref{G2}-\eqref{G5} \text{ hold for the value } R_1=q_1N_* }.
	\end{align*}
	Now choose 
	$q_a\geq q_1$, 
	$C_*$, and $\al_*$, with $C_*\geq \al_*$, such that $m(\Om''(q_a,C_*,\al_*))\geq 1-\sfrac{\ep}{8}$. Set $R_a:=q_aN_*$. Set $\Om_G:=\Om''\cap \sg^{-R_a}(\Om')$. Then $\Om_G$ is the set of all $\om\in\Om$ which are good with respect to the numbers $B_*$, $R_a$, $\al_*$, $C_*$ and 
	$m(\Om_G)\geq 1-\sfrac{\ep}{4}$, and so we are done.  
\end{proof}

In what follows, given a value $B_*$, we will consider cone parameters 
\begin{align}\label{eq cone param a}
	a\geq a_0:=6B_* 
\end{align}
and we set 
\begin{align}\label{eq: def of R*}
	q_*=q_{a_0} 
	\quad\text{ and }\quad 
	R_*:=R_{a_0}=q_*N_*.
\end{align}
Note that \eqref{G1} together with Proposition~\ref{lem: LY ineq 2} 
implies that, for $\ep<\ep_0$ and $\om\in\Om_G$, we have 
\begin{align}
	\var(\~\cL_\om^{R_*}f)
	&\leq 
	B_*e^{-(\xi-\ep)R_*}\var(f)+B_*\nu_\om(f)
	\nonumber\\
	&\leq 
	B_*q_*e^{-\frac{\xi}{2}R_*}\var(f)+B_*\nu_\om(f)
	\nonumber\\
	&\leq 
	\frac{1}{3}\var(f)+B_*\nu_\om(f).
	\label{G1 cons}
\end{align}
For each $\om\in\Om$, $a>0$, and some set $Y\sub \sC_{\om,a}$ we let
\begin{align*}
	\diam_{\om,a}(Y):=\sup_{x,y\in Y} \Ta_{\om,a}(x,y)
\end{align*}
and 
\begin{align*}
	\diam_{\om,+}(Y):=\sup_{x,y\in Y} \Ta_{\om,+}(x,y)
\end{align*}
denote the diameter of $Y$ in the respective cones $\sC_{\om,a}$ and $\sC_{\om,+}$ with respect to the respective metrics $\Ta_{\om,a}$ and $\Ta_{\om,+}$.
We come now to the final, and main, result of this section which shows that the normalized operator is a contraction on the fiber cones $\sC_{\om,a}$ and that the image has finite diameter.
\begin{lemma}\label{lem: cone contraction for good om}
	If $\om$ is good with respect to the numbers $\ep$, $a_0$, $B_*$, $R_*$, $\al_*$, $C_*$ then for each $a\geq a_0$ we have 
	\begin{align*}
	\~\cL_\om^{R_*}(\sC_{\om,a})\sub \sC_{\sg^{R_*}(\om), \sfrac{a}{2}}\sub \sC_{\sg^{R_*}(\om),a}
	\end{align*}
	and that 
	\begin{align*}
		\diam_{\sg^{R_*}(\om), a}\lt(\~\cL_\om^{R_*}(\sC_{\om,a})\rt)\leq \Dl_a:= 2\log \frac{C_*(3+a)}{\al_*} <\infty. 
	\end{align*}
\end{lemma}
\begin{proof}
	For $\om$ good and $f\in\sC_{\om,a}$, \eqref{G1 cons}, \eqref{eq cone param a}, and \eqref{eq: equivariance prop of nu} give
	\begin{align*}
		\var(\~\cL_\om^{R_*} f)
		&\leq 
		\frac{1}{3}\var(f) + B_*\nu_\om(f) 
		\leq 
		\lt(\frac{a}{3} + B_* \rt)\nu_\om(f) 
		\\
		&\leq 
		\lt(\frac{a}{3}+\frac{a}{6}\rt)\nu_\om(f) 
		=\frac{a}{2}\nu_{\sg^{R_*}(\om)}(\~\cL_\om^{R_*} f).
	\end{align*}
	Hence we have 
	\begin{align*}
		\~\cL_\om^{R_*}(\sC_{\om,a})\sub \sC_{\sg^{R_*}(\om), \sfrac{a}{2}}\sub \sC_{\sg^{R_*}(\om),a}
	\end{align*}
	as desired. Now, towards finding the diameter of $\~\cL_\om^{R_*}(\sC_{\om,a})$ in $\sC_{\sg^{R_*}(\om),a}$ we note that Lemma~\ref{lem: summary of cone dist prop} implies 
	\begin{align}\label{sup dist to one in future cone}
		\Ta_{\sg^{R_*}(\om),a}(\~\cL_\om^{R_*}f,\ind_{\sg^{R_*}(\om)})
		\leq 
		\log\frac{\norm{\~\cL_\om^{R_*} f}_\infty+\frac{1}{2}
			\nu_{\sg^{R_*}(\om)}(\~\cL_\om^{R_*}f)}
		{\min\set{\inf \~\cL_\om^{R_*} f, \frac{1}{2}\nu_{\sg^{R_*}(\om)}(\~\cL_\om^{R_*} f)}}.
	\end{align}
	Thus, to finish the proof we look for bounds for $\norm{\~\cL_\om^{R_*} f}_\infty$ and $\inf \~\cL_\om^{R_*} f$ in terms of $\nu_{\sg^{R_*}(\om)}(\~\cL_\om^{R_*} f)$. To that end, we note that since $\~\cL_\om^{R_*} f\in\sC_{\sg^{R_*}(\om),a/2}$ we have 
	\begin{align}\label{sup norm L_om^R bound for finite diam}
		\norm{\~\cL_\om^{R_*} f}_\infty
		\leq 
		\nu_{\sg^{R_*}(\om)}(\~\cL_\om^{R_*} f)+\var(\~\cL_\om^{R_*} f)
		\leq 
		\lt(1+\frac{a}{2}\rt)\nu_{\sg^{R_*}(\om)}(\~\cL_\om^{R_*} f).
	\end{align}
	Since $R_*\geq N_{\om,a_0}(\cU_{\om,a_0})\geq t_{\om,a_0}$ we may apply Lemma~\ref{lem: LSV lemma 4.6} to find $U_f\in\cU_{\om,a_0}$, which, taken together with \eqref{G4} -- \eqref{G5}, gives 
	\begin{align}\label{inf L_om^R bound for finite diam}
		\inf\~\cL_\om^{R_*} f 
		\geq 
		\inf\~\cL_\om^{R_*} (f\ind_{U_f}) 
		\geq 
		\frac{1}{2}\nu_\om(f)\inf\~\cL_\om^{R_*}\ind_{U_f}
		\geq 
		\frac{1}{2}\nu_\om(f)\frac{\al_*}{\lm_\om^{R_*}}
		\geq 
		\frac{\al_*}{2C_*}\nu_{\sg^{R_*}(\om)}(\~\cL_\om^{R_*}f).
	\end{align}
	Inserting \eqref{sup norm L_om^R bound for finite diam} and \eqref{inf L_om^R bound for finite diam} into \eqref{sup dist to one in future cone}, simplifying, and applying the triangle inequality finishes the proof.
\end{proof}

\section{Dealing With Bad Fibers}\label{sec:bad}
In this section we again follow Buzzi \cite{buzzi_exponential_1999} to show that for fibers in the small measure set of ``bad'' fibers, $\Om_B:=\Om\bs\Om_G$, the cone $\sC_{\om,a}$ of positive functions is invariant after sufficiently many iterations for sufficiently large parameters $a>0$. We accomplish this by introducing the concept of coating intervals, which we then show make up a relatively small portion of an orbit. 

Recall that $R_*$ is given by \eqref{eq: def of R*}, and for each $\om\in\Om$ we let 
\begin{align}\label{eq: def of y_*}
	0\leq y_*(\om)<R_*
\end{align} 
be the smallest integer such that for either choice of sign $+$ or $-$ we have 
\begin{flalign} 
& \lim_{m\to\infty} \frac{1}{m}\#\set{0\leq k< m: \sg^{\pm kR_*+y_*(\om)}(\om)\in\Om_G} >1-\ep,
\label{def j*1} 
&\\
& \lim_{m\to\infty} \frac{1}{m}\#\set{0\leq k< m: C_\ep\lt(\sg^{\pm kR_*+y_*(\om)}(\om)\rt)\leq B_*} >1-\ep.
\label{def j*2} 
\end{flalign}
This is possible, since by e.g.\ Lemma 33 \cite{gonzalez-tokman_stability_2018} ergodicity of $\sigma$ yields the existence of a $\sg^{R_*}$-invariant subset $E\subset \Omega$ and a divisor $d$ of $R_*$ such that  the ergodic components of $m$ under $\sg^{R_*}$ are the sets $\sigma^{-i} (E)$, $i=0,\ldots,d-1$, and $m(E)=1/d$.
Thus, for $m$-a.e.\ $\omega\in \Omega$, both limits (\ref{def j*1}) and (\ref{def j*2}) exist by Birkhoff.
When $\sigma^{y_*(\om)}(\om)\in \sigma^{-i}(E)$, by Birkhoff the limit (\ref{def j*1}) equals $m(\sigma^{-i}(E)\cap \Om_G)\cdot d$.
Similarly, since $\sigma^{-R_*}(\Omega')\subset \Omega_G$ (proof of Lemma \ref{lem: constr of Om_G}), and (\ref{def j*1}) is a stronger condition by the definition of $\Omega_G$, the limit (\ref{def j*2}) is not less than $m(\sigma^{-i}(E)\cap \Om_G)\cdot d$.
Since $m(\Om_G)>1-\ep/4$ and since $\Omega_G$ is fixed, we may choose $y_*(\om)$ so that the limits (\ref{def j*1}) and (\ref{def j*2}) are not less than $1-(\epsilon/4)\cdot d$ (this worst case is achieved when the mass of $\Omega_G$ is distributed equally amongst the $\sigma^{-i} (E)$, $i=0,\ldots,d-1$).

Clearly, $y_*:\Om\to\NN$ is a measurable function such that 
\begin{flalign} 
& y_*(\sg^{y_*(\om)}(\om))=0,
\label{prop j*1} 
&\\
& y_*(\sg^{R_*}(\om))=y_*(\om).
\label{prop j*2} 
\end{flalign}
In particular, \eqref{prop j*1} and \eqref{prop j*2} together imply that 
\begin{align}\label{prop j*3}
	y_*(\sg^{y_*(\om)+kR_*}(\om))=0
\end{align}
for all $k\in\NN$.
Let 
\begin{equation}\label{def Gm}
	\Gm(\om):=\prod_{k=0}^{q_*-1} L_{\sg^{kN_*}(\om)}^{N_*}
	,
\end{equation}
and for each $\om\in\Om$ we define the \textit{coating length} $\ell(\om)$ as follows:
\begin{itemize}
	\item if $\om\in\Om_G$ then set $\ell(\om):=1$, 
	\item if $\om\in\Om_B$ then 
	\begin{align}\label{def: coating length}
		\ell(\om):=\min\set{n\in\NN: \frac{1}{n}\sum_{0\leq k< n} \lt(\ind_{\Om_B}\log \Gm\rt)(\sg^{kR_*}(\om)) \leq \log \ep^{\frac{1}{2}}\rho R_*}.
	\end{align}
	If the minimum is not attained we set $\ell(\om)=\infty$.
\end{itemize}
Since $L_\om^{N_*}\geq 55^{N_*}$ by Lemma~\ref{lem: buzzi LY1}, we must have that 
\begin{align}\label{eq: Bm geq 55^R}
	\Gm(\om)\geq 55^{R_*}
\end{align}
for all $\om\in\Om$. 
It follows from Lemma~\ref{lem: buzzi LY1} 
that for $m$-a.e. $\om\in\Om$ we have 
\begin{align}\label{eq: LY ineq for bad fibers}
	\var(\~\cL_\om^{R_*}f)
	&\leq L_\om^{R_*}(\var(f)+\nu_\om(f))
	= \Gm(\om)(\var(f)+\nu_\om(f)).  
\end{align}
Furthermore, if $\om\in\Om_G$ it follows from \eqref{G2} that 
\begin{align}\label{eq: up and low bds for Gamma on good om}
	R_*(\rho-\ep)\leq \log\Gm(\om)\leq R_*(\rho+\ep).
\end{align}
Now let
\begin{align}\label{eq: def of ep_9}
\ep<\ep_1:=\min\set{\left(\frac{\rho}{2+\rho}\right)^2, \lt(\frac{\log 55}{\rho}\rt)^2 }.
\end{align}
The following proposition collects together some of the key properties of the coating length $\ell(\om)$. 
\begin{proposition}\label{prop: ell(om) props}
	Concerning the number $\ell(\om)$, we have the following.
		\begin{flalign*}
		& \text{For }m\text{-a.e. } \om\in\Om \text{ such that } y_*(\om)=0 \text{ we have } \ell(\om)<\infty,
		\tag{i}\label{prop: ell(om) props item i}
		&\\
		& \text{If }\om\in\Om_B \text{ then }\ell(\om)\geq 2.
		\tag{ii} \label{prop: ell(om) props item ii}
		\end{flalign*}
\end{proposition}
\begin{proof}
To see \eqref{prop: ell(om) props item i} we first note that 
\begin{align}\label{eq: ell(om) props item i num 1}
	\frac{1}{n}\sum_{0<k\leq n}\lt(\ind_{\Om_B}\log \Gm\rt)(\sg^{kR_*}(\om))
	=
	\frac{1}{n}\sum_{0<k\leq n}\log\Gm(\sg^{kR_*}(\om))
	-
	\frac{1}{n}\sum_{0<k\leq n }\lt(\ind_{\Om_G}\log \Gm\rt)(\sg^{kR_*}(\om)).
\end{align}
Towards estimating each of the terms in the right-hand side of \eqref{eq: ell(om) props item i num 1}, 
we recall that $R_*=q_{a_0}N_*$ and use \eqref{def Gm} 
and \eqref{eq: rho geq log 6} to see that 
\begin{align}\label{eq: ell(om) props item i num 2}
	\frac{1}{n}\sum_{0<k\leq n}\log\Gm(\sg^{kR_*}(\om))
	&=
	\frac{1}{n}\sum_{k=0}^{n-1}\sum_{j=0}^{q_{a_0}-1} \log L_{\sg^{jN_*+kR_*}(\om)}^{N_*}
	\leq 
	(\rho+\ep)R_*
\end{align}
for all $n\in\NN$ sufficiently large, while  \eqref{def j*1} (with $y_*(\om)=0$) together with \eqref{eq: up and low bds for Gamma on good om} ensure that 
\begin{align}\label{eq: ell(om) props item i num 3}
	\frac{1}{n}\sum_{0<k\leq n}\lt(\ind_{\Om_G}\log \Gm\rt)(\sg^{kR_*}(\om))
	\geq 
	(1-\ep)\lt((\rho-\ep)R_*\rt)
\end{align}
for all $n\in\NN$ sufficiently large. Now, combining \eqref{eq: ell(om) props item i num 2} and \eqref{eq: ell(om) props item i num 3} together with \eqref{eq: ell(om) props item i num 1}
and then using \eqref{eq: def of ep_9}, for $n\in\NN$ sufficiently large we have 
\begin{align}
	\frac{1}{n}\sum_{0<k\leq n }\lt(\ind_{\Om_B}\log \Gm\rt)(\sg^{kR_*}(\om))
	&\leq (\rho+\ep)R_* - (1-\ep)\lt((\rho-\ep)R_*\rt)
	\nonumber\\
	&\leq \ep(2+\rho)R_*
	\label{eq: Buzzi 3.2 equiv first ineq}
	\\
	&\leq \ep^{\frac{1}{2}}\rho R_*.
	\label{eq: Buzzi 3.2 equiv}
\end{align}
In light of the definition of $\ell(\om)$, \eqref{def: coating length}, we see that $\ell(\om)<\infty$.

Now to see \eqref{prop: ell(om) props item ii} we suppose that $\om\in\Om_B$ and derive a contradiction if $\ell(\om)=1$. Indeed, in light of \eqref{eq: Bm geq 55^R}, we see that $\ell(\om)=1$ implies that  
\begin{align*}
	R_*\log 55 \leq \log\Gm(\om)\leq \ep^{\frac{1}{2}}\rho R_*,
\end{align*} 
and thus that 
\begin{align*}
	\ep\geq \lt(\frac{\log 55}{\rho}\rt)^2,
\end{align*}
which contradicts the fact that we have chosen $\ep<\ep_1$.
\end{proof}
\begin{remark}\label{rem: y*=0 implies ell finite}
	Given $\om_0\in\Om$, for each $j\geq 0$ let $\om_{j+1}=\sg^{\ell(\om_j)R_*}(\om_j)$. 
	As a consequence of Proposition~\ref{prop: ell(om) props} item \eqref{prop: ell(om) props item i} and \eqref{prop j*2}, we see that for $m$-a.e. $\om_0\in\Om$ with $y_*(\om_0)=0$, we must have that $\ell(\om_j)<\infty$ for all $j\geq 0$.
\end{remark}
\begin{definition}\label{def: coating intervals}	
	We will call a (finite) sequence $\om, \sg(\om), \dots, \sg^{\ell(\om)R_*-1}(\om)$ of $\ell(\om)R_*$ many fibers a \textit{good block} (originating at $\om$) if $\om\in\Om_G$ (which implies that $\ell(\om)=1$). If, on the other hand, $\om\in\Om_B$ we call such a sequence a \textit{bad block} (originating at $\om$).

For each $\om\in\Om$ we define the \textit{coating intervals} along the orbit starting at $\om$ to be bad blocks of the form $\sg^{a_jR_*}(\om), \sg^{a_jR_*+1}(\om), \dots, \sg^{b_jR_*-1}(\om)$  where $a_j,b_j$ are given by
\begin{align*}
a_j:=\min\set{k\geq b_{j-1}: \sg^{kR_*}(\om)\in\Om_B}
\quad\text{ and }\quad
b_j=a_j+\ell(\sg^{a_jR_*}(\om))
\end{align*}
for $j\geq 1$ and setting $b_0:=-1$.

\end{definition}

Choosing $\ep< \ep_1$, 
we may define the number
\begin{align}\label{eq: def of gm constant}
\gm=\gm(\ep):=\frac{\ep^{\frac{1}{2}}\rho R_*}{R_*\log 55}<1.
\end{align}
Note that since $\sfrac{1}{\log 55}<\sfrac{1}{4}$ we can find $\ep_2<\ep_1$ sufficiently small such that $\gm<\sfrac{1}{4}$.
%
%
Now we set
\begin{align*}
\ep_3:=\min\set{\ep_2,   \lt(\frac{\xi}{8\rho}\rt)^2}.
\end{align*}

\begin{observation} 
	By our choice of $\ep<\ep_3$ we must have that 
	\begin{align}\label{obs: expon leq xi/2}
	-\frac{\xi}{2}> 2\ep^{\frac{1}{2}}\rho-(\xi-\ep)(1-\gm) > \ep^{\frac{1}{2}}\rho-(\xi-\ep)(1-\gm).
	\end{align}
	To see this we note that $\ep<\ep_3$ implies that 
	\begin{align*}
		2\ep^{\frac{1}{2}}\rho < \frac{\xi}{4}
		\quad\text{ and }\quad 
		\gm<\frac{1}{4},
	\end{align*}
	which in turn implies that 
	\begin{align*}
		2\ep^{\frac{1}{2}}\rho-(\xi-\ep)(1-\gm)
		&<
		\frac{\xi}{4}+\xi(\gm-1)
		< 
		\frac{\xi}{4}-\frac{3\xi}{4}
		=-\frac{\xi}{2}	
	\end{align*}
	as desired.
\end{observation}

The next lemma shows that the normalized operator is weakly contracting (i.e. non-expanding) on the fiber cones $\sC_{\om,a}$ for sufficiently large values of $a>a_0$.
\begin{lemma}\label{lem: cone cont for coating blocks}
	For every $\ep<\ep_3$ and all $\om\in\Om$ with $\ell(\om)<\infty$ we have that 
	\begin{align*}
		\~\cL_\om^{\ell(\om)R_*}(\sC_{\om,a_*})\sub \sC_{\sg^{\ell(\om)R_*}(\om), a_*}
	\end{align*}
	where 
	\begin{align}\label{eq: def of a_*}
	a_*=a_*(\ep):=a_0e^{\ep^{\frac{1}{2}}\rho R_*}=6B_*e^{\ep^{\frac{1}{2}}\rho R_*}.
	\end{align}
\end{lemma}
\begin{proof}
	Lemma~\ref{lem: cone contraction for good om} covers the case of $\om\in\Om_G$, so suppose $\om\in\Om_B$. Throughout the proof we will denote $\ell(\om)=\ell\geq 2$. Using \eqref{eq: LY ineq on good fibers} on good fibers and \eqref{eq: LY ineq for bad fibers} on bad fibers, for any $p\geq 1$ and $f\in\sC_{\om,+}$ we have 
	\begin{align}\label{eq: var coating length}
		\var(\~\cL_\om^{pR_*} f)
		&\leq
		\lt(\prod_{j=0}^{p-1} \Phi_{\sg^{jR_*}(\om)}^{(R_*)}\rt)\var(f)
		+ 
		\sum_{j=0}^{p-1}\lt(D_{\sg^{jR_*}(\om)}^{(R_*)}\cdot \prod_{k=j+1}^{p-1}\Phi_{\sg^{kR_*}(\om)}^{(R_*)}\rt)\nu_\om(f),
	\end{align}
	where 
	\begin{equation}\label{eq: def of Phi_tau^R}
	\Phi_\tau^{(R_*)}=
	\begin{cases}
	B_* e^{-(\xi-\ep)R_*} &\text{for } \tau\in\Om_G\\
	\Gm(\tau) &\text{for } \tau\in\Om_B
	\end{cases}
	\end{equation}
	and 
	\begin{equation}\label{eq: def of D_tau^R}
	D_\tau^{(R_*)}=
	\begin{cases}
	B_* &\text{for } \tau\in\Om_G\\
	\Gm(\tau) &\text{for } \tau\in\Om_B.
	\end{cases}
	\end{equation}
	For any $0\leq j<\ell$ we can write 
	\begin{align*}
		\sum_{0\leq k< \ell}\lt(\ind_{\Om_B}\log \Gm\rt)(\sg^{kR_*}(\om))
		=
		\sum_{0\leq k< j }\lt(\ind_{\Om_B}\log \Gm\rt)(\sg^{kR_*}(\om))
		+
		\sum_{j\leq k< \ell }\lt(\ind_{\Om_B}\log \Gm\rt)(\sg^{kR_*}(\om)).
	\end{align*}
	The definition of $\ell(\om)$, \eqref{def: coating length}, then implies that 
	\begin{align*}
		\frac{1}{j}\sum_{0\leq k< j }\lt(\ind_{\Om_B}\log \Gm\rt)(\sg^{kR_*}(\om))
		> 
		\ep^{\frac{1}{2}}\rho R_*,
	\end{align*}
	and consequently that 
	\begin{align}\label{eq: avg sum log Gm for bad om}
		\frac{1}{\ell-j}\sum_{j\leq k< \ell}\lt(\ind_{\Om_B}\log \Gm\rt)(\sg^{kR_*}(\om))
		\leq
		\ep^{\frac{1}{2}}\rho R_*. 
	\end{align}
	Now, using \eqref{eq: Bm geq 55^R}, \eqref{eq: avg sum log Gm for bad om}, and \eqref{eq: def of gm constant} we see that the proportion of bad blocks is given by 
	\begin{align}
		\frac{1}{\ell-j}\#\set{j\leq k<\ell: \sg^{kR_*}(\om)\in\Om_B}
		&=
		\frac{1}{\ell-j}\sum_{j\leq k<\ell} \lt(\ind_{\Om_B}\rt)(\sg^{kR_*}(\om))
		\nonumber\\
		&\leq
		\frac{1}{(\ell-j)R_*\log 55}\sum_{j\leq k<\ell} \lt(\ind_{\Om_B}\log \Gm\rt)(\sg^{kR_*}(\om))
		\leq \gm.
		\label{eq: proportion bad blocks}
	\end{align}
	In view of \eqref{eq: def of Phi_tau^R}, using \eqref{eq: avg sum log Gm for bad om}, \eqref{eq: proportion bad blocks}, for any $0\leq j<\ell$ we have 
	\begin{align}
		\prod_{k=j}^{\ell-1} \Phi_{\sg^{kR_*}(\om)}^{(R_*)}
		&=
		\prod_{\substack{j\leq k<\ell \\ \sg^{kR_*}(\om)\in\Om_G}}
		B_*e^{-(\xi-\ep)R_*}
		\cdot 
		\prod_{\substack{j\leq k<\ell \\ \sg^{kR_*}(\om)\in\Om_B}}
		\Gm(\sg^{kR_*}(\om))
		\nonumber\\
		&\leq 
		\lt(B_*e^{-(\xi-\ep)R_*}\rt)^{(1-\gm)(\ell-j)}\cdot \exp\lt(\lt(\ep^{\frac{1}{2}}\rho R_*\rt)(\ell-j)\rt)
		\nonumber\\
		&=
		\lt(B_*^{1-\gm}\exp\lt(\lt(\ep^{\frac{1}{2}}\rho-(\xi-\ep)(1-\gm)\rt)R_*\rt)\rt)^{\ell-j}
		\nonumber\\
		&< 
		\lt(B_*\exp\lt(\lt(\ep^{\frac{1}{2}}\rho-(\xi-\ep)(1-\gm)\rt)R_*\rt)\rt)^{\ell-j}.
		\label{eq: est for coating lengths 1}
	\end{align}
	Now, since $B_*,\Gm(\om)\geq 1$ for all $\om\in\Om$, using \eqref{eq: def of D_tau^R} and \eqref{eq: avg sum log Gm for bad om}, we have that for $0\leq j<\ell$ 
	\begin{align}\label{eq: est for coating lengths 2}
		D_{\sg^{jR_*}(\om)}^{(R_*)}
		\leq
		B_*\Gm(\sg^{jR_*}(\om))
		\leq 
		B_*\cdot\prod_{\substack{j\leq k< \ell \\ \sg^{kR_*}(\om)\in\Om_B}} \Gm(\sg^{kR_*}(\om))
		\leq
		B_*\lt(e^{\ep^{\frac{1}{2}}\rho R_*}\rt)^{(\ell-j)}.		
	\end{align}
	Thus, inserting \eqref{eq: est for coating lengths 1} and \eqref{eq: est for coating lengths 2} into \eqref{eq: var coating length} we see that
	\begin{align*}
		\var\lt(\~\cL_\om^{\ell R_*} f\rt)
		&\leq
		\lt(\prod_{j=0}^{\ell-1} \Phi_{\sg^{jR_*}(\om)}^{(R_*)}\rt)\var(f)
		+ 
		\sum_{j=0}^{\ell-1}\lt(D_{\sg^{jR_*}(\om)}^{(R_*)}\cdot \prod_{k=j+1}^{\ell-1}\Phi_{\sg^{kR_*}(\om)}^{(R_*)}\rt)\nu_\om(f),
		\\
		&\leq 
		\lt(B_*\exp\lt(\lt(\ep^{\frac{1}{2}}\rho-(\xi-\ep)(1-\gm)\rt)R_*\rt)\rt)^{\ell}\var(f)
		\\
		&\qquad
		+
		\nu_\om(f)\sum_{j=0}^{\ell-1}
		B_*\lt(e^{\ep^{\frac{1}{2}}\rho R_*}\rt)^{(\ell-j)}
		\cdot 
		\lt(B_*\exp\lt(\lt(\ep^{\frac{1}{2}}\rho-(\xi-\ep)(1-\gm)\rt)R_*\rt)\rt)^{\ell-j-1}
		\\
		&
		=
		\lt(B_*\exp\lt(\lt(\ep^{\frac{1}{2}}\rho-(\xi-\ep)(1-\gm)\rt)R_*\rt)\rt)^{\ell}\var(f)
		\\
		&\qquad
		+
		B_*e^{\ep^{\frac{1}{2}}\rho R_*}\cdot \nu_\om(f)\sum_{j=0}^{\ell-1}		
		\lt(B_*\exp\lt(\lt(2\ep^{\frac{1}{2}}\rho-(\xi-\ep)(1-\gm)\rt)R_*\rt)\rt)^{\ell-j-1}.
	\end{align*}
	Therefore, using \eqref{obs: expon leq xi/2} in conjunction with \eqref{G1}, we have that
	\begin{align*}
		\var\lt(\cL_\om^{\ell R_*}f\rt)
		&\leq
		\lt(B_*e^{-\frac{\xi}{2}R_*}\rt)^\ell\var(f)
		+
		B_*e^{\ep^{\frac{1}{2}}\rho R_*}\cdot \nu_\om(f)\sum_{j=0}^{\ell-1}		
		\lt(B_*e^{-\frac{\xi}{2}R_*}\rt)^{\ell-j-1} 
		\\
		&\leq 
		\lt(\frac{1}{3}\rt)^\ell\var(f)
		+
		B_*e^{\ep^{\frac{1}{2}}\rho R_*}\cdot \nu_\om(f)\sum_{j=0}^{\ell-1}		
		\lt(\frac{1}{3}\rt)^{\ell-j-1},
	\end{align*} 
	and so we must have that 
	\begin{equation*}
		\var\lt(\~\cL_\om^{\ell R_*} f\rt)
		\leq 
		\frac{1}{3}\var(f)+2B_*e^{\ep^{\frac{1}{2}}\rho R_*}\nu_\om(f).
		\footnote{Our final estimate of Lemma~\ref{lem: cone cont for coating blocks} differs from Buzzi's \cite[Lemma 3.4]{buzzi_exponential_1999} because of a term in the sum that was overlooked. The problem can be corrected by carefully choosing $R_*=R_{a_0}$ and then by defining $a_*\geq a_0$ in terms of $R_*$ as we have done here.}
	\end{equation*}
	Thus, for any $f\in\sC_{\om,a_*}$ we have that 
	\begin{align*}
		\var(\cL_\om^{\ell R_*} f)
		\leq
		\frac{a_*}{3}\nu_\om(f)+\frac{a_*}{3}\nu_\om(f), 
	\end{align*}
	and consequently we have 
	\begin{align*}
		\~\cL_\om^{\ell R_*}(\sC_{\om,a_*})\sub \sC_{\sg^{\ell R_*}(\om),\sfrac{2a_*}{3}}\sub \sC_{\sg^{\ell R_*}(\om),a_*}
	\end{align*}
	as desired.
\end{proof}

The next lemma shows that the total length of the bad blocks take up only a small proportion of an orbit, however before stating the result we establish the following notation. For each  $n\in\NN$ we let $K_n\geq 0$ be the integer such that 
\begin{align}\label{eq: n KR+zt decomp}
	n=K_nR_*+\zt(n)
\end{align}
where $0\leq \zt(n)<R_*$ is a remainder term. Given $\om_0\in\Om$, let 
\begin{align}\label{eq: om_j notation}
	\om_j=\sg^{\ell(\om_{j-1})R_*}(\om_{j-1})
\end{align} 
for each $j\geq 1$. Then for each $n\in\NN$ we can break the $n$-length $\sg$-orbit of $\om_0$ in $\Om$ into $k_{\om_0}(n)$ blocks of length $\ell(\om_j)R_*$ (for $0\leq j\leq k_{\om_0}(n)$) plus some remaining block of length $r_{\om_0}(n)R_*$ where $0\leq r_{\om_0}(n)<\ell(\om_{k_{\om_0}(n)+1})$ plus a remainder segment of length $\zt(n)$, i.e. we can write 
\begin{align}\label{eq: n length orbit in om_j}
	n=\sum_{0\leq j\leq k_{\om_0}(n)}\ell(\om_j)R_* +r_{\om_0}(n)R_* + \zt(n); 
\end{align}
see Figure~\ref{figure: om_j}. We also note that \eqref{eq: n KR+zt decomp} and \eqref{eq: n length orbit in om_j} imply that 
\begin{align}\label{eq: Kn equality}
	K_n=\sum_{0\leq j\leq k_{\om_0}(n)}\ell(\om_j) +r_{\om_0}(n).
\end{align}
	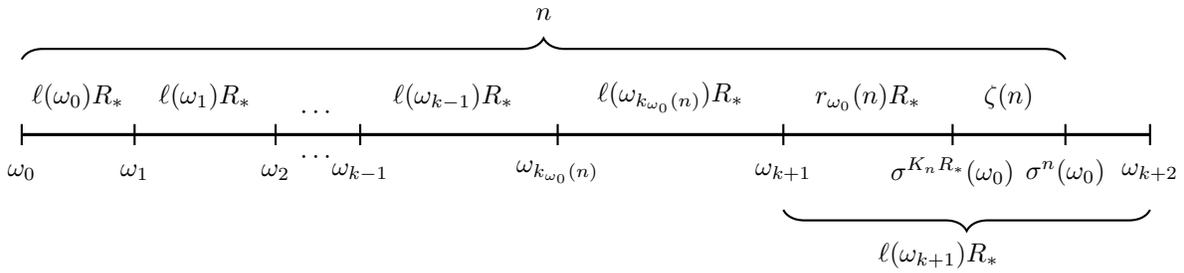
\begin{figure}[h]
	\centering	
	\begin{tikzpicture}[y=1cm, x=1.5cm, thick, font=\footnotesize]    
	\draw[line width=1.2pt,  >=latex'](0,0) -- coordinate (x axis) (10,0);       
	
	\draw (0,-4pt)   -- (0,4pt) {};
	\node at (0, -.5) {$\om_0$};
	
	\node at (.5, .5) {$\ell(\om_0)R_*$};
	
	\draw (1,-4pt)   -- (1,4pt) {};
	\node at (1,-.5) {$\om_1$};

	\node at (1.625,.5) {$\ell(\om_1)R_*$};
	
	\draw (2.25,-4pt)  -- (2.25,4pt)  {};
	\node at (2.25,-.5) {$\om_2$};
	
	\node at (2.625,.3) {$\cdots$};
	\node at (2.625,-.3) {$\cdots$};
	
	\draw (3,-4pt)  -- (3,4pt)  {};
	\node at (3,-.5) {$\om_{k-1}$};
	
	\node at (3.825,.5) {$\ell(\om_{k-1})R_*$};
	
	\draw (4.75,-4pt)  -- (4.75,4pt)  {};
	\node at (4.75,-.5) {$\om_{k_{\om_0}(n)}$};
	
	\node at (5.75,.5) {$\ell(\om_{k_{\om_0}(n)})R_*$};	
		
	\draw (6.75,-4pt)  -- (6.75,4pt)  {};
	\node at (6.75,-.5) {$\om_{k+1}$};
	
	\node at (7.5,.5) {$r_{\om_0}(n)R_*$};
	
	\draw (8.25,-4pt)  -- (8.25,4pt)  {};
	\node at (8.25,-.5) {$\sg^{K_nR_*}(\om_0)$};
	
	\node at (8.75,.5) {$\zt(n)$};
	
	\draw (9.25,-4pt)  -- (9.25,4pt)  {};
	\node at (9.25,-.5) {$\sg^{n}(\om_0)$};
	
	\draw (10,-4pt)  -- (10,4pt)  {};
	\node at (10,-.5) {$\om_{k+2}$};

	\draw[decorate,decoration={brace,amplitude=8pt,mirror}] 
	(6.75,-1)  -- (10,-1) ; 
	\node at (8.125,-1.6){$\ell(\om_{k+1})R_*$};
	
	\draw[decorate,decoration={brace,amplitude=8pt}] 
	(0,1)  -- (9.25,1) ; 
	\node at (4.625,1.6){$n$};
	
	\end{tikzpicture}
	\caption{The decomposition of $n=\sum_{0\leq j\leq k_{\om_0}(n)}\ell(\om_j)R_*+r_{\om_0}(n)R_* +\zt(n)$ and the fibers $\om_j$.}
	\label{figure: om_j}
\end{figure}
\begin{lemma}\label{lem: proportion of bad blocks}
	There exists a measurable function $N_0:\Om\to\NN$ such that for all $n\geq N_0(\om_0)$ and for $m$-a.e. $\om_0\in\Om$ with $y_*(\om_0)=0$  we have 
	\begin{align*}
		&E_{\om_0}(n):=
		\sum_{\substack{0\leq j \leq k_{\om_0}(n) \\ \om_j\in\Om_B}}\ell(\om_j) +r_{\om_0}(n)
		<
		H\cdot \ep K_n
		\leq 
		\frac{H}{R_*}\ep n
	\end{align*}
	where 
	\begin{align}\label{eq: def of H}
	H=H_\ep:=\frac{2((2+\rho) R_*)}{\ep^{\frac{1}{2}}\rho R_*},
	\end{align}
	and where $K_n$ is as in \eqref{eq: n KR+zt decomp}, 
	 $\om_j$ is as in \eqref{eq: om_j notation}, and $k_{\om_0}(n)$ and $r_{\om_0}(n)$ are as in \eqref{eq: n length orbit in om_j}.
\end{lemma}	
\begin{proof}
	Using \eqref{def: coating length}, we see that for any $\om\in\Om_B$ and $2\leq m\leq \ell(\om)$ then $m-1\geq m/2$, and thus 
	\begin{align}\label{eq: lower bound for sum of log Gamma}
		\sum_{0\leq j < m}\lt(\ind_{\Om_B}\log \Gm\rt)(\sg^{jR_*}(\om))
		>
		(\ep^{\frac{1}{2}}\rho R_*)(m-1)
		\geq 
		\frac{m}{2}(\ep^{\frac{1}{2}}\rho R_*).
	\end{align}
	Now, in view of Remark~\ref{rem: y*=0 implies ell finite}, we suppose that $\om_0\in\Om$ with $y_*(\om_0)=0$ and $\ell(\om_j)<\infty$ for each $j\geq 0$.
	Writing $n=K_nR_*+\zt(n)$ and using \eqref{eq: Kn equality} and \eqref{eq: lower bound for sum of log Gamma} we have 
	\begin{align}
		\sum_{0\leq j < K_n}\lt(\ind_{\Om_B}\log \Gm\rt)(\sg^{jR_*}(\om_0))
		&=
		\sum_{\substack{0\leq j \leq k_{\om_0}(n) \\ \om_j\in\Om_B}}\log\Gm(\om_j)
		+
		\sum_{\substack{0\leq j < r_{\om_0}(n) \\ \sg^{jR_*}(\om_{k_{\om_0}(n)})\in\Om_B}}
		\log\Gm(\sg^{jR_*}(\om_{k_{\om_0}(n)}))
		\nonumber\\
		&>
		\sum_{\substack{0\leq j \leq k_{\om_0}(n) \\ \om_j\in\Om_B}}
		\lt(\frac{\ell(\om_j)}{2}(\ep^{\frac{1}{2}}\rho R_*)\rt)
		+
		\frac{r_{\om_0}(n)}{2}(\ep^{\frac{1}{2}}\rho R_*)
		\nonumber\\
		&=
		\frac{1}{2}(\ep^{\frac{1}{2}}\rho R_*) \lt(\sum_{\substack{0\leq j \leq k_{\om_0}(n) \\ \om_j\in\Om_B}}\ell(\om_j)
		+
		r_{\om_0}(n)\rt)
		\nonumber\\
		&=
		\frac{1}{2}(\ep^{\frac{1}{2}}\rho R_*) E_{\om_0}(n).
		\label{eq: lower bound for sum of log Gamma with E(n)} 
	\end{align}
	Now using \eqref{eq: lower bound for sum of log Gamma with E(n)} and \eqref{eq: Buzzi 3.2 equiv first ineq} for all $n$ sufficiently large, say $n\geq N_0(\om_0)$, we can write 
	\begin{align*}
		\frac{1}{2}(\ep^{\frac{1}{2}}\rho R_*) E_{\om_0}(n)
		<
		\sum_{0\leq j < K_n }\lt(\ind_{\Om_B}\log \Gm\rt)(\sg^{jR_*}(\om_0))
		\leq 
		\ep((2+\rho) R_*)K_n,
	\end{align*}
	and thus we have 
	\begin{align*}
		E_{\om_0}(n)
		<
		\frac{2((2+\rho) R_*)}{+\ep^{\frac{1}{2}}\rho R_*}\ep K_n
		\leq 
		\frac{2((2+\rho) R_*)}{(\ep^{\frac{1}{2}}\rho R_*)R_*}\ep n
	\end{align*}
	as desired.
\end{proof}
To end this section we note that $\ep \cdot H_\ep\to 0$ as $\ep\to 0$. Thus, we let 
\begin{align}\label{eq: def ep 4}
	\ep_4:=\min\set{\ep_3, \sup\set{\ep>0: \ep \cdot H_\ep<\frac{1}{4}}}.
\end{align}

\section{Invariant Family of Measures}\label{sec:invariant}

In this section we adapt the methods of \cite{liverani_conformal_1998} and \cite{buzzi_exponential_1999} to establish the existence of invariant measures on the fibers $X_\om$ as well as some of their properties. 
We now present the main technical result from which all the rest of our convergence results follow. Set 
\begin{align*}
	\Dl:=\Dl_{a_*}=2\log \frac{C_*(3+a_*)}{\al_*}
\end{align*}
to be the diameter value defined in Lemma~\ref{lem: cone contraction for good om} for the cone parameter $a_*$ defined in Lemma~\ref{lem: cone cont for coating blocks}, and set 
\begin{align}\label{eq: def of ep 5}
	\ep_5:=\min\set{\ep_4, \frac{\xi}{2(4+3\xi)}}.
\end{align}
\begin{lemma}\label{lem: exp conv in C+ cone}
	Let $\ep<\ep_5$ and $V:\Om\to(0,\infty)$ not necessarily measurable. Then there exists $\vta\in(0,1)$ and a (not necessarily measurable) function $N_3:\Om\to\NN$ such that for $m$-a.e. $\om\in\Om$, all $n\geq N_3(\om)$, all $l\geq 0$, and all $|p|\leq n$ we have 
	\begin{align}\label{eq: lem 8.1 ineq}
		\Ta_{\sg^{n+p}(\om), +}\lt(\~\cL_{\sg^p(\om)}^n f, \~\cL_{\sg^{p-l}(\om)}^{n+l} h\rt) \leq \Dl\vta^n
	\end{align} 
	for all $f\in\sC_{\sg^p(\om), +}$ with $\nu_{\sg^p(\om)}(f)=1$ and $\var(f)\leq e^{\ep n}V(\om)$, and all $h\in\sC_{\sg^{p-l}(\om), +}$ with $\nu_{\sg^{p-l}(\om)}(h)=1$ and $\var(h)\leq e^{\ep (n+l)}V(\om)$. Furthermore, $\Dl$ and $\vta$ do not depend on $V$.
\end{lemma}
\begin{proof}
We begin by noting that by \eqref{eq: L_om is a weak contraction on C_+} for each $l\geq 0$ we have that $\~\cL_{\sg^{p-l}(\om)}^l h\in\sC_{\sg^{p}(\om),+}$ for each $h\in\sC_{\sg^{p-l}(\om),+}$, and let 
\begin{align*}
	h_l=\~\cL_{\sg^{p-l}(\om)}^l h.
\end{align*}
Set $y_*=y_*(\sg^p(\om))$ (defined in Section~\ref{sec:bad}) and let $d_*=d_*(\sg^p(\om))\geq 0$ be the smallest integer that satisfies 
\begin{flalign} 
	& y_*+d_*R_*\geq \frac{\ep n+\log V(\om)}{\xi-\ep}, 
	\label{eq: d prop1} 
	&\\
	& C_\ep(\sg^{p+y_*+d_*R_*}(\om))\leq B_*.
	\label{eq: d prop2}  
\end{flalign}
where $\xi$ was defined in \eqref{eq: def of xi}.
Choose 
\begin{align}\label{eq: coose N_1 ge log V}
	N_1(\om)\ge \frac{\log V(\om)}{\ep}
\end{align} 
and let $n\geq N_1(\om)$. Now using \eqref{eq: coose N_1 ge log V} to write
\begin{align}\label{eq: d_* low bound}
	\frac{4\ep n}{\xi}=\frac{\ep n+\ep n}{\xi/2}\geq \frac{\ep n+\log V(\om)}{\xi/2}
\end{align}
and then using \eqref{eq: def of ep_0} and \eqref{eq: def of y_*}, we see that \eqref{eq: d prop1} is satisfied for any $d_*R_*\geq 4\epsilon n/\xi$.
Using \eqref{def j*2}, the construction of $y_*$, and the ergodic decomposition of $\sigma^{R_*}$ following \eqref{def j*2}, we have for $m$-a.e. $\omega\in\Om$ there is an infinite, increasing sequence of integers $d_j\ge 0$ satisfying \eqref{eq: d prop2}. Furthermore, \eqref{def j*2} implies that 
\begin{align*}
	\lim_{n\to\infty}\frac{1}{n/R_*}\#\set{0\leq k< \frac{n}{R_*}: C_\ep\lt(\sg^{\pm kR_*+y_*(\om)}(\om)\rt)> B_*} <\ep,
\end{align*}
and thus for $n\in\NN$ sufficiently large (depending measurably on $\om$), say $n\geq N_2(\om)\geq N_1(\om)$, we have that 
\begin{align}\label{eq: d_* up bound}
	\#\set{0\leq k< \frac{n}{R_*}: C_\ep\lt(\sg^{\pm kR_*+y_*(\om)}(\om)\rt)> B_*} < \frac{\ep n}{R_*}.
\end{align}
Thus, combining \eqref{eq: d_* low bound} and \eqref{eq: d_* up bound}, we see that  the smallest integer $d_*$ satisfying \eqref{eq: d prop1} and \eqref{eq: d prop2} also satisfies 
\begin{align}\label{eq: dR is Oepn}
	d_*R_*\leq\frac{4\ep n}{\xi}+\ep n=\lt(\frac{4+\xi}{\xi }\rt)\ep n.
\end{align}
Let 
\begin{align}\label{eq: def of hat y_*}
	\hat y_*=y_*+d_*R_*.
\end{align}
	\begin{figure}[h]
		\centering	
		\begin{tikzpicture}[y=1cm, x=1.5cm, thick, font=\footnotesize]    
		\draw[line width=1.2pt,  >=latex'](0,0) -- coordinate (x axis) (10,0);       
		
		\draw (0,-4pt)   -- (0,4pt) {};
		\node at (0, -.5) {$\tau_{-l}$};
		
		\node at (.5, .5) {$l$};
		
		\draw (1,-4pt)   -- (1,4pt) {};
		\node at (1,-.5) {$\tau=\sg^p(\om)$};
		
		\node at (1.5,.5) {$\hat y_*$};
		
		\draw (2,-4pt) -- (2,4pt) {};
		\node at (2,-.5) {$\tau_0$};
		
		\node at (2.5,.5) {$\ell(\tau_0)R_*$};
		
		\draw (3,-4pt)  -- (3,4pt)  {};
		\node at (3,-.5) {$\tau_1$};
		
		\node at (3.5,.3) {$\cdots$};
		\node at (3.5,-.3) {$\cdots$};
		
		\draw (4,-4pt)  -- (4,4pt)  {};
		\node at (4,-.5) {$\tau_{j}$};
		
		\node at (4.5,.5) {$\ell(\tau_{j})R_*$};
		
		\draw (5,-4pt)  -- (5,4pt)  {};
		\node at (5,-.5) {$\tau_{0}^*$};
		
		\node at (5.5,.5) {$\ell(\tau_{0}^*)R_*$};	
		
		\draw (6,-4pt)  -- (6,4pt)  {};
		\node at (6,-.5) {$\tau_1^*$};
		
		\node at (6.5,.3) {$\cdots$};
		\node at (6.5,-.3) {$\cdots$};

		\draw (7,-4pt)  -- (7,4pt)  {};
		\node at (7,-.5) {$\tau_{k}^*$};
		
		\node at (7.6,.5) {$\hat r_{\tau_0}(n)R_*$};
		
		\draw (8.2,-4pt)  -- (8.2,4pt)  {};
		\node at (8.7,.5) {$\hat\zt(n)$};
		
		\draw (9.2,-4pt)  -- (9.2,4pt)  {};
		\node at (9.2,-.5) {$\sg^n(\tau)$};
		
		\draw (10,-4pt)  -- (10,4pt)  {};
		\node at (10,-.5) {$\tau_{k+1}^*$};

		
		\draw[decorate,decoration={brace,amplitude=8pt,mirror}] 
		(2,-2.1)  -- (8.2,-2.1) ; 
		\node at (5.1,-2.7){$\Sg R_*$};		
		
	\draw[decorate,decoration={brace,amplitude=8pt,mirror}] 
		(2,-1)  -- (5,-1) ; 
		\node at (3.5,-1.6){$\Sg_I R_*$};	
		
	\draw[decorate,decoration={brace,amplitude=8pt,mirror}] 
		(5,-1)  -- (7,-1) ; 
		\node at (5.9,-1.6){$\Sg_C R_*$};

		\draw[decorate,decoration={brace,amplitude=8pt,mirror}] 
		(7,-1)  -- (10,-1) ; 
		\node at (8.4,-1.6){$\ell(\tau_{k}^*)R_*$};
		
		\draw[decorate,decoration={brace,amplitude=8pt}] 
		(1,1)  -- (9.2,1) ; 
		\node at (5.1,1.6){$n$};
		
		\end{tikzpicture}
		\caption{The fibers $\tau_i$, $\tau_i^*$ and the decomposition of $n=\hat y_*+\Sg R_*+\hat \zt(n)$.}
		\label{figure: om_k}
	\end{figure}
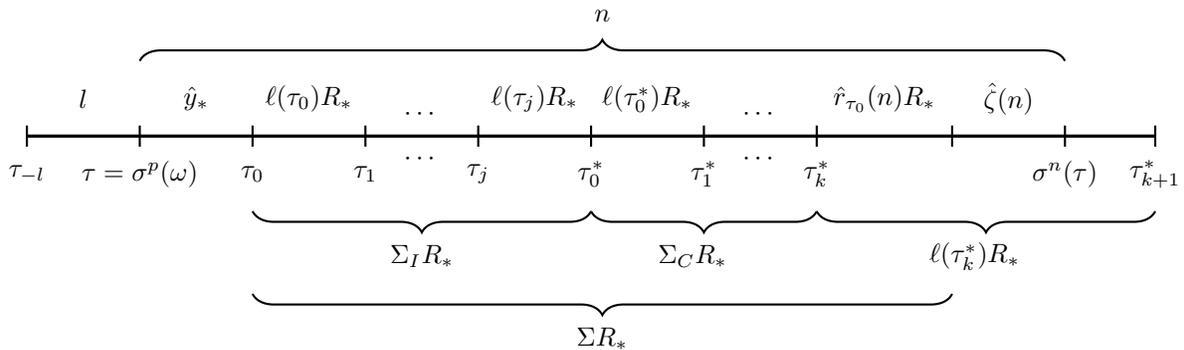
Now, we wish to examine the iteration of our operator cocycle along a collection $\Sg R_*$ of blocks, each of length $\ell(\om)R_*$, so that the images of $\~\cL_\om^{\ell(\om)R_*}$ are contained in $\sC_{\sg^{\ell(\om)R_*}(\om),a_*}$ as in Lemma~\ref{lem: cone cont for coating blocks}. This collection $\Sg R_*$
\footnote{Here the number $\Sg$ is playing the same role as the number $K_{n-\hat y_*}$ from Lemma~\ref{lem: proportion of bad blocks}.}
will be comprised of an initial segment, $\Sg_I R_*$, of bad blocks followed by a segment $\Sg_C R_*$ of blocks, beginning with a good block, over which we obtain cone contraction with a finite diameter image, and finally a remainder segment, of length $\hat r_{\tau_0}(n)R_*$, which is not long enough to ensure cone contraction; see Figure~\ref{figure: om_k}.

We begin by establishing some simplifying notation. To that end, set $\tau=\sg^{p}(\om)$, $\tau_{-l}=\sg^{p-l}(\om)$, and $\tau_0=\sg^{p+\hat y_*}(\om)$; see Figure~\ref{figure: om_k}. Note that in light of \eqref{prop j*3} and \eqref{eq: def of hat y_*} we have that 
\begin{align}\label{eq: y_* tau_0}
	y_*(\tau_0)=y_*(\sg^{p+\hat y_*}(\om))=0.
\end{align}
Now, by our choice of $d_*$, we have that if $f\in\sC_{\tau,+}$ with $\nu_{\tau}(f)=1$ and $\var(f)\leq e^{\ep n}V(\om)$, then 
\begin{align}\label{eq: f in a cone}
	\~\cL_\tau^{\hat y_*} f\in\sC_{\tau_0,a_*}.
\end{align}
Indeed, applying Proposition~\ref{lem: LY ineq 2} and \eqref{eq: d prop1} and \eqref{eq: d prop2} from  above, we have 
\begin{align}
	\var(\~\cL_\tau^{\hat y_*}f)
	&\leq 
	C_{\ep}(\sg^{\hat y_*}(\tau))e^{-(\xi-\ep)\hat y_*}\var(f)+C_{\ep}(\sg^{\hat y_*}(\tau))\nu_{\tau}(f)
	\nonumber\\
	&\leq 
	B_*e^{-(\xi-\ep)\hat y_*}\var(f)+B_*\nu_{\tau}(f)
	\nonumber\\&
	\leq 
	B_*\frac{\var(f)}{e^{\ep n}V(\om)}+B_* \leq 2B_*\leq \frac{a_*}{3}, 
\end{align}
where we recall that $a_*> 6B_*$ is defined in \eqref{eq: def of a_*}. 
A similar calculation yields that if $h\in\sC_{\tau_{-l},+}$ with $\nu_{\tau_{-l}}(h)=1$ and $\var(h)\leq e^{\ep (n+l)}V(\om)$, then $\~\cL_{\tau_{-l}}^{l+\hat y_*} h\in\sC_{\tau_0,a_*}$.

Next, we wish to find the first good block; we denote the origin of this first good block by $\tau_0^*$. If $\tau_0\in\Om_G$, 
then we are done, and we set $\tau_0=\tau_0^*$. However, it may be the case that $\tau_0\in\Om_B$ and there may even be several bad blocks (originating at $\tau_0$) in a row before we encounter the first good block. 

Now, if $\tau_0\in\Om_B$, for each $i\geq 1$ while $\sg^{\ell(\tau_{i-1})R_*}(\tau_{i-1})\in\Om_B$ set $\tau_i=\sg^{\ell(\tau_{i-1})R_*}(\tau_{i-1})$. Note that since $y_*(\tau_0)=0$, by Remark~\ref{rem: y*=0 implies ell finite}, we must have that $\ell(\tau_i)<\infty$ for each $i\geq 0$.
Let $j\geq 1$ be the largest integer such that $\tau_i\in\Om_B$ for each $1\leq i\leq j$, and let 
\begin{align*}
	\Sg_I:=\sum_{i=0}^j\ell(\tau_i)
\end{align*}  
be the total length of the initial coating intervals starting at $\tau_0$. Hence $\Sg_IR_*$ is the total length of all the consecutive initial bad blocks; see Figure~\ref{figure: om_k}. Such an integer $j$ must exist for $m$-a.e. $\om\in\Om$ as Lemma~\ref{lem: proportion of bad blocks} 
ensures that, for all $n\geq N_1(\om)\geq N_0(\om)$, we have
\begin{align*}
	\Sg_I
	\leq E_{\tau_0}(n-\hat y_*) \leq E_{\tau_0}(n)
	< \frac{H}{R_*}\ep n, 
\end{align*}
and thus an initial segment of consecutive bad blocks can only take up a small portion of an orbit of length $n$. 
Therefore, we set $\tau_0^*=\sg^{\ell(\tau_{j})R_*}(\tau_{j})$ to be the origin of the first good block after $\tau$, and for each $i\geq 1$ let $\tau_i^*=\sg^{\ell(\tau_{i-1}^*)R_*}(\tau_{i-1}^*)$.

As there are only finitely many blocks (good and bad) that will occur within an orbit of length $n$, let $k\geq 1$ be the integer such that 
\begin{align*}
	\hat y_*+\Sg_IR_*+\sum_{i=0}^{k-1} \ell(\tau_i^*)R_* \leq n < \hat y_*+\Sg_IR_*+\sum_{i=0}^{k} \ell(\tau_i^*)R_*,
\end{align*}
and let 
\begin{equation*}
	\Sg_C:=\sum_{i=0}^{k-1} \ell(\tau_i^*)
	\qquad \text{ and } \qquad
	\hat r_{\tau_0}(n):=r_{\tau_0}(n-\hat y_*)
\end{equation*} 
where $r_{\tau_0}(n-\hat y_*)$ is the number defined in \eqref{eq: n length orbit in om_j}.
In other words, $\Sg_CR_*$ is the total length of the consecutive good and bad blocks beginning with the first good block originating from $\tau_0^*$. 
Finally setting 
$$
 	\Sg=\Sg_I+\Sg_C+\hat r_{\tau_0}(n)
 	\footnote{Note that if $\tau_0\in\Om_G$, in which case we have $\tau_0=\tau_0^*$, then $\Sg_I=0$ and $\Sg=\Sg_+\hat r_{\tau_0}(n)$.}
 	\qquad \text{ and } \qquad
 	\hat \zt(n):=n-\hat y_*-\Sg R_*,
$$
we have the right decomposition of our orbit length $n$ into blocks which do not expand distances in the fiber cones $\sC_{\om,a_*}$ and $\sC_{\om,+}$.
Now let
\begin{align}\label{eq: def of N_3}
	n\geq N_3(\om):=\max\set{N_2(\om), \frac{R_*}{\ep}}.
\end{align}
Since $y_*, \hat\zt(n)\leq R_*$, by \eqref{eq: def of N_3}, \eqref{eq: dR is Oepn}, and our choice of $\ep<\ep_5$ \eqref{eq: def of ep 5}, we must have that
\begin{align}\label{eq: SgR is big O ep n}
	\Sg R_*
	&=
	n-\hat y_*-\hat \zt(n)
	=
	n-y_*-d_*R_*-\hat\zt(n)
	\nonumber\\
	&\geq 
	n-\lt(\frac{4+\xi}{\xi}\rt)\ep n-2R_*
	\geq 
	n-\lt(\frac{4+\xi}{\xi}\rt)\ep n-2\ep n
	\nonumber\\
	&\geq 
	n\lt(1-\ep\lt(\frac{4+3\xi}{\xi}\rt)\rt) > \frac{n}{2}.
\end{align}
Now we note that since $\~\cL_{\tau_k^*}^{\hat\zt(n)}(\sC_{\tau_k^*,+})\sub\sC_{\sg^{n+p}(\om), +}$ we have that $\~\cL_{\tau_k^*}^{\hat\zt(n)}$ is a weak contraction, and hence, we have 
\begin{align}\label{eq: L^r a weak contraction on C_+}
	\Ta_{\sg^{n+p}(\om), +}\lt(\~\cL_{\tau_k^*}^{\hat\zt(n)} f' , \~\cL_{\tau_k^*}^{\hat\zt(n)} h' \rt)\leq \Ta_{\tau_k^*,+}(f',h'), 
	\qquad f',h'\in\Ta_{\tau_k^*,+}.	
\end{align}
Recall that $E_{\tau_0}(n-\hat y_*)$, defined in Lemma~~\ref{lem: proportion of bad blocks}, is the total length of the coating intervals (bad blocks) of the $n-\hat y_*$ length orbit starting at $\tau_0$, i.e. 
\begin{align*}
	E_{\tau_0}(n-\hat y_*)&=\Sg_I+\sum_{\substack{0\leq j< k \\ \tau_j^*\in\Om_B}} \ell(\tau_j^*) +r_{\tau_0}(n-\hat y_*).
\end{align*}
Lemma~\ref{lem: proportion of bad blocks} then gives that 
\begin{align}\label{eq: est of S}
	E_{\tau_0}(n-\hat y_*)<H\ep\Sg. 
\end{align}
We are now poised to calculate \eqref{eq: lem 8.1 ineq}, but first we note that we can write
\begin{align}
	n&=\hat y_* + \Sg R_* +\hat\zt(n)
	\nonumber\\
	&=\hat y_* + \Sg_I R_* +\Sg_C R_*+\hat\zt(n)
	\nonumber\\
	&=\hat y_* + \Sg_I R_* +R_*+(\Sg_C-1) R_*+\hat\zt(n)
	\label{eq: lem 8.1 decomp of n final ineq}	
\end{align}
and that the number of good blocks contained in the orbit of length $n-\hat y_*$ is given by 
\begin{align}\label{eq: number of good blocks}
	\Sg_G:=\#\set{0\leq j\leq k: \tau_j^*\in\Om_G}
	=\Sg-E_{\tau_0}(n-\hat y_*)
	\leq \Sg_C.
\end{align} 
Now, using \eqref{eq: lem 8.1 decomp of n final ineq} we combine (in order) \eqref{eq: L^r a weak contraction on C_+}, \eqref{eq: Ta+ leq Ta}, and Lemma~\ref{lem: cone contraction for good om} (repeatedly) in conjunction with the fact that $\tau_0^*\in\Om_G$ to see that
\begin{align}
	&\Ta_{\sg^{n+p}(\om), +}
	\lt(\~\cL_{\tau}^{n} (f), \~\cL_{\tau_{-l}}^{n+l} (h) \rt)
	\nonumber\\
	&\quad
	=
	\Ta_{\sg^{n+p}(\om), +}
	\lt(
	\~\cL_{\tau_k^*}^{\hat\zt(n)} \circ \~\cL_{\tau_0}^{\Sg R_*} \circ \~\cL_{\tau}^{\hat y_*} (f)
	,
	\~\cL_{\tau_k^*}^{\hat\zt(n)} \circ \~\cL_{\tau_0}^{\Sg R_*} \circ \~\cL_{\tau}^{\hat y_*}\circ \~\cL_{\tau_{-l}}^{l} (h) 
	\rt)
	\nonumber\\
	&\quad
	\leq
	\Ta_{\tau_k^*, +}
	\lt(
	\~\cL_{\tau_0}^{\Sg R_*} \circ \~\cL_{\tau}^{\hat y_*} (f)
	,
	\~\cL_{\tau_0}^{\Sg R_*} \circ \~\cL_{\tau}^{\hat y_*} (h_l) 
	\rt)	
	\nonumber\\
	&\quad
	\leq
	\Ta_{\tau_k^*, a_*}
	\lt(
	\~\cL_{\tau_0}^{\Sg R_*} \circ \~\cL_{\tau}^{\hat y_*} (f)
	,
	\~\cL_{\tau_0}^{\Sg R_*} \circ \~\cL_{\tau}^{\hat y_*} (h_l) 
	\rt)
	\nonumber\\
	&\quad
	=
	\Ta_{\tau_k^*, a_*}
	\lt(
	\~\cL_{\tau_1^*}^{(\Sg_C-1) R_*} \circ \~\cL_{\tau_0^*}^{R_*} \circ
	\~\cL_{\tau_0}^{\Sg_I R_*} \circ \~\cL_{\tau}^{\hat y_*} (f)
	,
	\~\cL_{\tau_1^*}^{(\Sg_C-1) R_*} \circ \~\cL_{\tau_0^*}^{R_*} \circ
	\~\cL_{\tau_0}^{\Sg_I R_*} \circ \~\cL_{\tau}^{\hat y_*} (h_l) 
	\rt)
	\nonumber\\	
	&\quad
	\leq
	\lt(\tanh\lt(\frac{\Dl}{4}\rt)\rt)^{\Sg_G-1}
	\Ta_{\tau_1^*, a_*}
	\lt(
	\~\cL_{\tau_0^*}^{R_*} \circ
	\~\cL_{\tau_0}^{\Sg_I R_*} \circ \~\cL_{\tau}^{\hat y_*} (f)
	,
	\~\cL_{\tau_0^*}^{R_*} \circ
	\~\cL_{\tau_0}^{\Sg_I R_*} \circ \~\cL_{\tau}^{\hat y_*} (h_l) 
	\rt).
	\label{eq: cone Cauchy 1}
\end{align} 
Now since $\tau_0^*\in \Om_G$ and since 
$\~\cL_{\tau_0}^{\Sg_I R_*} \circ \~\cL_{\tau}^{\hat y_*}$ is a weak contraction (and thus  non-expansive with respect to $\Ta_{\tau_1^*, a_*}$), in light of \eqref{eq: f in a cone}, applying Lemmas~\ref{lem: cone contraction for good om} and \ref{lem: cone cont for coating blocks} gives that 
\begin{align}\label{eq: cone Cauchy Ta leq Dl}
	\Ta_{\tau_1^*, a_*}
	\lt(
	\~\cL_{\tau_0^*}^{R_*} \circ
	\~\cL_{\tau_0}^{\Sg_I R_*} \circ \~\cL_{\tau}^{\hat y_*} (f)
	,
	\~\cL_{\tau_0^*}^{R_*} \circ
	\~\cL_{\tau_0}^{\Sg_I R_*} \circ \~\cL_{\tau}^{\hat y_*} (h_l) 
	\rt)
	\leq \Dl.
\end{align}
Using \eqref{eq: number of good blocks}, the fact that $E_{\tau_0}(n-\hat y_*)\geq 1$, \eqref{eq: est of S}, and \eqref{eq: SgR is big O ep n}, we see that
\begin{align}
	\Sg_G-1
	&= 
	\Sg-(E_{\tau_0}(n-\hat y_*)+1)
	\nonumber\\
	&\geq 
	\Sg-2E_{\tau_0}(n-\hat y_*)
	\nonumber\\
	&\geq
	\Sg-2H\ep\Sg
	\nonumber\\
	&=
	\Sg\lt(1-2H\ep\rt)
	\nonumber\\
	&\geq
	\frac{\lt(1-2H\ep\rt)n}{2R_*}.
	\label{eq: cone Cauchy exponent}
\end{align}
Note that $1-2H\ep>0$ by our choice of $\ep<\ep_5\leq\ep_4$ \eqref{eq: def ep 4}.
Finally, inserting \eqref{eq: cone Cauchy Ta leq Dl} and \eqref{eq: cone Cauchy exponent} into \eqref{eq: cone Cauchy 1} gives
\begin{align*}
	&\Ta_{\sg^{n+p}(\om), +}
	\lt(\~\cL_{\tau}^{n} (f), \~\cL_{\tau_{-l}}^{n+l} (h) \rt)
	\leq 
	\Dl \vta^n,
\end{align*}
where 
	$$
		\vta:=\lt(\tanh\lt(\frac{\Dl}{4}\rt)\rt)^{\frac{\lt(1-2H\ep\rt)}{2R_*}}<1,
	$$
	which completes the proof.
\end{proof}
\begin{remark}\label{rem: N_1 measurable}
	Note that if $V:\Om\to(0,\infty)$ is measurable then so is $d_*:\Om\to\NN$. Thus we may choose $N_1,N_2:\Om\to\NN$ measurable, which implies that $N_3:\Om\to\NN$ is also measurable. Also note that our choice of $N_3$ depends on $V$.
\end{remark}
The following corollary establishes the existence of a family of fiberwise invariant (with respect to the normalized operator $\~\cL_\om$) densities (with respect to $\nu_\om$).
\begin{corollary}\label{cor: exist of unique dens q}
	There exists a family of functions $(q_\om)_{\om\in\Om}$ with $q_\om\in\BV(X_\om)$ and $\nu_\om(q_\om)=1$ for $m$-a.e. $\om\in\Om$ such that 
	\begin{align*}
		\cL_\om q_\om = \lm_\om q_{\sg(\om)}.
	\end{align*}
\end{corollary}
\begin{proof}
	Let $\ep<\ep_5$, and let $f\in\BV(I)$ with $f\geq 0$ such that $\nu_\om(f_\om)=1$ for $m$-a.e. $\om\in\Om$. Take $V(\om)=\var(f)$. Then $f_\om\in\sC_{\om,+}$ for each $\om\in\Om$. Setting 
	\begin{align*}
		f_{\om,n}:=\~\cL_{\sg^{-n}(\om)}^n f_{\sg^{-n}(\om)} 
	\end{align*}
	for each $n\geq 0$, we note that Lemma~\ref{lem: exp conv in C+ cone} (with $p=-n$) and Lemma~\ref{lem: birkhoff cone contraction} (with $\vrho=\nu_{\sg^{p+n}(\om)}$ and $\norm{\spot}=\norm{\spot}_\infty$) together imply that the sequence $(f_{\om,n})_{n\geq 0}$ is Cauchy in $\sC_{\om,+}$. Indeed, we have 
	\begin{align*}
		\norm{ f_{\om,n} -f_{\om,n+l}}_\infty
		&\leq 
		\norm{f_{\om,n}}_\infty \lt(e^{\Dl\vta^n}-1\rt)
	\end{align*} 
	for each $n\geq N_3(\om)$ and $l\geq 0$. Thus, there must exist some $q_{\om,f}\in\sC_{\om,+}\sub \BV(X_\om)$ such that 
	\begin{align}\label{eq: lim def of inv dens}
		\lim_{n\to \infty} f_{\om,n}=q_{\om,f}.
	\end{align}
	Since $\nu_\om(f_\om)=1$ for $m$-a.e. $\om\in\Om$ we must also have that $\nu_\om(q_{\om,f})=1$ for $m$-a.e. $\om\in\Om$.
	\footnote{Indeed, one can see this by considering the continuous linear functional $\Lm_\om$ which is equivalent to $\nu_\om$ on $\BV(I)$. As $f_{\om,n}\to q_{\om,f}$ we must have $1=\Lm_{\om}(f_{\om,n})\to \Lm(q_{\om,f})=\nu_\om(q_{\om,f})=1$.} 
	
	Now, to see that the limit in \eqref{eq: lim def of inv dens} does not depend on our choice of $f$, we suppose that $h\in\BV(I)$ with $\nu_\om(h_\om)=1$ for $m$-a.e. $\om\in\Om$ and $V(\om)=\var(h)$. Then, as before we obtain 
	\begin{align*}
		q_{\om,h}=\lim_{n\to \infty} h_{\om,n}. 
	\end{align*} 
	Using Lemma~\ref{lem: exp conv in C+ cone} again, we get 
	\begin{align}
		\Ta_{\om,+}\lt(q_{\om,f}, q_{\om,h}\rt)
		\leq 
		\Ta_{\om,+}\lt(q_{\om,f}, f_{\om,n}\rt)
		+
		\Ta_{\om,+}\lt(f_{\om,n}, h_{\om,n}\rt)
		+
		\Ta_{\om,+}\lt(q_{\om,h}, h_{\om,n}\rt)
		\leq 
		3\Dl\vta^n
		\label{eq: Ta estimate for q_om,f}
	\end{align}
	for each $n\geq N_3(\om)$. 
	Inserting \eqref{eq: Ta estimate for q_om,f} into Lemma~\ref{lem: birkhoff cone contraction} with $\vrho=\nu_{\sg^{p+n}(\om)}$ and $\norm{\spot}=\norm{\spot}_\infty$ then gives 
	\begin{align*}
		\norm{q_{\om,f} - q_{\om,h}}_\infty 
		&\leq 
		\lt(e^{\Ta_{\om,+}\lt(q_{\om,f}, q_{\om,h}\rt)}-1\rt)\norm{q_{\om,f}}_\infty
		\\
		&\leq
		\lt(e^{3\Dl\vta^n}-1\rt)\norm{q_{\om,f}}_\infty,
	\end{align*}
	which goes to zero exponentially fast. 
	But this of course implies that $q_{\om,f}=q_{\om,h}$ for $m$-a.e. $\om\in\Om$, and thus we denote their common value by $q_\om$. 
	In particular, \eqref{eq: lim def of inv dens} gives that 
	\begin{align}\label{eq: q_om fixed by norm tr op}
		\~\cL_\om q_\om =\~\cL_\om\lt(\lim_{n\to\infty}\~\cL_{\sg^{-n}(\om)}^n\ind_{\sg^{-n}(\om)} \rt)
		=\lim_{n\to\infty} \~\cL_{\sg^{-n}(\om)}^{n+1}\ind_{\sg^{-n}(\om)}
		=q_{\sg(\om)}.
	\end{align}

\end{proof}
\begin{remark}
	The proof of the previous Lemma shows that $q_\om\in\BV(X_\om)$ for each $\om\in\Om$. Furthermore, we note that  it follows from Proposition~\ref{lem: LY ineq 2} and \eqref{eq: q_om fixed by norm tr op} that 
	\begin{align}\label{eq: var q leq C_ep}
		\var(q_\om)\leq C_\ep(\om).
	\end{align}
	Indeed, for each $n\in\NN$, Proposition~\ref{lem: LY ineq 2} gives 
	\begin{align*}
		\var(\~\cL_{\sg^{-n}(\om)}^n\ind_{\sg^{-n}(\om)})
		\leq 
		C_\ep(\om)e^{-(\xi-\ep)n}\var(\ind_{\sg^{-n}(\om)})+C_\ep(\om),
	\end{align*}
	and in light of \eqref{eq: q_om fixed by norm tr op}, we must in fact have \eqref{eq: var q leq C_ep}.
	
	Using \eqref{eq: var q leq C_ep} we get that 
	\begin{align}\label{eq: measurable up bdd for q}
		\norm{q_\om}_{\BV}
		=
		\norm{q_{\om}}_{\infty}+\var(q_\om)
		\leq 
		2\var(q_\om)+\nu_\om(q_\om)
		\leq 
		2C_\ep(\om)+1.
	\end{align}
\end{remark}
The following lemma shows that the $\BV$ norm of the invariant density $q_\om$ does not grow too much along a $\sg$-orbit of fibers by providing a measurable upper bound. 
\begin{lemma}\label{lem: BV norm q om growth bounds}
	For all $\dl>0$ there exists a measurable random constant $C(\om,\dl)>0$ such that for all $k\in\ZZ$ and $m$-a.e. $\om\in\Om$ we have 
	\begin{align*}
		\norm{q_{\sg^k(\om)}}_{\BV}
		=
		\norm{q_{\sg^k(\om)}}_{\infty}+\var(q_{\sg^k(\om)})
		\leq
		C(\om,\dl)e^{\dl|k|}.
	\end{align*} 
\end{lemma}
 \begin{proof}
	Let $\ep<\ep_5$ and let $0<\dl<2(\rho_0+1)$ where 
	\begin{align*}
		\rho_0:=\int_\Om \log L_\om \, dm(\om).
	\end{align*}
	By the ergodic theorem for $m$-a.e. $\om\in\Om$ there exists a measurable $n_\dl(\om)\in\NN$ such that for all $n\geq n_\dl(\om)$ and each choice of sign $+$ or $-$, we have
	\begin{align}\label{eq: sum of log L}
		\sum_{j=0}^{n-1}\log L_{\sg^{\pm j}(\om)} \leq (\rho_0+1)n.
	\end{align}
	Now let $n_*\in\NN$ so large that 
	\begin{align*}
		m\lt(\set{\om\in\Om: n_\dl(\om)> n_*}\rt)<\frac{\dl}{32(\rho_0+1)}.
	\end{align*}
	By \eqref{eq: var q leq C_ep} we have that $\var(q_\om)\leq C_\ep(\om)<\infty$ for $m$-a.e. $\om\in\Om$, and so for $A_\dl\geq 1$ sufficiently large we have 
	\begin{align*}
		m\lt(\set{\om\in\Om: C_\ep(\om)>A_\dl}\rt)<\frac{\dl}{32(\rho_0+1)}.
	\end{align*}
	Thus the ergodic theorem gives that for $m$-a.e. $\om\in\Om$ we have 
	\begin{align}\label{eq: Buzzi eq 4.1}
		\lim_{n\to\pm\infty}\frac{1}{|n|}\#\set{0\leq j<|n|: C_\ep(\sg^{\pm j}(\om))>A_\dl \text{ or } n_\dl(\sg^{\pm j}(\om))>n_*}
		<
		\frac{\dl}{8(\rho_0+1)}
		<
		1.
	\end{align}
	Then for $m$-a.e. $\om\in\Om$ and all large $2|k|$ (depending on $\om$) there exists an integer  
	\begin{align}\label{eq: choice of m}
		p\in \lt(k-\frac{\dl|k|}{2(\rho_0+1)},\, k-\frac{\dl|k|}{4(\rho_0+1)}\rt)
	\end{align}
	such that 
	\begin{align}\label{eq: var q and n_dl good bound}
		\var(q_{\sg^p(\om)})\leq C_\ep(\sg^{p}(\om))\leq A_\dl
		\quad\text{ and }\quad
		n_\dl(\sg^{p}(\om))\leq n_*.
	\end{align}
		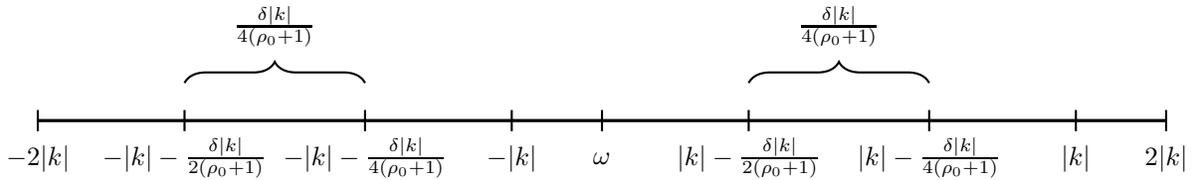
\begin{figure}[H]
		\centering	
		\begin{tikzpicture}[y=1cm, x=1.5cm, thick, font=\footnotesize]    
		\draw[line width=1.2pt,  >=latex'](0,0) -- coordinate (x axis) (10,0);       
		
		\draw (0,-4pt)   -- (0,4pt) {};
		\node at (0, -.5) {$-2|k|$};
		
		\draw (4.2,-4pt) -- (4.2,4pt) {};
		\node at (4.2,-.5) {$-|k|$};
		
		\draw (5,-4pt)  -- (5,4pt)  {};
		\node at (5,-.5) {$\om$};
		
		\draw (9.2,-4pt)  -- (9.2,4pt)  {};
		\node at (9.2,-.5) {$|k|$};
		
		\draw (10,-4pt)  -- (10,4pt)  {};
		\node at (10,-.5) {$2|k|$};
		
		\draw (1.3,-4pt)  -- (1.3,4pt)  {};
		\node at (1.3,-.5) {$-|k|-\frac{\dl|k|}{2(\rho_0+1)}$};
		
		\draw (2.9,-4pt)  -- (2.9,4pt)  {};
		\node at (2.9,-.5) {$-|k|-\frac{\dl|k|}{4(\rho_0+1)}$};
		
		\draw (6.3,-4pt)  -- (6.3,4pt)  {};
		\node at (6.3,-.5) {$|k|-\frac{\dl|k|}{2(\rho_0+1)}$};
		
		\draw (7.9,-4pt)  -- (7.9,4pt)  {};
		\node at (7.9,-.5) {$|k|-\frac{\dl|k|}{4(\rho_0+1)}$};
		
		\draw[decorate,decoration={brace,amplitude=8pt}] 
		(1.3,.5)  -- (2.9,.5) ; 
		\node at (2.1,1.25){$\frac{\dl|k|}{4(\rho_0+1)}$};
		
		\draw[decorate,decoration={brace,amplitude=8pt}] 
		(6.3,.5)  -- (7.9,.5) ; 
		\node at (7.1,1.25){$\frac{\dl|k|}{4(\rho_0+1)}$};
		\end{tikzpicture}
		\caption{The $2|k|$ length forward and backward $\sg$-orbits originating at $\om$.}
		\label{figure: 2k orbit}
	\end{figure}
	Without loss of generality we assume that 
	\begin{align}\label{eq: k large}
		|k|\geq \frac{4n_*(\rho_0+1)}{\dl}.
	\end{align}
	Now we note that since $k>p$, we have that 
	\begin{align}\label{eq: k-p positive}
		0<k-p<|k|
	\end{align} 
	for each $k\in\ZZ$ sufficiently large. 
	Furthermore, using \eqref{eq: choice of m} and \eqref{eq: k large} we have that 
	\begin{align}\label{eq: k-m is good}
		n_*\leq \frac{\dl|k|}{4(\rho_0+1)}<  k-p < \frac{\dl|k|}{2(\rho_0+1)} <|k|\frac{\dl}{2}.
	\end{align} 
	Now using \eqref{eq: k-m is good} and the fact that $x<e^x$ for $x>0$ gives that 
	\begin{align}\label{eq: Adl +k-p est}
		A_\dl+k-p
		< 
		A_\dl + |k|\frac{\dl}{2}
		<
		A_\dl\cdot |k|\frac{\dl}{2}
		<
		A_\dl e^{|k|\frac{\dl}{2}}.
	\end{align}
	Using, in order, \eqref{eq: q_om fixed by norm tr op} and \eqref{eq: BV norm bound using L} (with $n=1$) repeatedly followed by \eqref{eq: var q and n_dl good bound} and \eqref{eq: Adl +k-p est}, we see that, for all large $|k|$, say $|k|\geq k_\dl(\om)\geq 1$ with $k_\dl:\Om\to\NN$ measurable and satisfying \eqref{eq: choice of m} and \eqref{eq: k large}, \eqref{eq: sum of log L} gives that 
	\begin{align}
		\norm{q_{\sg^k(\om)}}_{\BV}
		&=
		\norm{\~\cL_{\sg^p(\om)}^{k-p}q_{\sg^p(\om)}}_{\BV}
		\nonumber\\
		&=
		\var\lt(\~\cL_{\sg^{k-1}}\circ\dots\circ\~\cL_{\sg^p(\om)}(q_{\sg^p(\om)})\rt)
		+
		\norm{\~\cL_{\sg^{k-1}}\circ\dots\circ\~\cL_{\sg^p(\om)}(q_{\sg^p(\om)})}_{\infty}
		\nonumber\\
		&\leq 
		L_{\sg^{k-1}(\om)}\lt(\var\lt(\~\cL_{\sg^{k-2}}\circ\dots\circ\~\cL_{\sg^p(\om)}(q_{\sg^p(\om)})\rt)+1\rt)
		\nonumber\\
		&\leq 
		\prod_{j=0}^{k-p-1} L_{\sg^{p+j}(\om)}\var(q_{\sg^p(\om)})
		+
		\sum_{j=0}^{k-p-1}\prod_{i=j}^{k-p-1} L_{\sg^{p+i}(\om)}
		\label{eq: unrefined up bdd q BV norm}
		\\
		&\leq 
		\lt(A_\dl+k-p\rt)\prod_{j=0}^{k-p-1} L_{\sg^{p+j}(\om)}
		\nonumber\\
		&< 
		A_\dl e^{|k|\frac{\dl}{2}}\prod_{j=0}^{k-p-1} L_{\sg^{p+j}(\om)}
		\leq 
		A_\dl e^{|k|\frac{\dl}{2}}e^{(k-p)(\rho_0+1)}
		\leq 
		A_\dl e^{|k|\frac{\dl}{2}}e^{|k|\frac{\dl}{2}}
		=A_\dl e^{|k|\dl}.
		\label{eq: var q subexp for large k}
	\end{align}
	To obtain the inequality \eqref{eq: var q subexp for large k} for all $k\in\ZZ$, and not just $|k|\geq k_\dl(\om)$, we use \eqref{eq: measurable up bdd for q} to define the measurable constant 
	\begin{align*}
		C(\om,\dl):=\max_{|k|<k_\dl(\om)}\set{2C_\ep(\sg^k(\om))+1}.
	\end{align*}
	Thus combining with \eqref{eq: var q subexp for large k}, we have 
	\begin{align*}
		\norm{q_{\sg^k(\om)}}_{\BV}
		\leq 
		C(\om,\dl)e^{|k|\dl}
	\end{align*}
	for all $k\in\ZZ$, which finishes the proof.
	
\end{proof}

\begin{definition}\label{defn: definition of cD}
	Given $f\in\prod_{\om\in\Om}\BV(I)$, $\om\in\Om$, and $n\in\ZZ$ we will say that \eqref{cond cD1}, respectively \eqref{cond cD2}, hold for the triple $(f,\om,n)$ if for each $\ep>0$ there exists $V_{f,\ep}:\Om\to(0,\infty)$ such that the following respective conditions hold:
	\begin{flalign}
	& \var(f_{\sg^n(\om)})\leq V_{f,\ep}(\om)e^{\ep|n|},
	\label{cond cD1}\tag{\cD1}
	&\\
	&\nu_{\sg^n(\om)}(|f_{\sg^n(\om)}|)\geq V_{f,\ep}^{-1}(\om)e^{-\ep |n|}.
	\label{cond cD2}\tag{\cD2}
	\end{flalign}
	We let $\cD$ denote the collections of functions $f\in\prod_{\om\in\Om}\BV(I)$ such that \eqref{cond cD1} and \eqref{cond cD2} hold for the triple $(f,\om,n)$ for $m$-a.e. $\om\in\Om$ and each $n\in\ZZ$. 
	Let $\cD^+\sub\cD$ denote the collection of all functions $f\in\cD$ such that $f_\om\geq 0$ for each $\om\in\Om$.
\end{definition}
\begin{remark}\label{rem: remark on def of cD}
	 Note that $\cD^+$ is nonempty as it contains the functions $\ind$ and $q$ by Lemma~\ref{lem: BV norm q om growth bounds}.
	Furthermore, note that \eqref{cond cD1} and \eqref{cond cD2} are clearly satisfied for all $f\in\BV_\Om^l(I)$, which was defined in \eqref{eq: def of B_Om^l}. 
\end{remark}

The following lemma provides conditions ensuring that \eqref{cond cD1} and \eqref{cond cD2} hold for products and quotients.  Its proof consists entirely of elementary calculations, and is therefore left to the reader.
\begin{lemma}\label{lem: props of cD}
	If $f\in\cD$ then 
	\begin{flalign}
		& \lt(\frac{f_\om}{\nu_\om(f_\om)}\rt)_{\om\in\Om}\in\cD.&
		\tag{a} \label{lem: props of cD item c}
	\end{flalign}
	Furthermore, if $h\in\prod_{\om\in\Om}BV^+(I)$ and there exists a sequence $(n_k)_{k\in\NN}$ 
	such that 
	$$
		\inf h_{\sg^{n_k}(\om)}>0 
		\quad\text{ and }\quad
		\lim_{k\to\infty} \frac{1}{n_k}\log \inf h_{\sg^{n_k}(\om)}
		=0 
	$$ 
	and \eqref{cond cD1} holds for $(h,\om,n_k)$ for each $n_k$, then we have the following:	
	\begin{flalign*}
		& \eqref{cond cD1}\text{ and }\eqref{cond cD2} \text{ hold for } (f h,\om,n_k) \text{ for each } n_k,
		\tag{b}\label{lem: props of cD item a}
		&\\
		& \eqref{cond cD1}\text{ and }\eqref{cond cD2} \text{ hold for } \lt(\frac{f}{h}, \om,n_k\rt) \text{ for each } n_k.
		\tag{c} \label{lem: props of cD item b}
	\end{flalign*}
\end{lemma}

We now wish to describe the rate at which the convergence to the invariant density $q_\om$ from Corollary~\ref{cor: exist of unique dens q} is achieved.

\begin{corollary}\label{cor: exp conv to dens in sup norm}
There exists $0<\kp<1$ such that if $f\in\prod_{\om\in\Om}BV^+(I)$, with $\nu_{\om}(f_{\om})=1$ for $m$-a.e. $\om\in\Om$ then there exists $A_f(\om):\Om\to(0,\infty)$ such that for all $n\in\NN$ and all $|p|\leq n$, with \eqref{cond cD1} holding for $(f,\sg^p(\om), n)$, we have 
	\begin{align}\label{eq: cor state exp conv to density}
		\norm{\~\cL_{\sg^{p}(\om)}^n f_{\sg^{p}(\om)} - q_{\sg^{p+n}(\om)} }_\infty 
		\leq A_f(\om) \norm{f_{\sg^p(\om)}}_\BV \kp^n.
	\end{align} 
\end{corollary}
\begin{proof}
We begin by noting that Lemma~\ref{lem: buzzi LY1}, together with the fact that $\norm{f_\om}_\BV,\, \norm{q_\om}_\infty\geq 1$ for $m$-a.e. $\om\in\Om$, implies that for each $n\in\NN$ we have 
\begin{align}
	\norm{\~\cL_{\sg^{p}(\om)}^n f_{\sg^{p}(\om)} - q_{\sg^{p+n}(\om)}}_\infty
	&\leq 
	\norm{\~\cL_{\sg^{p}(\om)}^n f_{\sg^{p}(\om)}}_\infty +\norm{q_{\sg^{p+n}(\om)}}_\infty
	\nonumber\\
	&\leq 
	2L_{\sg^{p}(\om)}^{n}\norm{f_{\sg^p(\om)}}_\BV \norm{q_{\sg^{p+n}(\om)}}_\infty	
	\nonumber\\
	&\leq 
	2L_{\sg^{-n}(\om)}^{3n}\norm{f_{\sg^p(\om)}}_\BV \norm{q_{\sg^{p+n}(\om)}}_\infty.
	\label{eq: Cor 8.6 ineq for all n}
\end{align}
Now, fix $\ep<\ep_5$ small as in Lemma~\ref{lem: exp conv in C+ cone}. Lemma~\ref{lem: BV norm q om growth bounds} implies that 
\begin{align*}
	\var(q_{\sg^{p+n}(\om)})\leq C(\om,\ep)e^{\ep |p+n|}.
\end{align*}
for all $n\geq 0$. 
Let 
\begin{align}\label{eq: defn of V(om)}
	V(\om)=\max\set{C(\om,\ep), V_{f,\ep}(\om)}.
\end{align}
Then Lemma~\ref{lem: exp conv in C+ cone} gives that 
\begin{align*}
	\Ta_{\sg^{p+n}(\om),+}\lt(\~\cL_{\sg^{p}(\om)}^nf_{\sg^{p}(\om)}, q_{\sg^{p+n}(\om)}\rt)
	=
	\Ta_{\sg^{p+n}(\om),+}\lt(\~\cL_{\sg^{p}(\om)}^nf_{\sg^{p}(\om)}, \~\cL_{\sg^{p}(\om)}^nq_{\sg^{p}(\om)}\rt)
	\leq \Dl\vta^n
\end{align*}
for all $n\geq N_3(\om)$, which was defined in \eqref{eq: def of N_3} in the proof of Lemma~\ref{lem: exp conv in C+ cone}.
Applying Lemma~\ref{lem: birkhoff cone contraction} with $\vrho=\nu_{\sg^{p+n}(\om)}$ and $\norm{\spot}=\norm{\spot}_\infty$, together with Lemmas~\ref{lem: exp conv in C+ cone} and \ref{lem: BV norm q om growth bounds} then gives 
\begin{align}
	\norm{\~\cL_{\sg^{p}(\om)}^n f_{\sg^{p}(\om)} - q_{\sg^{p+n}(\om)}}_\infty
	&\leq 
	\norm{q_{\sg^{p+n}(\om)}}_\infty 
	\lt(\exp\lt(\Ta_{\sg^{p+n}(\om),+}\lt(\~\cL_{\sg^{p}(\om)}^nf_{\sg^{p}(\om)}, q_{\sg^{p+n}(\om)}\rt)\rt)-1\rt)
	\nonumber\\
	&\leq 
	\norm{q_{\sg^{p+n}(\om)}}_\infty\lt(e^{\Dl\vta^n}-1\rt)
	\leq \norm{q_{\sg^{p+n}(\om)}}_\infty\~\kp^n
	\label{eq: Cor 8.6 exp ineq for large n}
\end{align}
for some $\~\kp\in(0,1)$. Combining \eqref{eq: Cor 8.6 ineq for all n} with \eqref{eq: Cor 8.6 exp ineq for large n} we see that 
\begin{align}
	\norm{\~\cL_{\sg^{p}(\om)}^n f_{\sg^{p}(\om)} - q_{\sg^{p+n}(\om)}}_\infty
	&\leq
	2\max_{0\leq j\leq N_3(\om)}\set{L_{\sg^{p}(\om)}^{j} \~\kp^{-j}}
	\norm{f_{\sg^p(\om)}}_\BV \norm{q_{\sg^{p+n}(\om)}}_\infty \~\kp^n
	\nonumber\\
	&=2L_{\sg^{p}(\om)}^{N_3(\om)} \~\kp^{-N_3(\om)}
	\norm{f_{\sg^p(\om)}}_\BV \norm{q_{\sg^{p+n}(\om)}}_\infty \~\kp^n
	\nonumber\\
	&\leq 
	2L_{\sg^{-N_3(\om)}(\om)}^{3N_3(\om)} \~\kp^{-N_3(\om)}
	\norm{f_{\sg^p(\om)}}_\BV \norm{q_{\sg^{p+n}(\om)}}_\infty \~\kp^n
	\label{eq: Cor 8.6 exp dec ineq with sup q}
\end{align}
for all $n\in\NN$. Now in light of Lemma~\ref{lem: BV norm q om growth bounds} we may rewrite \eqref{eq: Cor 8.6 exp dec ineq with sup q} so that we have 
\begin{align*}
\norm{\~\cL_{\sg^{p}(\om)}^n f_{\sg^{p}(\om)} - q_{\sg^{p+n}(\om)}}_\infty
&\leq 
A_f(\om) \norm{f_{\sg^p(\om)}}_\BV \kp^n
\end{align*}
for all $n\in\NN$, where 
\begin{align*}
	A_f(\om):= 	2C(\om,\dl)L_{\sg^{-N_3(\om)}(\om)}^{3N_3(\om)} \~\kp^{-N_3(\om)}
\end{align*} 
and $\kp:=\~\kp e^{2\dl}$ for some $\dl>0$ sufficiently small.
\end{proof}

The following proposition serves two purposes: we show that the invariant density $q_\om$ is strictly positive and finite for $m$-a.e. $\om\in\Om$, and we provide measurable upper and lower bounds for $q_\om$ despite the fact that (at this point) we do not know whether the maps $\om\mapsto \inf q_\om$ or $\om\mapsto \norm{q_\om}_\infty$ are measurable as we also do not currently know that the maps $\om\mapsto \nu_\om$ and $\om\mapsto\lm_\om$ are measurable. 
\begin{proposition}\label{prop: upper and lower bound for density}
	There exist measurable functions $u,U:\Om\to(0,\infty)$ such that for $m$-a.e. $\om\in\Om$ we have
	\begin{align*}
		 u(\om)\leq q_\om \leq U(\om). 
	\end{align*} 
\end{proposition}
\begin{proof}
	For the upper bound we simply set $U(\om)=C(\om,\dl)$ as in Lemma~\ref{lem: BV norm q om growth bounds}. 
	Now, for the lower bound we begin by noting that \eqref{eq: var q leq C_ep}, our choice of $B_*$ in the proof of Lemma~\ref{lem: constr of Om_G}, and the fact that we have $a_*>6B_*$ gives that $q_\om\in\sC_{\om,a_*}$ for each $\om\in\Om_G$.
	Thus, for $\om\in\Om_G$ \eqref{inf L_om^R bound for finite diam} implies that 
	\begin{align}\label{eq: q_om > 0 for good om}
	\inf q_{\sg^{R_*}(\om)}
	=
	\inf \~\cL_\om^{R_*} q_\om 
	\geq 
	\frac{1}{2}\inf\~\cL_\om^{R_*}\ind_{U_{q_\om}}
	\geq 
	\frac{\al_*}{2\lm_\om^{R_*}}
	\geq 
	\frac{\al_*}{2C_*}
	=:t_*
	>0, 
	\end{align}
	where $U_{q_\om}\in\cU_\om$ is coming from Lemma~\ref{lem: LSV lemma 4.6}, and thus 
	\begin{align}\label{eq: def of Om_+}
	\Om_+:=\sg^{R_*}(\Om_G)\sub\cH:=\set{\om\in\Om: \inf q_\om>0}.
	\end{align} 
	In light of \eqref{eq: q_om > 0 for good om}, for $\om\in\Om_+$ we have 
	\begin{align*}
		\inf q_\om\geq t_*,
	\end{align*}
	and \eqref{eq: def of Om_+} together with Lemma~\ref{lem: constr of Om_G} imply that 
	$$
	m(\Om_+)\geq m(\Om_G)\geq 1-\frac{\ep}{4}.
	$$
	Note that for any $\om\in\cH$ (up to a set of measure zero) and any $n\in\NN$ we have 
	\begin{align}
		\inf q_\om
		&=
		\inf\~\cL_{\sg^{-n}(\om)}^n q_{\sg^{-n}(\om)}
		\nonumber\\
		&\geq
		\lt(\lm_{\sg^{-n}(\om)}^n\rt)^{-1}\inf q_{\sg^{-n}(\om)} \inf\cL_{\sg^{-n}(\om)}^n\ind_{\sg^{-n}(\om)}
		\nonumber\\
		&\geq 
		\inf q_{\sg^{-n}(\om)} \frac{\inf\cL_{\sg^{-n}(\om)}^n\ind_{\sg^{-n}(\om)}}{\norm{\cL_{\sg^{-n}(\om)}^n\ind_{\sg^{-n}(\om)}}_\infty}
		>0
		\label{eq: lower bound for inf q_om}
	\end{align}
	since Lemma~\ref{lem: LY setup} implies that the fraction on the right-hand side is positive for $m$-a.e. $\om\in\Om$. 
	Now since $\Om_+$ has positive measure, the ergodic theorem ensures that we may define the measurable number 
	\begin{align}
		\~n_\om:=\min\set{n\in\NN: \sg^{-n}(\om)\in\Om_+}.
	\end{align}
	Then, using \eqref{eq: lower bound for inf q_om}, we have 
	\begin{align*}
	\inf q_{\om}
	&=
	\inf \~\cL_{\sg^{-\~n_\om}(\om)}^{\~n_\om} q_{\sg^{-\~n_\om}(\om)} 
	\geq
	\lt(\lm_{\sg^{-\~n_\om}(\om)}^{\~n_\om}\rt)^{-1}\inf q_{\sg^{-\~n_\om}(\om)} \inf\cL_{\sg^{-\~n_\om}(\om)}^{\~n_\om}\ind_{\sg^{-\~n_\om}(\om)} 
	\\
	&\geq
	t_*\frac{\inf\cL_{\sg^{-\~n_\om}(\om)}^{\~n_\om}\ind_{\sg^{-\~n_\om}(\om)}}{\norm{\cL_{\sg^{-\~n_\om}(\om)}^{\~n_\om}\ind_{\sg^{-\~n_\om}(\om)}}_\infty}	
	>0.
	\end{align*}
	Taking 
	$$
		u(\om)=t_*\cdot\frac{\inf\cL_{\sg^{-\~n_\om}(\om)}^{\~n_\om}\ind_{\sg^{-\~n_\om}(\om)}}{\norm{\cL_{\sg^{-\~n_\om}(\om)}^{\~n_\om}\ind_{\sg^{-\~n_\om}(\om)}}_\infty}
	$$ 
	finishes the proof.	
\end{proof}
For each $\om\in\Om$ we may now define the measure $\mu_\om\in\cP(I)$ by 
\begin{align}\label{eq: def of mu_om}
	\mu_\om(f):=\int_{X_\om} fq_\om \,d\nu_\om,  \qquad f\in L^1_{\nu_\om}(I).
\end{align}
Proposition~\ref{prop: existence of conformal family} and Proposition~\ref{prop: upper and lower bound for density} together show that, for $m$-a.e. $\om\in\Om$, $\mu_\om$ is non-atomic, positive on non-degenerate intervals, and absolutely continuous with respect to $\nu_\om$. Furthermore, in view of Proposition~\ref{prop: upper and lower bound for density}, for $m$-a.e. $\om\in\Om$, we may now define the fully normalized transfer operator $\hcL_\om:\BV(X_\om)\to\BV(X_{\sg(\om)})$ by 
\begin{align}\label{eq: def fully norm tr op}
	\hcL_\om f:= \frac{1}{q_{\sg(\om)}}\~\cL_\om(fq_\om) = \frac{1}{\lm_\om q_{\sg(\om)}}\cL_\om (fq_\om), 
	\qquad f\in \BV(X_\om).
\end{align}
Alternatively, we may think of the fully normalized transfer operator $\hcL_\om$ as the usual transfer operator taken with the potential $\hat\phi_\om$ which is given by
	\begin{align}\label{eq: def of hat phi_om}
		\hat\phi_\om:=\phi_\om+\log q_\om-\log q_{\sg(\om)}\circ T_\om - \log \lm_\om,
	\end{align}
	i.e. $\hcL_\om=\cL_{\hat\phi, \om}$. This gives rise to the weight functions $\hat g_\om$ given by 
	\begin{align}\label{eq: def of hat g_om}
		\hat g_\om:=e^{\hat\phi_\om}=\frac{e^{\phi_\om}\cdot q_\om}{\lm_\om\cdot q_{\sg(\om)}\circ T_\om},
	\end{align}
	and the alternate formulation of $\hcL_\om$:
	\begin{align}\label{eq: alt defhat L}
		\hcL_\om(f)(x)=\sum_{y\in T_\om^{-1}(x)}\hat g_\om(y)f(y)
		\qquad f\in \BV(X_\om), \, x\in X_{\sg(\om)}.
	\end{align}
As an immediate consequence of \eqref{eq: q_om fixed by norm tr op} and \eqref{eq: def fully norm tr op}, we get that 
\begin{align}\label{eq: fully norm op fix ind}
	\hcL_\om\ind_\om =\ind_{\sg(\om)}.
\end{align}
Furthermore, the operator $\hcL_\om$ satisfies the identity
\begin{align}\label{eq: tr op identity}
	 \hcL_\om\lt(h\cdot(f\circ T_\om^n)\rt)
	=
	 f\cdot \hcL_\om^n h
\end{align}
for all $h\in\BV(X_\om)$, $f\in\BV(X_{\sg^n(\om)})$, and all $n\in\NN$.
We end this section with the following proposition which shows that the family $(\mu_\om)_{\om\in\Om}$ of measures is $T$-invariant.
\begin{proposition}\label{prop: mu_om T invar}
	The family $(\mu_\om)_{\om\in\Om}$ defined by \eqref{eq: def of mu_om} is $T$-invariant in the sense that 
	\begin{align}\label{eq: mu_om T invar}
		\int_{X_\om} f\circ T_\om \, d\mu_\om 
		=
		\int_{X_{\sg(\om)}} f \, d\mu_{\sg(\om)}
	\end{align}
	for $f\in L^1_{\mu_{\sg(\om)}}(I)=L^1_{\nu_{\sg(\om)}}(I)$.	
\end{proposition} 
\begin{proof}
In light of \eqref{eq: equivariance prop of nu norm op}, we note that for $f\in L^1_{\mu_{\sg(\om)}}(I)$ we have 
\begin{align}\label{eq: equivariance prop of full norm op}
	\int_{X_{\sg(\om)}} \hcL_\om f \,d\mu_{\sg(\om)}
	=
	\int_{X_{\sg(\om)}} \~\cL_\om(fq_\om)\,d\nu_{\sg(\om)}
	=
	\int_{X_\om} fq_\om\,d\nu_\om
	=
	\int_{X_\om} f \,d\mu_\om.
\end{align}
Now, if $h\cdot(f\circ T_\om)\in L^1_{\mu_\om}(I)$, \eqref{eq: equivariance prop of full norm op} and \eqref{eq: tr op identity} give that
\begin{align*}
	\int_{X_\om} h\cdot(f\circ T_\om) \, d\mu_\om
	=
	\int_{X_{\sg(\om)}} \hcL_\om\lt(h\cdot(f\circ T_\om)\rt) \, d\mu_{\sg(\om)}
	=
	\int_{X_{\sg(\om)}} f\cdot \hcL_\om h \, d\mu_{\sg(\om)}.
\end{align*}
In view of \eqref{eq: fully norm op fix ind}, taking $h=\ind$ finishes the proof.
\end{proof}

\section{Random Measures}\label{sec:randommeasures}
In this section we show that the families of measures $(\nu_\om)_{\om,\in\Om}$ and $(\mu_\om)_{\om,\in\Om}$ are in fact random probability measures as defined in Definition~\ref{def: random prob measures} as well as show that the invariant density $q$ is unique. 
First we establish the existence of a measurable choice of an increasing sequence $(n_k(\om))_{k\geq 1}$ of positive integers along which we obtain the exponential convergence of $\hat\cL_\om^{n_k} f_\om$. 
\begin{lemma}\label{lem: exp conv to mu along subseq}
	For $m$-a.e. $\om\in\Om$ there exists a measurable choice of an increasing sequence $(n_k(\om))_{k\geq 1}$ of positive integers such that if $f\in\prod_{\om\in\Om}\BV(I)$ with \eqref{cond cD1} and \eqref{cond cD2} holding for $(f,\om,n_k(\om))$ for each $k\in\NN$ and $\om\in\Om$, then there exists 
	$\hat A_f:\Om\to(0,\infty)$ such that
	\begin{align*}
	\norm{\hat \cL_\om^{n_k} f_\om -\mu_\om(f_\om)\ind_{\sg^{n_k}(\om)}}_\infty
	\leq 
	\hat A_f(\om)\norm{f_\om}_\BV\kp^{n_k}.
	\end{align*}
\end{lemma}
\begin{proof}
	Suppose first that $f\in\cD^+$. In light of Lemma~\ref{lem: props of cD} item \eqref{lem: props of cD item c}, we consider Corollary~\ref{cor: exp conv to dens in sup norm} (with $p=0$) for the function $\ol f=(f_\om/\nu_\om(f_\om))_{\om\in\Om}\in\cD^+$, and thus, for each $n\geq N_3(\om)$, we have that 
	\begin{align}
		 \norm{\~\cL_\om^{n} \lt(\frac{f_\om}{\nu_\om(f_\om)}\rt) -q_{\sg^{n}(\om)}}_\infty
		 &\leq 
		 A_{\ol f}(\om)\norm{\frac{f_\om}{\nu_\om(f_\om)}}_\BV\kp^{n}.
		 \label{eq: subseq conv to 1 est 0*}		 
	\end{align}
	Multiplying both sides of \eqref{eq: subseq conv to 1 est 0*} by $\nu_\om(f_\om)$ yields
	\begin{align}\label{eq: subseq conv to 1 est 0}
		\norm{\~\cL_\om^{n} (f_\om) -\nu_\om(f_\om)q_{\sg^{n}(\om)}}_\infty
		\leq 
		A_{\ol f}(\om)\norm{f_\om}_\BV\kp^{n}.
	\end{align}
	Now let $u(\om)$ be the measurable lower bound for $q_\om$ coming from Proposition~\ref{prop: upper and lower bound for density}, and choose $u_*>0$ sufficiently small such that 
	\begin{align*}
		m\lt(\Om_*:=\set{\om\in\Om: u_*\leq u(\om)}\rt)>0.
	\end{align*}
	The ergodic theorem then gives that for $m$-a.e. $\om\in\Om$ there exists a measurable choice of an increasing sequence of positive integers $(n_k)_{k\geq 1}=(n_k(\om))_{k\geq 1}$ such that $\sg^{n_k}(\om)\in\Om_*$ for each $k\geq 1$. 
	It follows from Proposition~\ref{prop: upper and lower bound for density} that
	\begin{align}\label{eq: fixed u_* and U^* bds for q}
		0<u_*\leq u(\sg^{n_k}(\om))\leq q_{\sg^{n_k}(\om)}
	\end{align}
	for each $k\geq 1$. Then in light of Lemma~\ref{lem: BV norm q om growth bounds} and \eqref{eq: fixed u_* and U^* bds for q}, Lemma~\ref{lem: props of cD} item \eqref{lem: props of cD item a} implies that, for $m$-a.e. $\om\in\Om$, \eqref{cond cD1} and \eqref{cond cD2} hold for $(fq,\om, n_k)$  for  each  $k\geq 1$. 
	Thus, applying \eqref{eq: subseq conv to 1 est 0}, with $n=n_k$ and the function $f_\om q_\om$ in place of $f_\om$, gives 
	\begin{align}\label{eq: subseq conv to 1 est 1}
		\norm{\~\cL_\om^{n_k} (f_\om q_\om) -\mu_\om(f_\om)q_{\sg^{n_k}(\om)}}_\infty
		&\leq A_{\ol{fq}}(\om)\norm{f_\om}_\BV\norm{q_\om}_\BV \kp^{n_k}.		
	\end{align}
	It then follows from \eqref{eq: subseq conv to 1 est 1} and \eqref{eq: fixed u_* and U^* bds for q} that we have
	\begin{align}
		\norm{\hat \cL_\om^{n_k} f_\om -\mu_\om(f_\om)\ind_{\sg^{n_k}(\om)}}_\infty
		&
		\leq
		\frac{1}{\inf q_{\sg^{n_k}(\om)}}\cdot 
		\norm{\~\cL_\om^{n_k} (f_\om q_\om) -\mu_\om(f_\om)q_{\sg^{n_k}(\om)}}_\infty 
		\nonumber\\
		&
		\leq 
		A_{\ol{fq}}(\om)u_*^{-1} \norm{f_\om}_\BV\norm{q_\om}_\BV \kp^{n_k}.\label{eq: subseq conv to 1 est 2}
	\end{align}
	Now, for a function $f\in\cD$, we write $f=f^+- f^-$, with $f^+,f^-\geq 0$, and $f_\om=f_\om^+-f_\om^-$ for each $\om\in\Om$. Rerunning the previous argument with $f=f^+-f^-$ to \eqref{eq: subseq conv to 1 est 2} and applying the triangle inequality we get
	\begin{align*}
		\norm{\hat \cL_\om^{n_k} f_\om -\mu_\om(f_\om)\ind_{\sg^{n_k}(\om)}}_\infty
		&
		\leq 
		\hat A_f(\om)\norm{f_\om}_\BV\kp^{n_k}
	\end{align*}
	for each $k\in\NN$, where 
	\begin{align}\label{eq: def of hat A_f}
		\hat A_f(\om):=(A_{\ol{f^+q}}(\om)+A_{\ol{f^-q}}(\om))\norm{q_\om}_\BV u_*^{-1}.
	\end{align}
\end{proof}

In the following lemma we give a characterization of the measure $\nu_\om$ as a limit of measurable functions along the measurable sequence $n_k$ produced in the previous lemma. 
\begin{lemma}\label{lem: limit equality for nu}
	For each $f\in\cD$ and $m$-a.e.	$\om\in\Om$ we have
	\begin{align}\label{eq: subseq limit char of nu}
	\nu_\om(f_\om)
	=
	\lim_{k\to \infty} \frac{\norm{\cL_\om^{n_k} f_\om}_\infty}{\norm{\cL_\om^{n_k} \ind_\om}_\infty},
	\end{align}	
	where $(n_k(\om))_{k\geq 1}$ is the measurable sequence of positive integers coming from Lemma~\ref{lem: exp conv to mu along subseq}.
\end{lemma}
\begin{proof}
	Let $(n_k)_{k\geq 1}=(n_k(\om))_{k\geq 1}$ denote the measurable sequence of positive integers coming from Lemma~\ref{lem: exp conv to mu along subseq}. In light of Lemma~\ref{lem: BV norm q om growth bounds},  Lemma~\ref{lem: props of cD} item \eqref{lem: props of cD item b} implies that, for $m$-a.e. $\om\in\Om$, \eqref{cond cD1} holds for $(f/q,\om, n_k)$  for each $k\geq 1$. Furthermore, the uniform lower bound for $q_\om$ given in \eqref{eq: fixed u_* and U^* bds for q} together with Lemma~\ref{lem: props of cD} item \eqref{lem: props of cD item b} implies that, for $m$-a.e. $\om\in\Om$, \eqref{cond cD2} holds for $(f/q,\om, n_k)$  for each $k\geq 1$. Thus  Lemma~\ref{lem: exp conv to mu along subseq} gives that  
	\begin{align}\label{eq: nu limit nu estimate}
		\absval{\norm{\hcL_\om^{n_k} \lt(\frac{f_\om}{q_\om}\rt)}_\infty - \nu_\om(f_\om)}
		&\leq  
		\hat A_{\frac{f}{q}}(\om)\norm{\frac{f_\om}{q_\om}}_\BV\kp^{n_k}
	\end{align}
	for $m$-a.e. $\om\in\Om$ and each $k\in\NN$. Hence, taking limits, we get		
	\begin{align*}
		\lim_{k\to \infty} \frac{\norm{\cL_\om^{n_k} f_\om}_\infty}{\norm{\cL_\om^{n_k} \ind_\om}_\infty}
		=
		\lim_{k\to \infty} \frac{\norm{\hcL_\om^{n_k} \lt(\frac{f_\om}{q_\om}\rt)}_\infty}{\norm{\hcL_\om^{n_k} \lt(\frac{\ind_\om}{q_\om}\rt)}_\infty}
		=
		\frac{\nu_\om(f_\om)}{\nu_\om(\ind_\om)}
		=
		\nu_\om(f_\om).
	\end{align*}
\end{proof}
We now arrive at the main result of this section.

\begin{proposition}\label{prop: measurability}
	The map $\om\mapsto \lm_\om$ is  measurable with $\log\lm_\om\in L^1_m(\Om)$ and the families
	$(\nu_\om)_{\om\in\Om},(\mu_\om)_{\om\in\Om}$ are random probability measures giving rise to $\nu,\mu\in\cP_m(\Om\times I)$ defined by 
	\begin{align}\label{def: global conformal and inv random measures}
	\nu(f)=\int_\Om\int_I f_\om \, d\nu_\om\, dm(\om)
	\quad \text{ and } \quad
	\mu(f)=\int_\Om\int_I f_\om \, d\mu_\om\, dm(\om)
	\end{align}
	for $f\in L^1_\nu(\Om\times I)$.  Furthermore, the maps $\om \mapsto \|q_\om\|_\infty$ and  $\om \mapsto \inf q_\om$ are measurable.
\end{proposition}
\begin{proof}
We prove the proposition in four steps:
	\begin{itemize}
		\item[\mylabel{i}{it:nu-mble}]
		We first show that $(\nu_\om)_{\om\in\Om}$ is a random probability measure as in Definition~\ref{def: random prob measures}. Indeed, for every interval $J\subset I$,
		the function $\om \mapsto \nu_\om(J)$ is measurable, as it is the limit of measurable functions by \eqref{eq: subseq limit char of nu} applied to $f_\om=\ind_J$. 
		Since $\sB$ is generated by intervals, $\om \mapsto \nu_\om(B)$ is measurable for every $B\in \sB$. Furthermore, $\nu_\om$ is a Borel probability measure for $m$-a.e. $\om \in \Om$ from Proposition~\ref{prop: existence of conformal family}.
		
		\
		
		\item[\mylabel{ii}{it:lm-mble}] Given that $\lm_\om:=\nu_{\sg(\om)}(\cL_\om\ind_\om)$, \eqref{it:nu-mble} and Proposition~\ref{prop: random measure equiv}
		immediately imply that $\om\mapsto\lm_\om$ is measurable.
		The log-integrability claim then follows from  Lemma~\ref{lem: log integrability}, and the fact that 
		$$
		\inf\cL_\om\ind_\om \leq \lm_\om\leq \norm{\cL_\om\ind_\om}_\infty.
		$$
		
		\
		
		\item[\mylabel{iii}{it:q-mble}]
		The measurability of $\lm_\om$ together with \eqref{cond M2} and
		the fact that 
		$$
			q_\om= \lim_{n\to\infty} (\lm_\om^n)^{-1}\cL_{\sg^{-n}(\om)}^n \ind_{\sg^{-n}(\om)}
		$$
		(see \eqref{eq: q_om fixed by norm tr op}) 
		yield the measurability conditions on $q$. 
		
		\
		
		\item[\mylabel{iv}{it:mu-mble}]
		The Borel probability measures $(\mu_\om)_{\om\in\Om}$ also define a random probability measure.  Indeed, because intervals generate $\sB$, it is enough to check that for every interval $J\subset I$,
		the function $\om \mapsto \mu_\om(J)$ is measurable. This follows from \eqref{it:q-mble} and the fact that $\mu_\om(J)$ is the limit of measurable functions, coming from \eqref{eq: subseq limit char of nu} applied to $f_\om=\ind_J q_\om$.  
	\end{itemize}
\end{proof}
To end this section we now prove the uniqueness of the invariant density $q$. 
\begin{proposition}\label{prop: uniqueness of q}
	The global invariant density $q\in\BV_\Om(I)$ produced in Corollary~\ref{cor: exist of unique dens q} is the unique element of $L^1_\nu(\Om\times I)$ (modulo $\nu$) such that 
	\begin{align*}
	\~\cL_\om q_\om=q_{\sg(\om)}.
	\end{align*} 
\end{proposition}
\begin{proof}
	Towards a contradiction, suppose that there exists $\psi\in L^1_\nu(\Om\times I)$ with $\norm{\psi}_{L^1_\nu}=1$ such that 
	\begin{align}\label{eq: psi invariant}
		\~\cL_\om \psi_\om=\psi_{\sg(\om)}.
	\end{align} 
	By ergodicity we must have that $\norm{\psi_\om}_{L^1_{\nu_\om}}=1$ for $m$-a.e. $\om\in\Om$. Since $\BV(X_\om)$ is dense in $L^1_{\nu_\om}(X_\om)$ for each $\om\in\Om$, 
	and in particular the set 
	$\{f\in\BV(I): \norm{f}_{L^1_{\nu_\om}}=1 \}$ is dense in the set $\{f\in L^1_{\nu_\om}(I): \norm{f}_{L^1_{\nu_\om}}=1\}$, 
	for each $\dl>0$ we can find $f\in\BV_\Om(I)$ with $\norm{f_\om}_{L^1_{\nu_\om}}=1$ such that we have 
	\begin{align}\label{eq: f is L1 close to psi}
		\norm{\psi_\om-f_\om}_{L^1_{\nu_\om}}\leq \dl
	\end{align}
	for $m$-a.e. $\om\in\Om$.
	Now let $V_\psi>0$ be sufficiently large such that 
	\begin{align}\label{eq: Om_psi pos measure}
		m\lt(\Om_\psi:=\set{\om\in\Om: \var(f_\om)\leq V_\psi}\rt)>0,
	\end{align}
	and for each $\om\in\Om$ we define the function $\~\psi\in\BV_\Om(I)$ by 
	\begin{equation}\label{eq: def of tilde psi}
		\~\psi_\om:=\begin{cases}
			f_\om    & \text{ if } \om\in\Om_\psi,\\
			\ind_\om & \text{ if } \om\not\in\Om_\psi.
		\end{cases}
	\end{equation}
	Thus, using \eqref{eq: Om_psi pos measure} and \eqref{eq: def of tilde psi} we have that
	\begin{flalign}
	&\var(\~\psi_\om)\leq V_\psi<\infty, 
	\label{eq: unique dens unif bdd var} 
	&\\
	&\norm{\~\psi_\om}_{L^1_{\nu_\om}}=1 
	\label{eq: unique dens norm 1}
	\end{flalign}
	for each $\om\in\Om$ and using \eqref{eq: f is L1 close to psi} we see that
	\begin{flalign}
	&\norm{\psi_\om-\~\psi_\om}_{L^1_{\nu_\om}}\leq \dl& 
	\label{eq: unique dens L1 dense}
	\end{flalign}
	for each $\om\in\Om_\psi$.
	Now in light of \eqref{eq: unique dens unif bdd var} and \eqref{eq: unique dens norm 1}
	we see that $\~\psi\in\cD$, and thus we apply Corollary~\ref{cor: exp conv to dens in sup norm}, and more specifically equation \eqref{eq: Cor 8.6 exp ineq for large n}, to see that 
	\begin{align}
		&\norm{q_\om-\~\cL_{\sg^{-n}(\om)}^{n} \~\psi_{\sg^{-n}(\om)}}_\infty
		\leq 
		\norm{q_\om}_\infty\~\kp^n
		\label{eq: tilde psi conv to q in sup}
	\end{align}
	for all $n\in\NN$ sufficiently large, where $\~\kp\in(0,1)$.	
	Next we note that using \eqref{eq: equivariance prop of nu norm op} and \eqref{eq: psi invariant} gives that 
	\begin{align}
	\norm{\psi_\om-\~\cL_{\sg^{-n}(\om)}^n\~\psi_{\sg^{-n}(\om)}}_{L^1_{\nu_{\om}}}
	&=
	\norm{\~\cL_{\sg^{-n}(\om)}^n\psi_{\sg^{-n}(\om)}-\~\cL_{\sg^{-n}(\om)}^n\~\psi_{\sg^{-n}(\om)}}_{L^1_{\nu_\om}}
	\nonumber\\
	&\leq 
	\int_I \~\cL_{\sg^{-n}(\om)}^n\lt(\absval{\psi_{\sg^{-n}(\om)}-\~\psi_{\sg^{-n}(\om)}}\rt)\, d\nu_\om
	\nonumber\\
	&= 
	\norm{\psi_{\sg^{-n}(\om)}-\~\psi_{\sg^{-n}(\om)}}_{L^1_{\nu_{\sg^{-n}(\om)}}}
	\label{eq: L psi close to tilde L psi*}
	\end{align}
	for each $n\in\NN$. It then follows from \eqref{eq: Om_psi pos measure}, the ergodic theorem, \eqref{eq: unique dens L1 dense}, and \eqref{eq: L psi close to tilde L psi*} that for $m$-a.e. $\om\in\Om$ there are infinitely many $n\in\NN$ such that 
	\begin{align}\label{eq: L psi close to tilde L psi}
		\norm{\psi_\om-\~\cL_{\sg^{-n}(\om)}^n\~\psi_{\sg^{-n}(\om)}}_{L^1_{\nu_{\om}}}
		\leq 
		\norm{\psi_{\sg^{-n}(\om)}-\~\psi_{\sg^{-n}(\om)}}_{L^1_{\nu_{\sg^{-n}(\om)}}}
		\leq \dl.
	\end{align}
	Thus, in light of \eqref{eq: tilde psi conv to q in sup} and \eqref{eq: L psi close to tilde L psi}, for $m$-a.e. $\om\in\Om$ we have 
	\begin{align*}
	\norm{\psi_\om-q_\om}_{L^1_{\nu_\om}}
	\leq 
	\norm{q_\om-\~\cL_{\sg^{-n}(\om)}\~\psi_{\sg^{-n}(\om)}}_{L^1_{\nu_\om}} 
	+ 
	\norm{\psi_\om-\~\cL_{\sg^{-n}(\om)}^n\~\psi_{\sg^{-n}(\om)}}_{L^1_{\nu_\om}}
	\leq 2\dl
	\end{align*} 
	for infinitely many $n\in\NN$ sufficiently large.
	As this holds for each $\dl>0$, we must in fact have that $\psi_\om=q_\om$ modulo $\nu_\om$ for $m$-a.e. $\om\in\Om$, which implies that $\psi=q$ modulo $\nu$ as desired.
\end{proof}

\section{Expected Pressure}\label{sec:exp pres}
In this section we exploit the measurability produced in Proposition~\ref{prop: measurability}. 
We begin by defining the expected pressure, as in \cite{simmons_relative_2013}, and then prove an alternate characterization of the expected pressure in terms of limits. An important consequence of the proof of this characterization is that we obtain the temperedness of the quantities $\inf q_\om$ and $\norm{q_\om}_\infty$, which allows us to prove improved versions of Lemmas~\ref{lem: exp conv to mu along subseq} and \ref{lem: limit equality for nu}, thus completing the proof of Proof of Theorem~\ref{main thm: summary quasicompactness}.

Towards that end, given that $\log\lm_\om\in L^1_m(\Om)$ by Proposition~\ref{prop: measurability}, we define the expected pressure by 
\begin{align}\label{eq: expected press equalities}
\cE P(\phi):=\int_\Om \log\lm_\om \, dm(\om)=\lim_{n\to \infty}\frac{1}{n}\log\lm_\om^n=\lim_{n\to\infty}\frac{1}{n}\log\lm_{\sg^{-n}(\om)}^n,
\end{align}
where the second and third equalities hold $m$-a.e. and follow from Birkhoff's Ergodic Theorem. The following lemma provides alternate characterizations of the expected pressure.
\begin{lemma}\label{lem: conv of pressure limits}
	For $m$-a.e. $\om\in\Om$ we have that 
	\begin{align}\label{eq: backward pressure limit}
	\lim_{n\to\infty}\norm{\frac{1}{n}\log\cL_{\sg^{-n}(\om)}^n\ind_{\sg^{-n}(\om)} -\frac{1}{n}\log\lm_{\sg^{-n}(\om)}^n}_\infty=0
	\end{align}
	and 
	\begin{align}\label{eq: forward pressure limit}
	\lim_{n\to\infty}\norm{\frac{1}{n}\log\cL_{\om}^n\ind_{\om} -\frac{1}{n}\log\lm_{\om}^n}_\infty=0.
	\end{align}
	Furthermore, we have that 
	\begin{align*}
		\lim_{n\to \infty}\frac{1}{n}\log \inf q_{\sg^n(\om)}
		=
		\lim_{n\to \infty}\frac{1}{n}\log \norm{q_{\sg^n(\om)}}_\infty
		=0.
	\end{align*}
\end{lemma}
\begin{proof}
	Towards proving \eqref{eq: backward pressure limit}, Corollary~\ref{cor: exp conv to dens in sup norm} (with $p=-n$ and $f_{\sg^{-n}(\om)}=\ind_{\sg^{-n}(\om)}$) allows us to write
	\begin{align*}
	\inf q_\om - A_\ind(\om)\kp^n
	\leq 
	\frac{\cL_{\sg^{-n}(\om)}^n\ind_{\sg^{-n}(\om)}}{\lm_{\sg^{-n}(\om)}^n}
	\leq 
	\norm{q_\om}_\infty +A_\ind(\om)\kp^n
	\end{align*}	
	for each $n\in\NN$.
	Thus, for $n$ sufficiently large, we have 
	\begin{align*}
	\frac{1}{n}\log\lt(\inf q_\om - A_\ind(\om)\kp^n\rt)
	\leq 
	\frac{1}{n}\log\frac{\cL_{\sg^{-n}(\om)}^n\ind_{\sg^{-n}(\om)}}{\lm_{\sg^{-n}(\om)}^n}
	\leq 
	\frac{1}{n}\log\lt(\norm{q_\om}_\infty + A_\ind(\om)\kp^n\rt).
	\end{align*}
	Since Proposition~\ref{prop: upper and lower bound for density} implies that $\log\inf q_\om>-\infty$ and $\log\norm{q_\om}_\infty<\infty$, 
	letting $n\to\infty$ finishes the first claim.
	
	Before continuing on to prove the second claim, we first note that 
	Corollary~\ref{cor: exp conv to dens in sup norm} implies that 
	\begin{align*}
		\inf q_{\sg^n(\om)}
		\geq\inf\~\cL_\om^n\ind_\om - A_\ind(\om)\kp^n,
	\end{align*}
	for all $n\in\NN$ sufficiently large, and thus, 
	\begin{align}\label{eq: q temp 1}
		\frac{1}{n}\log\inf q_{\sg^n(\om)}\geq \frac{1}{n}\log\lt(\inf\~\cL_\om^n\ind_\om-A_\ind(\om)\kp^n\rt).
	\end{align}
	It follows from \eqref{eq: backward pressure limit} that $\inf\~\cL_\om\ind_\om$ is tempered, and thus for each $0<\dl<\kp$ there exists a $m$-a.e. finite function $K_\dl:\Om\to(0,\infty)$ such that 
	\begin{align}\label{eq: q temp 2}
		\inf \~\cL_\om^n\ind_\om\geq K_\dl(\om) e^{-\dl n}
		> A_\ind(\om)\kp^n
	\end{align}
	for each $n\in\NN$. 
	Inserting \eqref{eq: q temp 2} into \eqref{eq: q temp 1} and taking a limit as $n\to\infty$ gives 
	\begin{align*}
		\lim_{n\to\infty}\frac{1}{n}\log\inf q_{\sg^n(\om)}
		&\geq
		\lim_{n\to\infty}\frac{1}{n}\log\lt(\inf\~\cL_\om^n\ind_\om-A_\ind(\om)\kp^n\rt)
		\\
		&\geq
		\lim_{n\to\infty}\frac{1}{n}\log\lt(K_\dl(\om)e^{-\dl n}-A_\ind(\om)\kp^n\rt)
		\\
		&=
		\lim_{n\to\infty}\frac{1}{n}\log\lt(K_\dl(\om)e^{-\dl n}\rt)
		=
		-\dl.
	\end{align*}
	As this holds for every $\dl>0$ we must in fact have that 
	\begin{align}\label{eq: inf q and sup q in L1 plus BET}
	\lim_{n\to \infty}\frac{1}{n}\log \inf q_{\sg^n(\om)}
	=
	\lim_{n\to \infty}\frac{1}{n}\log \norm{q_{\sg^n(\om)}}_\infty
	=0,
	\end{align}
	where we recall that the second equality follows from Lemma~\ref{lem: BV norm q om growth bounds}.
	Now to prove the second claim, we again note that using Corollary~\ref{cor: exp conv to dens in sup norm} gives that for each $n\in\NN$ we have 
	\begin{align*}
	\inf q_{\sg^n(\om)} - A_\ind(\om)\kp^n
	\leq 
	\frac{\cL_\om^n\ind_\om}{\lm_\om^n}
	\leq 
	\norm{q_{\sg^n(\om)}}_\infty +A_\ind(\om)\kp^n,
	\end{align*}	
	and hence, for $n$ sufficiently large, we have  
	\begin{align}\label{eq: forward pressure limit eq 1}
	\frac{1}{n}\log\lt(\inf q_{\sg^n(\om)} - A_\ind(\om)\kp^n\rt)
	\leq 
	\frac{1}{n}\log\frac{\cL_\om^n\ind_\om}{\lm_\om^n}
	\leq 
	\frac{1}{n}\log\lt(\norm{q_{\sg^n(\om)}}_\infty + A_\ind(\om)\kp^n\rt).
	\end{align}	
	In view of Proposition~\ref{prop: upper and lower bound for density} and \eqref{eq: inf q and sup q in L1 plus BET},
	we see that letting $n\to\infty$ in \eqref{eq: forward pressure limit eq 1} finishes the proof of the second claim and thus the lemma. 
\end{proof}
The following corollary ensures the measurability of the coefficients $A_f$ and $\hat A_{f}$ which appeared respectively in Corollary~\ref{cor: exp conv to dens in sup norm} and Lemma~\ref{lem: exp conv to mu along subseq}.
\begin{corollary}\label{rem: B_f measurable}
	For each $f\in\BV_\Om^l(I)$ the maps $\om\mapsto A_f(\om)$ and $\om\mapsto \hat A_f(\om)$ are measurable and $f/\nu(f),\, fq,\, f/q\in\cD$.
\end{corollary}
\begin{proof}	
	From Remark~\ref{rem: remark on def of cD}, we see that $\BV_\Om^l(I)\sub\cD$ and using \eqref{eq: inf q and sup q in L1 plus BET}, Proposition~\ref{prop: upper and lower bound for density}, and Lemma~\ref{lem: props of cD}, we have that $f/\nu(f),\, fq,\, f/q\in\cD$ for each $f\in\BV_\Om^l(I)$. 

	Now to see the claim that the random coefficients $A_f$ and $\hat A_f$ are measurable we first note that the map $\om\mapsto V_{f,\ep}(\om)$ from Definition~\ref{defn: definition of cD}, which depends on $\norm{f_\om}_\BV$ and $\norm{f_\om}_{L^1_{\nu_{\om}}(I)}$, can be chosen measurably for each $f\in\BV_\Om^l(I)$. Thus the map $\om\mapsto V(\om)$, defined in \eqref{eq: defn of V(om)}, must also be measurable. In light of Remark~\ref{rem: N_1 measurable}, we see that the map $\om\mapsto N_3(\om)$, defined by \eqref{eq: def of N_3}, is measurable, which then implies that the map $\om\mapsto A_f(\om)$ is measurable for each $f\in\BV_\Om^l(I)$. 
	Furthermore, since $N_3(\om)$ depends on $\norm{f_\om}_\BV$ and $\norm{f_\om}_{L^1_{\nu_\om}(I)}$ (via $V_{f,\ep}(\om)$), so does $A_f(\om)$. 
	The measurability of $\hat A_f$ then follows from \eqref{eq: def of hat A_f}.
\end{proof}
Now in light of the proof of Lemma~\ref{lem: conv of pressure limits}, in particular \eqref{eq: inf q and sup q in L1 plus BET}, we can give improved versions of Lemmas~\ref{lem: exp conv to mu along subseq} and \ref{lem: limit equality for nu} which no longer depend upon a particular sequence of integers $n_k$.
\begin{theorem}\label{thm: exp convergence of tr op}
For each $f\in\BV_\Om^l(I)$ and each $\vkp\in(\kp,1)$ there exist measurable functions $B_f, C_f:\Om\to(0,\infty)$ such that for $m$-a.e. $\om\in\Om$, all $n\in\NN$, and all $|p|\leq n$ we have 
	\begin{align*}
		\norm{\~\cL_{\sg^{p}(\om)}^n f_{\sg^{p}(\om)} -\nu_{\sg^{p}(\om)}(f_{\sg^{p}(\om)})q_{\sg^{p+n}(\om)}}_\infty 
		\leq 
		B_f(\om)\norm{f_{\sg^{p}(\om)}}_\BV \kp^n,
	\end{align*}
	and
	\begin{align*}
		\norm{\hat \cL_{\sg^{p}(\om)}^n f_{\sg^{p}(\om)} -\mu_{\sg^{p}(\om)}(f_{\sg^{p}(\om)})\ind_{\sg^{p+n}(\om)}}_\infty 
		\leq 
		C_f(\om)\norm{f_{\sg^{p}(\om)}}_\BV \vkp^n,
	\end{align*}
	where $\kp<1$ is as in Corollary~\ref{cor: exp conv to dens in sup norm}.
\end{theorem}
\begin{proof}
	Suppose first that $f\in\BV_\Om^l(I)$ with $f\geq 0$. Since Lemma~\ref{lem: props of cD} item \eqref{lem: props of cD item c} ensures that $\ol f=(f_\om/\nu_\om(f_\om))_{\om\in\Om}\in\cD^+$, we apply Corollary~\ref{cor: exp conv to dens in sup norm} to obtain
	\begin{align}
	&\norm{\~\cL_{\sg^{p}(\om)}^n \lt(\frac{f_{\sg^{p}(\om)}}{\nu_{\sg^{p}(\om)}(f_{\sg^{p}(\om))})}\rt) -q_{\sg^{n+p}(\om)}}_\infty
	\leq 
	A_{\ol f}(\om)\norm{\frac{f_{\sg^{p}(\om)}}{\nu_{\sg^{p}(\om)}(f_{\sg^{p}(\om)})}}_\BV\kp^{n}
	\label{eq: est0 in cor exp conv to 1*}		 
	\end{align}
	for each $n\in\NN$, where $A_{\ol f}(\om)$ is as in Corollary~\ref{cor: exp conv to dens in sup norm}.
	Multiplying both sides of \eqref{eq: est0 in cor exp conv to 1*} by $\nu_{\sg^{p}(\om)}(f_{\sg^{p}(\om)})$ yields 
	\begin{align}
		&\norm{\~\cL_{\sg^{p}(\om)}^n \lt(f_{\sg^{p}(\om)}\rt) -\nu_{\sg^{p}(\om)}(f_{\sg^{p}(\om)})q_{\sg^{n+p}(\om)}}_\infty
		\leq 
		A_{\ol f}(\om)\norm{f_{\sg^{p}(\om)}}_\BV\kp^{n}.
		\label{eq: est0 in cor exp conv to 1}
	\end{align}
	In view of Corollary~\ref{rem: B_f measurable}, applying \eqref{eq: est0 in cor exp conv to 1} with the function $fq\in\cD^+$ gives 
	\begin{align}\label{eq: est1 in cor exp conv to 1}
	\norm{\~\cL_{\sg^{p}(\om)}^n (f_{\sg^{p}(\om)}q_{\sg^p(\om)}) -\mu_{\sg^{p}(\om)}(f_{\sg^{p}(\om)})q_{\sg^{n+p}(\om)}}_\infty
	\leq A_{\ol{fq}}(\om)\norm{f_{\sg^{p}(\om)}}_\BV\norm{q_{\sg^{p}(\om)}}_\BV\kp^{n}		
	\end{align}
	for each $n\in\NN$.
	It follows from \eqref{eq: est1 in cor exp conv to 1} that
	\begin{align}
	&\norm{\hat \cL_{\sg^{p}(\om)}^n f_{\sg^{p}(\om)} -\mu_{\sg^{p}(\om)}(f_{\sg^{p}(\om)})\ind_{\sg^{p+n}(\om)}}_\infty
	\nonumber\\
	&\qquad\qquad
	\leq
	\frac{1}{\inf q_{\sg^{p+n}(\om)}}\cdot 
	\norm{\~\cL_{\sg^{p}(\om)}^n (f_{\sg^{p}(\om)}q_{\sg^p(\om)}) -\mu_{\sg^{p}(\om)}(f_{\sg^{p}(\om)})q_{\sg^{p+n}(\om)}}_\infty 
	\nonumber\\
	&\qquad\qquad
	\leq 
	A_{\ol{fq}}(\om)\norm{f_{\sg^{p}(\om)}}_\BV\norm{q_{\sg^{p}(\om)}}_\BV\frac{1}{\inf q_{\sg^{p+n}(\om)}} \kp^n
	\label{eq: est2 in cor exp conv to 1*}
	\end{align}
	for each $n\in\NN$. Now in light of \eqref{eq: inf q and sup q in L1 plus BET}, we have that 
	\begin{align*}
		\lim_{n\to\infty}\frac{1}{n}\log\frac{\norm{q_{\sg^n(\om)}}_\BV}{\inf q_{\sg^n(\om)}}=0,
	\end{align*}
	and thus, for each $\dl>0$, there exists a measurable constant $D(\om,\dl)>0$ such that
	\begin{align}\label{eq: est2 in cor exp conv to 1**}
		\frac{\norm{q_{\sg^n(\om)}}_\BV}{\inf q_{\sg^n(\om)}}\leq D(\om,\dl)e^{\dl n}
	\end{align} 
	for each $n\in\NN$. Inserting \eqref{eq: est2 in cor exp conv to 1**} into \eqref{eq: est2 in cor exp conv to 1*}, for $\dl>0$ sufficiently small and each $n\in\NN$ we have 
	\begin{align}
	&\norm{\hat \cL_{\sg^{p}(\om)}^n f_{\sg^{p}(\om)} -\mu_{\sg^{p}(\om)}(f_{\sg^{p}(\om)})\ind_{\sg^{p+n}(\om)}}_\infty
	\leq 
	A_{\ol{fq}}(\om)D(\om,\dl)\norm{f_{\sg^{p}(\om)}}_\BV\vkp^n
	\label{eq: est2 in cor exp conv to 1}
	\end{align}
	where $0<\kp<\vkp:=\kp e^\dl<1$.

	Now, we write a function $f\in\BV_\Om^l(I)$ as $f=f^+- f^-$ with $f^+,f^-\geq 0$, and apply the triangle inequality with \eqref{eq: est0 in cor exp conv to 1} to obtain
	\begin{align*}
	\norm{\~\cL_{\sg^{p}(\om)}^n f_{\sg^{p}(\om)} -\nu_{\sg^{p}(\om)}(f_{\sg^{p}(\om)})q_{\sg^{p+n}(\om)}}_\infty 
	\leq 
	B_f(\om)\norm{f_{\sg^{p}(\om)}}_\BV \kp^n,
	\end{align*}
	where 
	\begin{align}\label{eq: defn of B_f}
		B_f(\om):=(A_{\ol{f^+}}(\om)+A_{\ol{f^-}}(\om)).
	\end{align}
	Similarly, using the triangle inequality with \eqref{eq: est2 in cor exp conv to 1} yields
	\begin{align*}
	\norm{\hat \cL_{\sg^{p}(\om)}^n f_{\sg^{p}(\om)} -\mu_{\sg^{p}(\om)}(f_{\sg^{p}(\om)})\ind_{\sg^{p+n}(\om)}}_\infty
	&
	\leq 
	C_f(\om)\norm{f_{\sg^{p}(\om)}}_\BV\vkp^n,
	\end{align*} 
	where 
	\begin{align}\label{eq: defn of C_f}
		C_f(\om):=D(\om,\dl)(A_{\ol{f^+q}}(\om)+A_{\ol{f^-q}}(\om)).
	\end{align}
	To finish the proof we note that claim that $B_f$ and $C_f$ are measurable for $f\in\BV_\Om^l(I)$ follows from Corollary~\ref{rem: B_f measurable} and equations \eqref{eq: defn of B_f} and \eqref{eq: defn of C_f}.
\end{proof}
The following Proposition is an improvement on Lemma~\ref{lem: limit equality for nu} in that we prove a general limit for the measures $\nu_\om$, rather than a limit along a subsequence. Furthermore, we now prove that the random measure $\nu$ is unique. 
\begin{proposition}\label{prop: uniqueness of nu and mu}
	The family $(\nu_\om)_{\om\in\Om}$ of probability measures is uniquely determined by \eqref{eq: conformal measure property}. 
	In particular, we have that for each $f\in \BV_\Om^l(I)$
	\begin{align}\label{eq: limit char of nu}
	\nu_\om(f_\om)
	=
	\lim_{n\to \infty} \frac{\norm{\cL_\om^n f_\om}_\infty}{\norm{\cL_\om^n \ind_\om}_\infty}.
	\end{align}	
	Furthermore, we have that the family $(\mu_\om)_{\om\in\Om}$ of $T$-invariant measures is also unique.
\end{proposition}
\begin{proof}
	Suppose $(\nu_\om')_{\om\in\Om}$ is another family of probability measures which satisfies \eqref{eq: conformal measure property}, i.e. 
	\begin{align*}
	\nu_{\sg(\om)}'(\cL_\om f_\om) = \lm_\om'\nu_\om'(f_\om), 
	\quad \text{ with } 
	\lm_\om'=\nu_{\sg(\om)}'(\cL_\om\ind_\om), 
	\quad f\in \BV_\Om^l(I).
	\end{align*}
	Further define $\mu_\om'=q_\om\nu_\om'$ for each $\om\in\Om$.
	For each $h\in \BV_\Om^l(I)$ Theorem~\ref{thm: exp convergence of tr op} (taken with $p=0$) applied to the measure $\mu_\om'$ gives that  
	\begin{align*}
	\absval{\norm{\hcL_\om^n h_\om}_\infty - \mu_\om'(h_\om)}
	\leq
	C_h(\om)\norm{h_\om}_\BV\vkp^n
	\end{align*}
	for each $n\in\NN$.
	Inserting the function $f/q\in\cD$, for $h\in \BV_\Om^l(I)$, into the previous estimate gives 
	\begin{align*}
	\absval{\norm{\hcL_\om^n \lt(\frac{f_\om}{q_\om}\rt)}_\infty - \nu_\om'(f_\om)}
	\leq 
	C_{\frac{f}{q}}(\om)\norm{\frac{f_\om}{q_\om}}_\BV\vkp^n
	\end{align*}
	for each $n\in\NN$.
	Thus we have	
	\begin{align*}
	\lim_{n\to \infty} \frac{\norm{\cL_\om^n f_\om}_\infty}{\norm{\cL_\om^n \ind_\om}_\infty}
	=
	\lim_{n\to \infty} \frac{\norm{\hcL_\om^n \lt(\frac{f_\om}{q_\om}\rt)}_\infty}{\norm{\hcL_\om^n \lt(\frac{\ind_\om}{q_\om}\rt)}_\infty}
	=
	\frac{\nu_\om'(f_\om)}{\nu_\om'(\ind_\om)}
	=
	\nu_\om'(f_\om).
	\end{align*}
	As the limit on the left-hand side has no dependence on the family $\nu_\om'$, the uniqueness of the limit, and of the density $q_\om$, implies the uniqueness of the families $(\nu_\om)_{\om\in\Om}$ and $(\mu_\om)_{\om\in\Om}$.
\end{proof}

We are now ready to prove Theorem~\ref{main thm: summary quasicompactness}.
\begin{proof}[Proof of Theorem~\ref{main thm: summary quasicompactness}]
	Theorem~\ref{main thm: summary quasicompactness} follows from Theorem~\ref{thm: exp convergence of tr op} with $p=-n$ and $p=0$.
\end{proof}

\section{Decay of Correlations}\label{sec:dec cor}
In this short section we prove Theorem~\ref{main thm: decay of corr}, i.e. the forward and backward exponential decay of correlations. As a consequence of this, we obtain that the invariant measure $\mu$ is ergodic, thus completing the proof of Theorem~\ref{main thm: summary existence of measures and density}.
\begin{theorem}\label{thm: dec of cor}
	For $m$-a.e. every $\om\in\Om$, every $n\in\NN$, every $|p|\leq n$, every $f\in L^1_{\mu}(\Om\times I)$, and every $h\in \BV_\Om^l(I)$ we have 
	\begin{align}\label{eq: exp dec of cor thm ineq}
	\absval{
		\mu_{\tau}
		\lt(\lt(f_{\sg^{n}(\tau)}\circ T_{\tau}^n\rt)h_{\tau} \rt)
		-
		\mu_{\sg^{n}(\tau)}(f_{\sg^{n}(\tau)})\mu_{\tau}(h_{\tau})
	}
	\leq 
	C_h(\om)\norm{f_{\sg^n(\tau)}}_{L^1_{\mu_{\sg^n(\tau)}}}\norm{h_\tau}_\BV \vkp^n,
	\end{align} 
	where $\tau=\sg^p(\om)$, and where $C_h(\om)$ and $\vkp$ are as in Theorem~\ref{thm: exp convergence of tr op}.
\end{theorem}
\begin{proof}
	First let $\hat h_\om= h_\om-\mu_\om(h_\om)$ for each $\om\in\Om$. Now, we note that 
	\begin{align}
	&\mu_{\tau}
	\lt(\lt(f_{\sg^{n}(\tau)}\circ T_{\tau}^n\rt)h_{\tau} \rt)
	-
	\mu_{\sg^{n}(\tau)}(f_{\sg^{n}(\tau)})\mu_{\tau}(h_{\tau})
	\nonumber\\
	&\qquad\qquad
	=
	\mu_{\sg^n(\tau)}
	\lt(f_{\sg^{n}(\tau)}\hcL_\tau^n h_{\tau} \rt)
	-
	\mu_{\sg^{n}(\tau)}(f_{\sg^{n}(\tau)})\mu_{\tau}(h_{\tau})
	\nonumber\\
	&\qquad\qquad
	=
	\mu_{\sg^n(\tau)}\lt(f_{\sg^n(\tau)}\hcL_\tau^n \hat h_\tau\rt).\label{eq: dec of corr est 1}				
	\end{align}
	Theorem~\ref{thm: exp convergence of tr op} then gives that 
	\begin{align}
	\norm{\hcL_\tau^n \hat h_\tau}_\infty
	=
	\norm{\hcL_\tau^n h_\tau - \mu_\tau(h_\tau)\ind_{\sg^n(\tau)}}_\infty
	\leq 
	C_h(\om)\norm{h_\tau}_\BV \vkp^n
	\label{eq: dec of corr est 2}
	\end{align} 
	for each $n\in\NN$.
	So by \eqref{eq: dec of corr est 1} and \eqref{eq: dec of corr est 2}, we have 
	\begin{align*}
	&
	\absval{
		\mu_{\tau}
		\lt(\lt(f_{\sg^{n}(\tau)}\circ T_{\tau}^n\rt)h_{\tau} \rt)
		-
		\mu_{\sg^{n}(\tau)}(f_{\sg^{n}(\tau)})\mu_{\tau}(h_{\tau})
	}
	\\
	&\qquad\qquad\qquad
	\leq
	\mu_{\sg^n(\tau)}\lt(\absval{f_{\sg^n(\tau)}\hcL_\tau^n \hat h_\tau}\rt)
	\leq
	C_h(\om)\mu_{\sg^n(\tau)}\lt(\absval{f_{\sg^n(\tau)}}\rt)\norm{h_\tau}_\BV \vkp^n
	\end{align*}
	for each $n\in\NN$,
	which finishes the proof.
\end{proof}

We are now able to prove Theorem~\ref{main thm: decay of corr}.
\begin{proof}[Proof of Theorem~\ref{main thm: decay of corr}]
	Theorem~\ref{main thm: decay of corr} now follows from Theorem~\ref{thm: dec of cor} taking $p=-n$ and $p=0$.
\end{proof}
\begin{remark}
We now show that the right-hand side of \eqref{eq: exp dec of cor thm ineq} does in fact converge to zero exponentially fast.
First, note that for $f\in L^1_\mu(\Om\times I)$ and $h\in\BV_\Om^l(I)$ the Birkhoff Ergodic Theorem implies that for each $\dl>0$, there exists some measurable function $W_{f,h}:\Om\to (0,\infty)$ such that 
\begin{align*}
	\norm{f_{\sg^n(\om)}}_{L^1_{\mu_{\sg^n(\om)}}}\leq W_{f,h}(\om)n\dl \leq W_{f,h}(\om)e^{n\dl}
	\quad\text{ and }\quad
	\norm{h_{\sg^{-n}(\om)}}_\infty\leq W_{f,h}(\om)n\dl\leq W_{f,h}(\om)e^{n\dl}.
\end{align*}
Thus we may rewrite the conclusions of Theorem~\ref{thm: dec of cor} (with $p=0$ and $p=-n$) to be  
\begin{align*}
	\absval{
		\mu_{\om}
		\lt(\lt(f_{\sg^{n}(\om)}\circ T_{\om}^n\rt)h \rt)
		-
		\mu_{\sg^{n}(\om)}(f_{\sg^{n}(\om)})\mu_{\om}(h)
	}
	\leq C_{h}(\om)W_{f,h}(\om)\kp_0^n
\end{align*} 
and 
\begin{align*}
\absval{
	\mu_{\sg^{-n}(\om)}
	\lt(\lt(f_\om\circ T_{\sg^{-n}(\om)}^n\rt)h_{\sg^{-n}(\om)} \rt)
	-
	\mu_\om(f_\om)\mu_{\sg^{-n}(\om)}(h_{\sg^{-n}(\om)})
}
\leq C_{h}(\om)W_{f,h}(\om)\kp_0^n
\end{align*}
for each $n\in\NN$,
where $0<\vkp<\kp_0=\vkp\cdot e^\dl<1$ for some $\dl>0$ sufficiently small. Thus we see that the right-hand side of \eqref{eq: exp dec of cor thm ineq} does in fact go to zero exponentially fast.
\end{remark}
\begin{remark}
	If one were able to show that the measurable function $\Om\ni\om\mapsto C_f(\om)\in\RR^+$ were in fact $m$-integrable, then, following the proof of \cite[Theorem 6.2]{mayer_random_2018}, which is written in the setting where the coefficient $C_f(\om)$ is independent of $f$ and $\om$ and uses the method of Gordin and Liverani, we immediately obtain the Central Limit Theorem as a consequence of Theorem~\ref{thm: dec of cor}. 
	For a more thorough treatment of stronger limit laws in the setting in which the coefficient $C_f(\om)$ is independent of $f$ and $\om$ see \cite{dragicevic_spectral_2018,dragicevic_almost_2018}.
\end{remark}
We are finally ready to prove the following proposition, which hinges on Theorem~\ref{thm: dec of cor}, and will complete the proof of Theorem~\ref{main thm: summary existence of measures and density}.
\begin{proposition}\label{prop: mu is ergodic}
	The measure $\mu$ defined in Proposition~\ref{prop: measurability} is ergodic. 
\end{proposition}
The proof of Proposition~\ref{prop: mu is ergodic} is the same as the proof of Proposition 4.7 in \cite{mayer_distance_2011}, and is therefore left to the reader.

We now prove Theorem~\ref{main thm: summary existence of measures and density}.
\begin{proof}[Proof of Theorem~\ref{main thm: summary existence of measures and density}]
	\
	
	Part (1) of Theorem~\ref{main thm: summary existence of measures and density} follows from Proposition~\ref{prop: existence of conformal family},  Proposition~\ref{prop: measurability}, and Proposition~\ref{prop: uniqueness of nu and mu}. 
	
	Part (2) follows from Corollary~\ref{cor: exist of unique dens q}, Proposition~\ref{prop: measurability},  Proposition~\ref{prop: uniqueness of q}, and Lemma~\ref{lem: conv of pressure limits}.
	
	Part (3) follows from Proposition~\ref{prop: mu_om T invar}, Proposition~\ref{prop: measurability}, Proposition~\ref{prop: uniqueness of nu and mu}, and Proposition~\ref{prop: mu is ergodic}.

\end{proof}
\section{Relative Equilibrium States}\label{sec:equil}
In this section we show that the random $T$-invariant probability measure $\mu$ defined in \eqref{def: global conformal and inv random measures} is in fact the unique relative equilibrium state for the system, thus completing the proof of Theorem~\ref{main thm: eq states}. 
Recall that $\cP_T(\Om\times I)\sub\cP_\Om(I)$ denotes the set of all $T$-invariant random probability measures on $I$, and for $\eta\in\cP_T(\Om\times I)$ the conditional information of the partition $\cZ_\om^*$ given $T_\om^{-1}\sB$, with respect to $\eta_\om$, is given by 
\begin{align*}
	I_{\eta_\om}=I_{\eta_\om}[\cZ_{\om}^*|T_\om^{-1}\sB]:=-\log g_{\eta,\om},
\end{align*} 
where 
\begin{align*}
	g_{\eta,\om}:=\sum_{Z\in\cZ_\om^*}\ind_Z E_{\eta_\om}\lt(\ind_Z \rvert T_\om^{-1}\sB \rt).
\end{align*}
	\begin{remark}\label{rem: g_mu = g hat}
		Note that $g_{\mu,\om}=\hat g_\om$, where $\hat g_\om$ is the weight function defined in \eqref{eq: def of hat g_om}. To see this we note that 
		\begin{align*}
			E_{\mu_\om}\lt(\ind_Z\rvert T_\om^{-1}\sB\rt)
			=
			\hat g_\om\rvert_Z 
		\end{align*}
		for each $Z\in\cZ_\om^*$ since we have
		\begin{align*}
			\hcL_\om\ind_Z = \hat g_\om\rvert_Z.
		\end{align*}
	\end{remark}
We now show that the $T$-invariant random measure $\mu$ defined in Proposition~\ref{prop: measurability} is a relative equilibrium state, where we have followed the approach of \cite{liverani_conformal_1998} using the conditional information $I_{\mu_\om}$ rather than the entropy $h_\mu(T)$, which for countable partitions may be infinite.
\begin{lemma}\label{lem: mu is an eq state}
The $T$-invariant random measure $\mu$ defined by \eqref{def: global conformal and inv random measures} is a relative equilibrium state, i.e. 
	\begin{align*}
		\cE P(\phi)=\int_\Om
		\lt(
		\int_{I}I_{\mu_\om}+\phi_\om \, d\mu_\om
		\rt)
		 \, dm(\om).
	\end{align*}
\end{lemma}
\begin{proof}
	In light of Remark~\ref{rem: g_mu = g hat} and \eqref{eq: def of hat g_om}, the conditional information of $\cZ_\om^*$ given $T_\om^{-1}\sB$, with respect to $\mu_\om$, is given by   
	\begin{align*}
		I_{\mu_\om}
		=
		-\log g_{\mu,\om}
		=
		-\log\hat g_\om
		=
		-\hat\phi_\om.
	\end{align*}
	Thus, we can write 
	\begin{align*}
		\int_{I} I_{\mu_\om} + \phi_\om \, d\mu_\om
		&
		=
		\int_{I} -\hat\phi_\om + \phi_\om \, d\mu_\om
		\\
		&=\log \lm_\om +  \int_{I}q_{\sg(\om)}\circ T_\om-q_\om \,d\mu_\om
		\\
		&=\log \lm_\om + \int_{I} q_{\sg(\om)}\, d\mu_{\sg(\om)} -\int_{I}q_\om \,d\mu_\om.		
	\end{align*}
	As $m$ is $\sg$-invariant, integrating over $\Om$ completes the proof. 
\end{proof}
In the following lemma we prove a variational principle for $T$-invariant random measures on $\cJ$. 
\begin{lemma}\label{lem: VP on fibers}
For all $\eta\in\cP_T(\cJ)$, the collection of all $T$-invariant random probability measures supported on $\cJ$, we have 
	\begin{align}\label{eq: vp on fibers clm 1}
		\log \lm_\om\geq \int_{X_\om}I_{\eta, \om}+\phi_\om \, d\eta_\om
	\end{align}
	for each $\om\in\Om$. Consequently, we have
	\begin{align}\label{eq: vp on fibers clm 2}
		\cE P(\phi)\geq \int_\Om
		\lt(
		\int_{X_\om}I_{\eta_\om} +\phi_\om \, d\eta_\om
		\rt)
		 \, dm(\om),
	\end{align}
	where equality holds if and only if $\eta_\om=\mu_\om$ for $m$-a.e. $\om\in\Om$. 
\end{lemma}
\begin{proof}
	Let $\eta\in\cP_T(\cJ)$. Then 
	\begin{align*}
		I_{\eta_\om}+\phi_\om
		&=
		-\log g_{\eta,\om}+\phi_\om
		\\
		&= 
		-\log g_{\eta,\om}+ \log \hat g_\om+\log\lm_\om-\log q_\om+\log q_{\sg(\om)}\circ T_\om
		\\
		&=
		\log\frac{\hat g_\om}{g_{\eta,\om}}+\Phi_\om,
	\end{align*}
	where 
	\begin{align*}
		\Phi_\om:=\log\lm_\om-\log q_\om+\log q_{\sg(\om)}\circ T_\om.
	\end{align*}
	Noting that $\eta_\om(\Phi_\om)=\log \lm_\om$, we see that to prove \eqref{eq: vp on fibers clm 1} it suffices to show that 
	\begin{align*}
		\int_{X_\om} \log\frac{\hat g_\om}{g_{\eta,\om}} \, d\eta_\om \leq 0.
	\end{align*} 
	Now, let $\cF_\om$ denote the set of bounded functions $f:X_\om\to\RR$ which are supported in a finite number of intervals of $\cZ_\om^{(1)}$.
	Let $\cL_{\eta,\om}:\cF_\om\to\cF_{\sg(\om)}$ be given by
	\begin{align}\label{eq: def of eta tr op}
	\cL_{\eta,\om} f(x) =\sum_{y\in T_\om^{-1}(x)}g_{\eta,\om}(y)f(y),
	\quad f\in\cF_\om,\,x\in X_{\sg(\om)}.
	\end{align}
	To complete the proof of \eqref{eq: vp on fibers clm 1} we prove the following claim. 
	\begin{claim}\label{clm: vp clm}
		\begin{flalign*}
			& \text{For all } f\in \cF_\om \text{ we have } \eta_{\sg(\om)}\lt(\cL_{\eta,\om}(f)\rt)=\eta_\om (f),
			\tag{i}\label{eq: eta tr op conf prop}
			&\\
			& \log^+\frac{\hat g_\om}{g_{\eta,\om}} \in L^1_{\eta_\om}(X_\om) \text{ for each } \om\in\Om, 
			\tag{ii} \label{item: i vp clm} 
			\\
			& \int_{X_\om}\log\frac{\hat g_\om}{g_{\eta,\om}}\leq 0. 	
			\tag{iii} \label{item: ii vp clm} 
		\end{flalign*}
	\end{claim}
	\begin{subproof}
		We begin by noting that we may write 
		\begin{align}\label{eq: alternate tr op form}
			\cL_{\eta,\om}(f)=\sum_{Z\in\cZ_\om^{(1)}}(f\cdot \ind_Z\cdot g_{\eta,\om})\circ T_{\om,Z}^{-1}.
		\end{align}
		To prove \eqref{eq: eta tr op conf prop} it suffices to prove the claim holds for the function $\ind_A$ for some set $A$ which is contained in a finite number of partition elements. To that end, using \eqref{eq: alternate tr op form} and the $T$-invariance of $\eta$, we calculate 
		\begin{align*}
			\eta_{\sg(\om)}(\cL_{\eta,\om}\ind_A)
			&=
			\eta_{\sg(\om)}\lt(\sum_{Z\in\cZ_\om^{(1)}}\lt(\ind_A\ind_Zg_{\eta,\om}\rt)\circ T_{\om,Z}^{-1}\rt)
			\\
			&=
			\eta_\om\lt(\sum_{Z\in\cZ_\om^{(1)}}\ind_A\ind_Zg_{\eta,\om}\rt) 
			\\
			&=
			\sum_{Z\in\cZ_\om^{(1)}}\eta_\om\lt(\ind_A\ind_Z \sum_{Y\in\cZ_\om^{(1)}}\lt(\ind_Y E_{\eta_\om}\lt(\ind_Y|T_\om^{-1}\sB\rt)\rt)\rt)
			\\
			&=
			\sum_{Z\in\cZ_\om^{(1)}}\eta_\om\lt(\ind_A\ind_ZE_{\eta_\om}\lt(\ind_Z|T_\om^{-1}\sB\rt)\rt)
			\\
			&=
			\sum_{Z\in\cZ_\om^{(1)}}\eta_\om(\ind_A\ind_Z) = \eta_\om(\ind_A).
			\end{align*}
		
		To prove \eqref{item: i vp clm} we first note that $g_{\eta,\om}>0$ $\eta_\om$-a.e. since, if $g_{\eta,\om}\rvert_A=0$ for some subset $A\sub X_\om$, then $\cL_{\eta,\om}\ind_A=0$, and in light of \eqref{eq: eta tr op conf prop} we would have $\eta_\om(A)=0$.
		
		Enumerate the intervals of $\cZ_\om^{(1)}$ by $Z_j$ for $j\in\NN$. Since $\log g_{\eta,\om}>-\infty$ $\eta_\om$-a.e. and $\sup\hat g_\om<\infty$ (by assumption \eqref{cond SP3}, \eqref{eq: lm>0}, and Proposition~\ref{prop: upper and lower bound for density}), we set 
		\begin{equation}\label{eq: def of F}
			F_\om(n)(x)
			:=
			\begin{cases}
				1 & \text{if } x\in Z_1\cup\dots\cup Z_n \text{ and } \absval{\log\frac{\hat g_\om}{g_{\eta,\om}}(x)}<n\\
				0 & \text{otherwise,}
			\end{cases}
		\end{equation}
		and 
		\begin{align*}
			\chi_\om^+:=\set{x\in X_\om:\log\frac{\hat g_\om}{g_{\eta,\om}}(x)\geq 0 }.
		\end{align*}
		Since $\hat g_\om\geq g_{\eta,\om}$ on $\chi_\om^+$, we have 
		\begin{align}\label{eq: L_eta leq L hat}
			0
			\leq
			\cL_{\eta,\om}\lt(F_\om(n)\ind_{\chi_\om^+}\rt)
			\leq
			\hcL_\om\lt(F_\om(n)\ind_{\chi_\om^+}\rt)
			\leq 
			\hcL_\om\ind_\om
			=
			\ind_{\sg(\om)}.
		\end{align}
		In light of the definitions of $\hcL_\om$, \eqref{eq: alt defhat L}, and $\cL_{\eta,\om}$, \eqref{eq: def of eta tr op}, we see that  
		\begin{align}\label{eq: eta tr op eq to hat L}
			\cL_{\eta,\om}\lt(f\frac{\hat g_\om}{g_{\eta,\om}}\rt)(x)
			=
			\hcL_\om f(x).
		\end{align} 
		Now, using \eqref{eq: def of F}, \eqref{eq: eta tr op conf prop}, the fact that $\log x\leq x-1$ for $x>0$, \eqref{eq: eta tr op eq to hat L}, and \eqref{eq: L_eta leq L hat} we thus get that 
		\begin{align}
			\int_{X_\om}\ind_{\chi_\om^+}\log\frac{\hat g_\om}{g_{\eta,\om}} \, d\eta_\om
			&=
			\lim_{n\to \infty}\int_{X_\om}F_\om(n) \ind_{\chi_\om^+}\log\frac{\hat g_\om}{g_{\eta,\om}} \, d\eta_\om
			\nonumber\\
			&=
			\lim_{n\to \infty}\int_{X_{\sg(\om)}}
			\cL_{\eta,\om}
			\lt(
			F_\om(n) \ind_{\chi_\om^+}\log\frac{\hat g_\om}{g_{\eta,\om}}
			\rt)
			\, d\eta_{\sg(\om)}
			\nonumber\\
			&\leq
			\lim_{n\to \infty}\int_{X_{\sg(\om)}}
			\cL_{\eta,\om}
			\lt(
			F_\om(n) \ind_{\chi_\om^+}\lt(\frac{\hat g_\om}{g_{\eta,\om}}-1
			\rt)\rt)
			\, d\eta_{\sg(\om)}
			\nonumber\\
			&=
			\lim_{n\to \infty}\int_{X_{\sg(\om)}}
			\lt(
			\cL_{\eta,\om}
			\lt(
			F_\om(n) \ind_{\chi_\om^+}\frac{\hat g_\om}{g_{\eta,\om}}\rt)
			-
			\cL_{\eta,\om}\lt(F_\om(n) \ind_{\chi_\om^+}\rt)
			\rt)
			\, d\eta_{\sg(\om)}	
			\nonumber\\
			&=
			\lim_{n\to \infty}\int_{X_{\sg(\om)}}
			\lt(
			\hcL_{\om}
			\lt(F_\om(n) \ind_{\chi_\om^+}\rt)
			-
			\cL_{\eta,\om}\lt(F_\om(n) \ind_{\chi_\om^+}\rt)
			\rt)
			\, d\eta_{\sg(\om)}
			\label{eq: long calc in clm 1 of vp}
			\\
			&\leq 1.
			\nonumber		
		\end{align}
		Hence, we have that $\chi_\om^+\log(\hat g_\om / g_{\eta,\om})\in L^1_{\eta_\om}(X_\om)$, which proves \eqref{item: i vp clm}. 
		
		Now to prove claim \eqref{item: ii vp clm}, we note that if 
		$$
			\int_{X_\om}\log\frac{\hat g_\om}{g_{\eta,\om}} \, d\eta_\om=-\infty
		$$
		then we are done. Otherwise, by \eqref{item: i vp clm}, we have $\log(\hat g_\om / g_{\eta,\om})\in L^1_{\eta_\om}(X_\om)$. Repeating the calculation from the beginning of the previous display block of equations to  \eqref{eq: long calc in clm 1 of vp} without the characteristic function $\ind_{\chi_\om^+}$ and then using \eqref{eq: fully norm op fix ind} and \eqref{eq: eta tr op conf prop} gives 
		\begin{align*}
			\int_{X_\om}\log\frac{\hat g_\om}{g_{\eta,\om}} \, d\eta_\om
			&\leq
			\lim_{n\to \infty}\int_{X_{\sg(\om)}}
			\hcL_{\om}
			\lt(F_\om(n)\rt)
			\, d\eta_{\sg(\om)} 
			-
			\lim_{n\to \infty}\int_{X_{\sg(\om)}}
			\cL_{\eta,\om}\lt(F_\om(n) \rt)
			\, d\eta_{\sg(\om)} 
			\\
			&\leq
			\int_{X_{\sg(\om)}}
			\hcL_{\om}
			\lt(\ind_\om\rt)
			\, d\eta_{\sg(\om)} 
			-
			\lim_{n\to \infty}\int_{X_{\sg(\om)}}
			\cL_{\eta,\om}\lt(F_\om(n) \rt)
			\, d\eta_{\sg(\om)} 
			\\
			&\leq
			1-\lim_{n\to\infty}\int_{X_\om} F_\om(n) \, d\eta_\om
			=0.
		\end{align*}
\end{subproof}
Thus, we have that
\begin{align*}
\int_{X_\om}I_{\eta_\om}+\phi_\om \, d\eta_\om 
= 
\int_{X_\om}\Phi_\om \, d\eta_\om
+
\int_{X_\om}\log\frac{\hat g_\om}{g_{\eta,\om}} \, d\eta_\om
\leq
\log \lm_\om,
\end{align*}
which finishes the proof of \eqref{eq: vp on fibers clm 1}. Thus, \eqref{eq: vp on fibers clm 2} follows by integrating over $\Om$ . 

Finally, we note that equality holds if and only if $\hat g_\om=g_{\eta,\om}$ $\eta_\om$-almost everywhere on $X_\om$, which implies that $\hcL_\om f(x)=\cL_{\eta,\om} f(x)$ for $\eta_\om$-a.e. $x\in X_\om$ and all $f\in\BV(X_\om)$. Thus, applying Theorem~\ref{thm: exp convergence of tr op} with $p=0$, we have 
		\begin{align*}
			\absval{\eta_\om(f)-\mu_\om(f)}
			&=
			\absval{\eta_{\sg^n(\om)}(\cL_{\eta,\om}^n f)-\mu_\om(f)}
			\\
			&=
			\absval{\eta_{\sg^n(\om)}(\hcL_{\om}^n f)-\mu_\om(f)}
			\\
			&=
			\absval{\eta_{\sg^n(\om)}\lt(\hcL_{\om}^n f-\mu_\om(f)\ind_{\sg^n(\om)}\rt)}
			\\
			&\leq 
			\absval{\eta_{\sg^n(\om)}\lt(\norm{\hcL_{\om}^n f-\mu_\om(f)\ind_{\sg^n(\om)}}_\infty\rt)}
			\\
			&\leq 
			C_f(\om)\norm{f}_\BV \vkp^n
		\end{align*}
		for each $n\in\NN$ and every $f\in\BV(X_\om)$. Thus we must in fact have that $\eta_\om(f)=\mu_\om(f)$ for each $f\in\BV(X_\om)$, and thus $\eta_\om=\mu_\om$ on a dense set of $\cC(X_\om)$. We conclude finally that $\mu_\om=\eta_\om$ for $m$-a.e. $\om\in\Om$, which completes the proof.
\end{proof}
To finish the proof of Theorem~\ref{main thm: eq states}, we now extend the variational principle on $\cJ$, proved in Lemma~\ref{lem: VP on fibers}, to $\Om\times I$.
\begin{lemma}\label{lem: vp on Om x I}
	For any $\eta\in \cP_T(\Om\times I)$ we have 
	\begin{align}\label{eq: vp on whole int statement}
		\cE P(\phi)\geq \int_\Om
		\lt(
		\int_I I_{\eta_\om} +\phi_\om \, d\eta_\om
		\rt)
		\, dm(\om). 
	\end{align}
	Moreover, equality holds if and only if $\eta_\om=\mu_\om$ for $m$-a.e. $\om\in\Om$. 
\end{lemma}
\begin{proof}
	Suppose $\eta\in\cP_T(\Om\times I)$ is a $T$-invariant random probability measure on $I$, and suppose that for each $\om\in\Om$ we have that 
	\begin{align*}
		\eta_\om=c_\om \eta_\om^c + a_\om \eta_\om^a
	\end{align*}
	where $\eta_\om^c$ is a non-atomic probability measure on $I$, $\eta_\om^a$ is a purely atomic probability measure on $I$, and $c_\om+a_\om=1$. 
	As $\eta_\om^c$ is non-atomic, we must have that $\eta_\om^c(\cS_\om)=0$ as the singular set $\cS_\om$, defined by \eqref{eq: def of sing set}, is countable. 
	Without loss of generality, we suppose that 
	$$
		\int_\Om a_\om \, dm(\om)>0,
	$$ 
	as otherwise the lemma would follow from Lemma~\ref{lem: VP on fibers}.

	We easily see that $I_{\eta_\om^a, \om} =0$ $\eta_\om^a$-a.e, and since $\eta_\om^a$ and $\eta_\om^c$ are singular, we can choose $I_{\eta_\om^a, \om} =0$ $\eta_\om^c$-a.e and $I_{\eta_\om^c, \om} =0$ $\eta_\om^a$-a.e. Therefore, we have that 
	\begin{align*}
		I_{\eta, \om} 
		&= 
		I_{\eta^c, \om} 
		+
		I_{\eta^a, \om} 
		=
		I_{\eta^c, \om} , 
	\end{align*}
	and hence, applying Lemma~\ref{lem: VP on fibers} gives
	\begin{align}
		\int_\Om
		\lt(
		\int_I I_{\eta_\om} +\phi_\om \, d\eta_\om 
		\rt)
		\, dm(\om)
		&=
		\int_\Om c_\om 
		\int_I I_{\eta^c,\om} +\phi_\om \, d\eta_\om^c
		\, dm(\om)
		+
		\int_\Om a_\om
		\int_I \phi_\om d\eta_\om^a
		\, dm(\om)
		\nonumber\\
		&\leq 
		\int_\Om c_\om \log\lm_\om\, dm(\om)
		+
		\int_\Om a_\om
		\int_I \phi_\om d\eta_\om^a
		\, dm(\om).	
		\label{eq: int of info}
	\end{align}
	Since $\eta$ is a random probability measure we have that $\eta_\om^a(\phi_\om)\in L^1_m(\Om)$. 
	Thus, applying the Birkhoff Ergodic Theorem, \eqref{cond CP1}, and Lemma~\ref{lem: conv of pressure limits}
	for each $\ep>0$ we can find $n\in\NN$ so large that
	\begin{align*}
		\int_\Om\int_I \phi_\om\, d\eta_\om^a\, dm(\om) -\ep
		&\leq 
		\frac{1}{n}S_{n,\sg}(\eta_\om^a(\phi_\om))
		\\
		&=
		\frac{1}{n}\lt(\eta_\om^a(\phi_\om)+\dots +\eta_\om^a(\phi_{\sg^{n-1}(\om)}\circ T_\om^{n-1})\rt)
		\\
		&=
		\frac{1}{n}\eta_\om^a(\log g_\om^{(n)})
		\leq 
		\frac{1}{n}\log \norm{g_\om^{(n)}}_\infty
		\\
		&<
		\frac{1}{n}\log\inf\cL_\om^n\ind_\om
		\leq 
		\cE P(\phi)+\ep.
	\end{align*}
	As this holds for every $\ep>0$ we must in fact have 
	\begin{align}\label{eq: atom int with a_om}
		\int_\Om a_\om \int_I \phi_\om \, d\eta_\om^a\, dm(\om)
			&<
		\int_\Om a_\om\log \lm_\om\, dm(\om).
	\end{align}	
	Inserting \eqref{eq: atom int with a_om} into \eqref{eq: int of info} we see that then 
	\begin{align*}
			\int_\Om\int_I I_{\eta_\om} +\phi_\om \, d\eta_\om\, dm(\om)
			&<
			\int_\Om c_\om\log\lm_\om\, dm(\om)
			+
			\int_\Om a_\om\log\lm_\om\, dm(\om)
			\\
			&=
			\int_\Om \log\lm_\om\, dm(\om).
	\end{align*}
	Thus, integrating with respect to $m$, we see that if $\int_\Om a_\om\, dm(\om)>0$, then 
	\begin{align*}
	\cE P(\phi)> 
	\int_\Om
	\lt(
	\int_I I_{\eta_\om} +\phi_\om \, d\eta_\om
	\rt)
	\, dm(\om). 
	\end{align*}
	Thus for equality to hold in \eqref{eq: vp on whole int statement} for $\eta\in\cP_T(\Om\times I)$ with $\eta_\om=c_\om\eta_\om^c+a_\om\eta_\om^a$, we see that we must have 
	$$
		\int_\Om a_\om\, dm(\om)=0,
	$$ 
	and thus, in light of Lemma~\ref{lem: VP on fibers}, we must have that $\eta_\om=\eta_\om^c=\mu_\om$ for $m$-a.e. $\om\in\Om$. 
\end{proof}
\begin{proof}[Proof of Theorem~\ref{main thm: eq states}]
	Theorem~\ref{main thm: eq states} now follows from Lemmas~\ref{lem: mu is an eq state}, \ref{lem: VP on fibers}, and \ref{lem: vp on Om x I}.
\end{proof}
\section{Examples}\label{sec:examples}
In this section we present several examples for which our general theory applies. In particular, we treat random  $\bt$-transformations, randomly translated random $\bt$-transformations, random infinitely branched Gauss-Renyi maps, random maps with non-uniform expansion (specifically random systems with Pomeau-Manneville maps and random maps with contracting branches), and a broad class of random Lasota-Yorke maps. 

Concerning examples, we would like to point out that $\beta$-transformations are an important class of systems, among the most studied in ergodic theory since the pioneering  work of Parry \cite{parry_-expansions_1960}; they serve also as a prototype  to test theories of random perturbations in the 
non-stationary setting. For instance limit theorems for sequential expanding dynamical systems were rigorously proved in \cite{ConzeRaugi} for $\beta$-transformations using quite advanced technical tools.  Similar care is necessary to treat the quenched case we deal with and we will show that the methods we use can also  be generalized to wider classes of map like the Lasota-Yorke ones.

As stated in Remark~\ref{rem: cond M5' and M6' and tilde N*}, it suffices to check conditions \eqref{cond M5'} and \eqref{cond M6'} rather than \eqref{cond M5} and \eqref{cond M6}. Thus, in the general strategy of checking that the conditions are satisfied for a particular example, one must find the number $N_*$ as defined in \eqref{eq: def of N}. For computational ease we have selected some of our examples (Sections \ref{sec:betadelta}, \ref{sec:infiniteBranches}, and \ref{sec:non-unifexp}) so that we will have $N_*=1$. However, we would like to point out that one could deal with the case of $N_*>1$ using similar arguments as the ones we present here. 

For the rest of this section, we assume $\Om$ is a topological space, $m$ is a complete probability measure on its Borel sigma-algebra and $\sigma: (\Omega, m) \to (\Omega, m)$ is  a continuous, ergodic, invertible, probability-preserving map.

\subsection{Random $\bt$-Transformations (I)}\label{sec:beta}
Let
$\beta: \Omega \to \RR_+$, given by $\om \mapsto \beta_\om$, be an $m$-continuous function. That is, there is a partition of $\Om \pmod{m}$  into at most countably many Borel sets $\Om_1, \Om_2, \dots$ such that $\beta$ is constant on each $\Om_j$, say $\beta|_{\Om_j}=\beta_j$. 
Consider a random beta transformation
$T_\om : [0,1] \to [0,1]$ given by
\[
T_\om(x) = \beta_\omega x \pmod{1}.
\]
\begin{lemma}\label{lem:randomcovering} 
	Assume $\beta$ satisfies the \emph{expanding on average} condition $E:=\int_\Om \log \beta_\om \,dm>0$.
	Then, the system has the random covering property \eqref{cond RC} (Definition~\ref{def: covering}). 
\end{lemma}
\begin{proof}
	Let $J\subset I$ be an interval.
	The expanding on average condition, combined with the Birkhoff ergodic theorem, ensures that for \maeom, in finite time, say $\ell-1 \in \NN$, the image of $J$  contains a discontinuity in its interior. Hence, at the following iterate, it contains an interval of the form $[0,a]$ for some $a>0$. 
	After this, it takes at most 
	$$
	\tilde n (\om)= \min \{ n\geq 1 : \beta_{\sigma^{\ell}\om}^{(n)} a \geq 1 \}
	$$ 
	more iterates for the images of $J$ 
	to cover $I$,  where $\beta_\om^{(n)} = \beta_\om \beta_{\sigma\om} \dots \beta_{\sigma^{n-1}\om}$. Once again, this time is finite for \maeom\ because of the expanding on average condition and Birkhoff ergodic theorem.
\end{proof}


\begin{lemma}\label{lem:excontPot}
	Let $\Phi:\Om \to BV(I)$ be an $m$-continuous function\footnote{Without loss of generality, we assume the partitions on which $\beta$ and $\Phi$ are constant coincide.}.
	Let
	$\phi:\Om \times I \to \RR$ be given by $\phi_\om:=\phi(\om, \cdot)=\Phi(\om)$. Then
	$g_\om = e^{\phi_\om}\in BV(I)$ for \maeom.
	Assume there exist  $\phi_-, \phi_+\in\RR$ such that $\phi_-\leq \phi(\om,x)<\phi_+$, and that 
	$$
	\phi_+ - \phi_-< \int_\Om \log \floor{\beta_\om} \,dm=:E_-.
	$$ 
	Furthermore, assume the \emph{integrability condition} $\int_\Om \log \beta_\om\, dm<\infty$
	\footnote{Note that this condition is equivalent with the assumption that $\int_\Om \log \ceil{\beta_\om}\, dm<\infty$}.
	Finally, assume that either 
	\begin{enumerate}[(i)]
		\item $g_\om=|T'_\om|^{-t}$ for some $t\geq 0$,
		\item condition \eqref{cond P1} holds for the dynamical partitions $\cP_{\om,n}(0,1)=\cZ_\om^{(n)}$.
	\end{enumerate}
	Then 
	$\phi$ is a summable contracting potential for $(\Om,\sg,m,(T_\om)_{\om\in\Om})$, and  
	conditions \eqref{cond GP} and  \eqref{cond M1}--\eqref{cond M4} are satisfied.
\end{lemma}
\begin{proof}
	Condition \eqref{cond GP} holds because the system is expanding on average.
	The summability conditions  (Definition~\ref{def: summable potential}) are straightforward to check. Indeed, \eqref{cond SP2} is a hypothesis. \eqref{cond SP1} and \eqref{cond SP3} hold because $\phi$ is bounded and the partitions are finite. 
	When $g_\om=|T'_\om|^{-t}$ for some $t\geq 0$,
	\eqref{cond P1} holds with $\hal=0$ because $g_\om^{(n)}$ is constant a.e. In general,
	\eqref{cond P2} holds with $\hgm=1$ because $\cP_{\om,n}(0,1)$ coincides with the dynamical partition $\cZ_\om^{(n)}$.
	\eqref{cond M1} and \eqref{cond M2} follow from $m$-continuity; see e.g. \cite{FLQ2}.
	\eqref{cond M3} holds because 
	$e^{\phi_-}\leq S_{\om}^{(1)} \leq \ceil{\beta_\om} e^{\phi_+}$, where $\ceil{\beta_\om}$  is the least integer greater than or equal to $\beta$ and also the cardinality of $\mathcal Z_\om^*$. Thus,
	$$ 
	\int_\Om | \log S_{\om}^{(1)} | \,dm \leq \max \set{|\phi_-|, \int_\Om \log \ceil{\beta_\om}\, dm + |\phi_+| }<\infty.
	$$ 
	Notice that $\inf \cL_\om \ind_\om \geq n(T_\om) e^{\phi_-}$,
	where $n(T_\om) =  \min  \{ \#  T_\om^{-1}(x) :  x\in I \}$ is the minimum number of preimages of a point $x\in I$ under $T_\om$. Since $T_\om$ is a $\bt$-transformation, 
	$$
	n(T_\om) = \#  (T_\om)^{-1}(1)= \floor{\beta_{\om}},
	$$ 
	the largest integer less than or equal to $\beta_\om$. 
	Thus, in view of Remark~\ref{rmk:checkCP}, with $N_1=N_3=1$, in order to show the potential is contracting (Definition~\ref{def: contracting potential}), it is sufficient to have that 
	\begin{equation*}\label{eq:excontPot}
	\phi_+<  \phi_- +\int_\Om  \log \floor{\beta_{\om}}\, dm,
	\end{equation*}
	as assumed.
	
	Furthermore, the condition $E_->\phi_+-\phi_-$ ensures that $\floor{\beta_\om}\geq1$ for \maeom. This implies that $\beta_\om\geq 1$ for \maeom, and so $\log \inf \cL_\om \ind \geq \phi_-$. This yields \eqref{cond M4}.
	%
\end{proof}
\begin{remark}
	Note that the hypotheses of Lemma~\ref{lem:excontPot} imply those of Lemma~\ref{lem:randomcovering}.
\end{remark}
\begin{lemma}\label{lem:condLem=>condThm}
	Under the hypotheses of Lemma~\ref{lem:excontPot}, Theorems~\ref{main thm: summary existence of measures and density} -- \ref{main thm: eq states} apply provided
	\begin{equation}\label{eq:IntCond2}
	M_{\om,N_*} \in L^1_m(\Om)\quad \text{and} \quad  \log b_\om^{(M_{\om,N_*})} \in L^1_m(\Om),
	\end{equation}
	where $N_*$ is defined in \eqref{eq: def of N}, $M_{\om,N_*} = \max_{ Z\in \mathcal Z_\om^{(N_*)}} M_\om(Z)$ is the maximum covering time for branches of $T_\om^{N_*}$  and
	$b_\om:=\ceil{\beta_\om}$ is the least integer greater than or equal to $\beta_\om$.
\end{lemma}
\begin{proof}
	In view of Lemmas~\ref{lem:randomcovering} and \ref{lem:excontPot}, in order to satisfy the hypotheses of Theorems~\ref{main thm: summary existence of measures and density} -- \ref{main thm: eq states}, it is enough to verify the conditions 
	\eqref{cond M5'} and \eqref{cond M6'}, namely
	\begin{equation}\label{eq:IntCondPot}
	\log \| \cL_\om^{M_{\om,N_*}}\ind \|_\infty \in L^1_m(\Om),\quad \log \essinf g_\om^{(M_{\om,N_*})} \in L^1_m(\Om).
	\end{equation}
	
	Let us assume $\phi, \phi_-, \phi_+\in\RR$ are as in Lemma~\ref{lem:excontPot}. Then
	$$
	M_{\om,N_*} \phi_-\leq \log \essinf g_\om^{(M_{\om,N_*})}\leq \log \esssup g_\om^{(M_{\om,N_*})} \leq M_{\om,N_*} \phi_+.
	$$ 
	Also, recalling that the map $x\mapsto \beta_\om x \pmod{1}$ has $b_\om=\ceil{\beta_\om}$ branches, we have that
	\[
	\log \essinf g_\om^{(M_{\om,N_*})} \leq    
	\log \| \cL_\om^{M_{\om,N_*}}\ind \|_\infty \leq  \log \esssup  (b_\om^{(M_{\om,N_*})}g_\om^{(M_{\om,N_*})}), 
	\]
	where $b_\om^{(n)}=\prod_{j=0}^{n-1} b_{\sigma^j\om}$.
	Therefore,
	\[
	M_{\om,N_*} \phi_- \leq    
	\log \| \cL_\om^{M_{\om,N_*}}\ind \|_\infty \leq  \log ( b_\om^{(M_{\om,N_*})}) + {M_{\om,N_*}} \phi_+. 
	\]
	Thus, under the hypotheses of the lemma, conditions \eqref{eq:IntCondPot} hold.
\end{proof}

In what follows, we discuss explicit conditions which will hold if \eqref{eq:IntCond2} holds.
Notice that once the first condition of \eqref{eq:IntCond2} holds, the second one also holds provided there exists $\beta_+\in \RR$ such that $\beta_\om\leq \beta_+$. Let us investigate the first condition of \eqref{eq:IntCond2} in more detail.


Let us first discuss the case $N_*=1$. That is, assume \eqref{eq: def of N} holds with $N_*=1$.
Let
\begin{equation}\label{eq:som}
s_\om =\tfrac{\{\beta_\om\}_+ }{\beta_\om}
\end{equation}
where $\{\beta\}_+$ denotes the fractional part of $\beta$ if $\beta\notin \NN$ and $\{\beta\}_+=1$ if $\beta\in\NN$.
Then, $s_\om$ is the length of the shortest branch of the monotonicity partition of $T_\om$.
Note that all but at most one of the elements $Z\in \mathcal Z_\om^*$ have the property that $T_\om (Z)=I$, and the remaining one is of the form $Z_\om=[1-s_\om, 1]$.
Thus, it takes at most $\tau_\om(s_\om)$ iterates for each $Z\in \mathcal Z_\om^*$ to cover $I$, where
\begin{equation}\label{eq:covTime}
\tau_\om(s):=\min \{ n\geq 1 : \beta_\om^{(n)} s \geq 1 \}.
\end{equation} 
Then, the covering time satisfies $M_{\om,1}= \tau_\om(s_\om)$. 
The fact that $\tau_\om(s)<\infty$ for every $s>0$ and \maeom \ follows from the expansion on average property.
Thus, in this case, \eqref{eq:IntCond2} holds provided
\[
\int_\Om  \tau_\om(s_\om)\,dm <\infty.
\]
One may construct examples (as done in the following subsection, where the upper bound on $\beta$ is also relaxed) when this occurs, for example by  controlling the occurrence of small pieces $s_\om$ and small values of $\beta_\om$.

When \eqref{eq: def of N} holds with $N_*>1$, 
\eqref{eq:IntCond2} holds provided
\begin{equation}\label{eq:intCovTimeN}
\int_\Om  \tau_\om(s_{\om,N_*}) \,dm <\infty,
\end{equation}
where $s_{\om,N_*}$ is length of the shortest branch of $\mathcal Z_\om^{(N_*)}$.
This length may be characterized inductively as follows: $s_{\om,1}=s_\om$ is given by \eqref{eq:som}.
Assuming $s_{\om,k}$ has been computed, 
\[
s_{\om,k+1} = \min \Big\{ \tfrac{s_{\sigma^k(\om),1}}{\beta_\om^{(k)}}, \tfrac{ \{\beta_{\om}^{(k+1)}s_{\om,k}\}_+ }{\beta_{\sigma^k(\om)}}  \Big \}.
\]

\subsection{Random $\bt$-Transformations (II)}\label{sec:betadelta}

In this subsection we assume that there is some $\dl>0$ such that for $m$-a.e. $\om\in\Om$ we have 
$$	
	\beta_\omega\in \bigcup_{k\in\NN} [k+\delta,k+1].
$$ 
We consider the potential $\varphi_\omega=-t\log|T'_\omega|$, $t\ge 0$.
Note that in this setting we have relaxed the uniform lower bound on $\phi_\om$ and the uniform upper bound on the number of branches of $T_\om$ from the previous subsection.

\begin{lemma}
	\label{uniformcoveringtime}
	Under the above conditions, the time for an interval of monotonicity of $T_\omega$ to cover is no more than $1+\left\lceil \frac{-\log\delta}{\log(1+\delta)}\right\rceil$.
\end{lemma}
\begin{proof}
	If the interval of monotonicity is a full branch it takes one iteration.
	Otherwise, we must be at the final (non-full) branch and therefore after one iteration we have an interval $[0,\delta)$.
	At each subsequent iteration this interval grows by a factor at least $1+\delta$.
	If the interval is cut at some future point, it means the interval has expanded across a full branch and we have covered.
	In the worst case, we repeatedly apply the map with least slope, namely $1+\delta$ and it takes $n$ further iterations beyond the first, where $\delta(1+\delta)^n\ge 1$.
	Solving for $n$ yields the claim.
\end{proof}

\begin{lemma}
	\label{lem:conds1}
	Under the $m$-continuity conditions of Lemma \ref{lem:excontPot} and the conditions on the maps in this subsection, additionally assume that
	\begin{equation}
	\label{log9}
	\int_\Om \log\lfloor\beta_\omega\rfloor\ dm(\omega)>\log 3
	\end{equation}
	and
	\begin{equation}
	\label{blb}
	\int_\Om\log\lceil\beta_\omega\rceil\ dm(\omega)<\infty.
	\end{equation}
	Then $\phi$ is a summable contracting potential, and the system satisfies 
	conditions \eqref{cond M1}--\eqref{cond M4} and \eqref{cond M5'}--\eqref{cond M6'}. Thus Theorems~\ref{main thm: summary existence of measures and density} -- \ref{main thm: eq states} apply.
\end{lemma}

\begin{proof}
	We first note that \eqref{cond SP1}, \eqref{cond SP2}, and hold since $g_\omega^{(1)}$ is constant for $m$-a.e.\ $\omega\in\Om$ and that
	\eqref{cond SP3} holds since $S^{(1)}_\omega=\frac{\lceil\beta_\omega\rceil}{\beta_\omega^t}<\infty$.
	Now we note that we may take the finite partition $\cP_{\om,n}(0,1)=\cZ_\om^{(n)}$. Indeed, condition \eqref{cond P1} holds with $\hal=0$ because $g_\omega^{(n)}$ is constant for $m$-a.e.\  $\omega\in\Om$, and Since $\mathcal{Z}_\omega^{(n)}$ is finite for $m$-a.e.\ $\omega\in\Om$, by Remark~\ref{rem: when partitions Z_om^n=P_om,n}, \eqref{cond P2} holds with $\hgm=1$.
	Conditions \eqref{cond M1} and \eqref{cond M2} hold by $m$-continuity, see \cite{FLQ2}, and
	\eqref{cond M3} holds since $\log S^{(1)}_\omega=\log\lceil\beta_\omega\rceil-t\log\beta_\omega$, using (\ref{blb}) and noting $\beta_\omega\ge 1+\delta$.
	Similarly, \eqref{cond M4} holds since
	$$
	\log\lfloor\beta_\omega\rfloor-t\log\beta_\omega\le \log \inf \mathcal{L}_\omega 1\le  S^{(1)}_\omega,
	$$
	again using (\ref{blb}) and noting $\beta_\omega\ge 1+\delta$.
	To show $\varphi$ is contracting we apply (\ref{eq: on avg CPN1N2}) with $N_1=N_3=1$.
	One has $$\int \log|g_\omega|_\infty\ dm=-t\int\log\beta_\omega\ dm<\int \log\lfloor\beta_\omega\rfloor\ dm-t\int\log\beta_\omega\ dm\le \int \log\inf\mathcal{L}_\omega\ind_\omega\ dm,$$
	with the central inequality holding by (\ref{log9}).
	
	For \eqref{cond M5'} and \eqref{cond M6'} we first show that $N_*=1$.
	The quantity $Q_\omega^{(1)}$ defined in (\ref{eq: def of Q and K}) is bounded between $3/\lceil\beta_\omega\rceil$ and  $3/\lfloor\beta_\omega\rfloor$.
	Recalling that $\beta_\omega\ge 1+\delta$ for $m$-a.e.\ $\omega\in\Om$ and using \eqref{log9}--\eqref{blb}, we obtain
	$$
	-\infty <\int_\Om\log Q^{(1)}_\omega \, dm<0,
	$$
	and therefore it follows from (\ref{eq:  def of N}) that we may take $N_*=1$.
	By Remark \ref{rem: when partitions Z_om^n=P_om,n} and the fact that $N_*=1$ we may check \eqref{cond M5'} and \eqref{cond M6'} for $\mathcal{P}_{\omega,1}=\mathcal{Z}_{\omega}^{(1)}$.
	Thus, to verify \eqref{cond M5'} and \eqref{cond M6'} we need to check that
	\begin{equation}
	\label{M5'check}
	\min_{P\in \mathcal{Z}_{\omega}^{(1)}}\log\inf_P g_\omega^{M_\omega(P)}=\min_{P\in \mathcal{Z}_{\omega}^{(1)}}\inf_P -t\log |(T_\omega^{(M_\omega(P))})'| \in L^1_m(\Omega)
	\end{equation}
	and
	\begin{equation}
	\label{M6'check}
	\max_{P\in \mathcal{Z}_{\omega}^{(1)}}\log\| \mathcal{L}_\omega^{M_\omega(P)}\ind_\omega\|_\infty\in L^1_m(\Omega),
	\end{equation}
	respectively.
	By Lemma \ref{uniformcoveringtime}, $$M_\omega(P)\le 1+\left\lceil \frac{-\log\delta}{\log(1+\delta)}\right\rceil.$$
	Thus (\ref{M5'check}) will hold if
	$$\log\beta_\omega^{\left(1+\left\lceil \frac{-\log\delta}{\log(1+\delta)}\right\rceil\right)}\in L^1_m(\Omega),$$
	or equivalently, if $\int\log\beta_\omega\ dm(\omega)<\infty$, which is the case by (\ref{blb}).
	
	Concerning (\ref{M6'check}),
	\begin{eqnarray}
	\nonumber|\log \mathcal{L}_\omega^{M_\omega(P)}\ind_\omega|&\le& \left|\log\left(\frac{\lceil\beta_\omega\rceil}{\beta_\omega^t}\right)^{(M_\omega(P))}\right|\\
	\nonumber&\le&\left|\log\left(\frac{\lceil\beta_\omega\rceil}{\beta_\omega^t}\right)^{\left(1+\left\lceil \frac{-\log\delta}{\log(1+\delta)}\right\rceil\right)}\right|\\
	\nonumber&=& \left|\log\lceil\beta_\omega\rceil^{\left(1+\left\lceil \frac{-\log\delta}{\log(1+\delta)}\right\rceil\right)}- t\log\beta_\omega^{\left(1+\left\lceil \frac{-\log\delta}{\log(1+\delta)}\right\rceil\right)}\right|\\
	\label{M6eqn}&\le&(1+t)\left|\log\lceil\beta_\omega\rceil^{\left(1+\left\lceil \frac{-\log\delta}{\log(1+\delta)}\right\rceil\right)}\right|.
	\end{eqnarray}
	As in the case of \eqref{cond M5'} above, integrability of (\ref{M6eqn}) is equivalent to $\int_\Om\log\lceil\beta_\omega\rceil\ dm(\omega)<\infty$, which follows from (\ref{blb}).
	
	
\end{proof}

\begin{remark}
	Note that we could replace the requirement that $\int_\Om \log\lfloor\beta_\omega\rfloor\ dm(\omega)>\log 3$ in the previous lemma
	with the requirement that $\bt_\om>1+\dl$ for $m$-a.e. $\om\in\Om$ at the expense of calculating the number $N_*>1$ to satisfy \eqref{cond CP2}, and checking \eqref{cond M5'} and \eqref{cond M6'}. 
	Furthermore, note that \eqref{cond M5'} and \eqref{cond M6'}  are automatically satisfied in the setting of finitely many maps $T_\om$.
\end{remark}

%
%
%
%

\subsection{Randomly Translated Random $\bt$-Transformations}\label{sec:betaTrans}
Let $\alpha, \beta: \Omega \to \RR_+$, be $m$-continuous functions such that $\beta_\om$ is essentially bounded and satisfies the expanding condition $\essinf_{\om \in \Om} \beta_\om = 2 + \delta$, for some $\delta>0$.  
Consider the random beta transformation
$T_\om : [0,1] \to [0,1]$ given by
\[
T_\om(x) = \beta_\omega x + \alpha_\om \pmod{1}.
\]
\begin{lemma}\label{lem:exBetaPlusTrans}
	Let $\phi:\Om \times I \to \RR$ be a potential satisfying the conditions of Lemma~\ref{lem:excontPot}.
	Then, Theorems~\ref{main thm: summary existence of measures and density} -- \ref{main thm: eq states} hold for the random beta transformation $(T_\om)_{\om \in \Om}$, provided 
	\begin{equation}\label{eq:IntCond3}
	\log s_{\om,N_*} \in  L^1_m(\Om),
	\end{equation}
	where $N_*$ is defined in \eqref{eq: def of N} and $s_{\om,N_*}$ is length of the shortest branch of $\mathcal Z_\om^{(N_*)}$.
\end{lemma}
The proof of this lemma relies on the following  consequence of  a result of Conze and Raugi \cite[Lemma 3.5]{ConzeRaugi}. 
\begin{lemma}\label{lem:covTimeBetaPlusTrans}
	Let $s_{\om,N}$ be the length of the shortest branch of $\mathcal Z_\om^{(N)}$, 
	$M_{\om,N}$ be the maximum of the covering time for the intervals of  $\mathcal Z_\om^{(N)}$
	under the sequence of maps $\{T_\om^{n}\}_{n\in\NN}$, and $$C=1+\tfrac{2}{\delta}=\sum_{r=0}^\infty \lt(\frac{2}{2+\delta} \rt)^r.$$ 
	Then we have that 
	\[
	M_{\om,N} \leq  \ceil{\frac{-\log \tfrac {s_{\om,N}} {C}}{\log (1+\tfrac{\delta}{2})}} + N.
	\]
\end{lemma}
\begin{proof}
	Let  $0<\ep<s_{\om,N}$.
	The proof of \cite[Lemma 3.5]{ConzeRaugi} shows that
	whenever $C\big(\frac{2}{2+\delta}\big)^r\leq \ep$, or equivalently 
	$$
		r\geq\ceil{\frac{-\log \tfrac \ep C}{\log (1+\tfrac{\delta}{2})}},
	$$
	then $[0,1]$ is covered, up to a set of measure $\ep$, by full branches for the maps $T^N_\om$, $T_\om^{N+1}$, $\dots$, $T_\om^{N+r-1}$. 
	Note that this implies that if $Y\in \mathcal Z_\om^{(N)}$ (and so $Y$ has measure greater than $\ep$), and 
	$$
	R\geq N+\ceil{\frac{-\log \tfrac \ep C}{\log (1+\tfrac{\delta}{2})}},
	$$
	then  $Y$ is a union of branches of 
	$\mathcal Z_\om^{(R)}$, and so it must contain a full branch for $T_\om^{R}$.
	Thus, $M_{\om,N} \leq R$.
The conclusion follows from letting $\ep \uparrow s_{\om,N}$.
\end{proof}
\begin{proof}[Proof of Lemma~\ref{lem:exBetaPlusTrans}]
	The proof goes as in Section~\ref{sec:beta}.
	Condition \eqref{cond GP} holds because the system is expanding on average. Hence, \eqref{cond RC} follows from Lemma~\ref{lem:covTimeBetaPlusTrans} because the partitions are finite (see Remark~\ref{rmk:coveringWFinitePart}).
	
	Let us assume $\phi, \phi_-, \phi_+\in\RR$ are as in Lemma~\ref{lem:excontPot}. 
	The summability conditions  (Definition~\ref{def: summable potential}) are straightforward to check. Indeed, \eqref{cond SP2} is a hypothesis. \eqref{cond SP1} and \eqref{cond SP3} hold because $\phi$ is bounded and the partitions are finite. 
	When $g_\om=|T'_\om|^{-t}$ for some $t\geq 0$,
	\eqref{cond P1} holds with $\hal=0$ because $g_\om^{(n)}$ is constant for each $\om \in \Om$. In general,
	\eqref{cond P2} holds with $\hgm=1$ because $\cP_{\om,n}(0,1)$ coincides with the dynamical partition $\cZ_\om^{(n)}$.
	
	Conditions
	\eqref{cond M1} and \eqref{cond M2} follow from $m$-continuity; see e.g. \cite{FLQ2}.
	\eqref{cond M3} holds because 
	the map $x\mapsto \beta_\om x + \al_\om \pmod{1}$ has at most $\tilde b_\om:=\ceil{\beta_\om}+1$ branches. Then,
	$e^{\phi_-}\leq S_{\om}^{(1)} \leq \tilde b_\om e^{\phi_+}$. 
	Since $\beta_\om$ is essentially bounded, say by $B$,
	$$ 
	\int_\Om | \log S_{\om}^{(1)} |\, dm \leq \max \set{|\phi_-|, \log (B+2) + |\phi_+| }<\infty.
	$$ 
	Notice that $\inf \cL_\om \ind_\om \geq n(T_\om) e^{\phi_-}$,
	where $n(T_\om) =  \min  \{ \#  T_\om^{-1}(x) :  x\in I \}$ is the minimum number of preimages of a point $x\in I$ under $T_\om$. Since $T_\om$ is a shifted $\bt$-transformation, 
	$
	n(T_\om) =  \floor{\beta_{\om}}.
	$
	Thus, in view of Remark~\ref{rmk:checkCP}, with $N_1=N_3=1$, in order to show the potential is contracting (Definition~\ref{def: contracting potential}), it is sufficient to have that 
	\begin{equation*}\label{eq:excontPot2}
	\phi_+<  \phi_- +\int_\Om  \log \floor{\beta_{\om}}\, dm,
	\end{equation*}
	as assumed.
	Note that this condition ensures that $\floor{\beta_\om}\geq1$ for \maeom, so $\beta_\om\geq 1$ for \maeom, and so $\log \inf \cL_\om \ind \geq \phi_-$. This yields \eqref{cond M4}.
	Finally, we show 
	conditions \eqref{cond M5'} and \eqref{cond M6'}, namely
	\begin{equation}\label{eq:IntCondPot2}
	\log \| \cL_\om^{M_{\om,N_*}}\ind \|_\infty \in L^1_m(\Om),\quad \log \essinf g_\om^{(M_{\om,N_*})} \in L^1_m(\Om).
	\end{equation}
	Note that
	$$
	M_{\om,N_*} \phi_-\leq \log \essinf g_\om^{(M_{\om,N_*})}\leq \log \esssup g_\om^{(M_{\om,N_*})} \leq M_{\om,N_*} \phi_+.
	$$ 
	Recalling that the map $x\mapsto \beta_\om x + \al_\om \pmod{1}$ has at most $\tilde b_\om:=\ceil{\beta_\om}+1$ branches, we get 
	\[
	\log \essinf g_\om^{(M_{\om,N_*})} \leq    
	\log \| \cL_\om^{M_{\om,N_*}}\ind \|_\infty \leq  \log \esssup  (\tilde b_\om^{(M_{\om,N_*})}g_\om^{(M_{\om,N_*})}), 
	\]
	where $\tilde b_\om^{(n)}=\prod_{j=0}^{n-1} \tilde b_{\sigma^j\om}$.
	Therefore,
	\[
	M_{\om,N_*} \phi_- \leq    
	\log \| \cL_\om^{M_{\om,N_*}}\ind \|_\infty \leq  \log (\tilde b_\om^{(M_{\om,N_*})}) + {M_{\om,N_*}} \phi_+. 
	\]
	Lemma~\ref{lem:covTimeBetaPlusTrans} shows that 
	$\int_\Om M_{\om,N_*} \,dm <\infty$ whenever
	\[
	\int_\Om | \log s_{\om,N_*} |\,dm <\infty,
	\]
	which is a hypothesis of the lemma.
	Recalling that $\beta_\om\leq B$, we get
	\[
	\int  \log \tilde b_\om^{(M_{\om,N_*})} \,dm\leq \int M_{\om,N_*} \log (B+2) \,dm <\infty.
	\]
	Hence, \eqref{eq:IntCondPot2} is satisfied.
\end{proof}

\begin{remark}
	We could use similar techniques to show that Theorems~\ref{main thm: summary existence of measures and density}--\ref{main thm: eq states} apply for maps with nonlinear branches. For example,  let $S_\om: I \to \RR$ be a $C^2$ expanding map, with 
	$$
		2+\delta < \essinf |S_\om'| \leq \esssup |S'_\om| < B
	$$  
	for \maeom, 
	for some $0<\delta,B$,	and let
	$\tilde T_\om(x)= S_\om(x) \pmod{1}$. Each map  $\tilde T_\om(x)$ has at most two non-full branches (the leftmost and rightmost ones). For simplicity, assume $N_*=1$.
	Lemma~\ref{lem:covTimeBetaPlusTrans} (with $N=1$) relies on \cite[Lemma 3.5]{ConzeRaugi} for the case $n=1$ only, and an inspection of the proof 
	shows that its conclusions remain valid for $\tilde T_\om$ instead of $T_\om$. 
	Then, the proof of Lemma~\ref{lem:exBetaPlusTrans} remains applicable  for $\tilde T_\om$ with minor modifications, after replacing the hypothesis of Lemma~\ref{lem:excontPot} involving $E_-$ with
	$$
		\phi_+ - \phi_-< \int_\Om \log  \essinf |S_\om'| \, dm,
	$$ 
	and ensuring hypothesis $(ii)$ of Lemma~\ref{lem:excontPot} is satisfied.
	These modifications essentially amount to changing $\floor{\beta_\om}$ and $\ceil{\beta_\om}$ to $\essinf |S_\om'|$ and  $\esssup |S'_\om|$.
\end{remark}

\subsection{Random Maps With Infinitely Many Branches}\label{sec:infiniteBranches}

In this section we consider compositions of the Gauss map $T_0$ and the Renyi map $T_1$, defined by 
\[
T_0(x) = \begin{cases}
0 & x=0,\\
\tfrac1x \pmod{1}& x\neq 0,
\end{cases}
\qquad \text{and} \qquad
T_1(x) = \begin{cases}
0 & x=1,\\
\tfrac{1}{1-x} \pmod{1}  & x\neq 1.
\end{cases}
\]
This class of random maps was investigated in \cite{KalleEtAl} in connection with continued fraction expansions.
Let $T:\Om \to \{0,1\}$ be a measurable function. Let $T_\om=T_0$ if $T(\om)=0$ and $T_\om=T_1$ otherwise.
We denote by $\cL_i$ the transfer operator associated to $T_i$, $i=0,1$.
Since both $T_0$ and $T_1$ have only full branches, the system is covering, with $M_{\om,n}=1$ for every $\om\in \Om, n\in \NN$.

\begin{lemma}
	Consider a potential of the form $\phi(\om,x)= - t \log |T_\om'(x)|$, for $0.5<t \leq 0.613$. Then,  
	the hypotheses of Theorems~\ref{main thm: summary existence of measures and density} -- \ref{main thm: eq states} 
	hold.
\end{lemma}

\begin{proof}
	We will use the shorthand notation $\phi_i(x)=\phi(\om,x)$ whenever $T_\om=T_i$, and similarly for $g_i(x)$.
	Note that since at the differentiability points, $|T_0'(x)|=\frac{1}{x^2}$, then  
	$$
	g_0(x)=x^{2t},
	\qquad 
	|\log g_0(x)|= |\phi_0(x)|= 2 t  \log  \frac{1}{x},
	$$
	and for each partition element $J_j=[\tfrac{1}{j+1},\tfrac{1}{j}]$, 
	$$
	| \log \essinf_{J_j} g_\om| =  2t\log (j+1).
	$$ 
	The analysis for $T_1$ is entirely analogous, since $T_1(x)=T_0(1-x)$.
	
	Then,
	\begin{enumerate}
		
		\item The summability conditions (Definition~\ref{def: summable potential}) are satisfied. \eqref{cond SP1} and \eqref{cond SP2} follow directly from the definition of $g_\om$.  \eqref{cond SP3} follows from the fact that 
		$$
		\sum_{j=1}^\infty \sup_{J_j} g_\om = \sum_{j=1}^\infty \frac{1}{j^{2t}}<\infty.
		$$
		The first part of \eqref{cond M3} holds because $\sup \phi_\om=0$, the second part follows from \eqref{cond SP3}, because $S_\om^{(1)}$ can only take two different values. 
		Also, \eqref{cond M1} and  \eqref{cond M2} hold because $T$ is measurable and of finite rank.
		
		\,
		
		\item
		Conditions \eqref{cond P1} and \eqref{cond P2}, with $\hat\al=\hgm=1$, are satisfied for $n=1$, with 
		partitions $\mathcal P_{0,1}(1,1)=\mathcal P_{1,1}(1,1)$ given by $\mathcal P_i=\{ Z_1^i, Z_2^i, \dots, Z_k^i, Z^i_\infty \}$, for $i=0,1$, where $k=k(t)$ is chosen such that  $ \sum_{j=k+1}^\infty \frac{1}{j^{2t}}\leq1$,  $Z^0_j=[\tfrac{1}{j+1}, \tfrac1j]$ for $1\leq j \leq k$, $Z^0_\infty=[0,\frac{1}{k+1}]$, and $Z^1_j$ obtained from $Z^0_j$ by the transformation $x\mapsto 1-x$.
		Indeed $g_\om$ is monotonic on $[0,1]$ and $\var_{[0,1]} g_\om =1$, so \eqref{cond P1} holds with $\hal=1$ regardless of the partition.
		The choice of $k$ precisely ensures condition \eqref{cond P2} holds with $\hgm=1$ and $n=1$. 
		
		\,
		
		\item 
		\eqref{cond CP1} holds, i.e. $\phi$ is a contracting potential. In fact, we check that condition \eqref{eq: def of N} holds with $N_*=1$, i.e. we have that
		$$ \int_\Om \log \frac{4\|g_\om\|_{\infty}}{\inf \cL_\om\ind}\,dm <0.$$ 
		Indeed,  $\| g_0(x)\|_\infty =1$ and $\cL_0\ind(x)= \sum_{j=1}^\infty \frac{1}{(j+x)^{2t}}$,
		so  
		\[
		\frac{4\|g_0\|_{\infty}}{\inf \cL_0 \ind} = \frac{4}{\sum_{j=2}^\infty \frac{1}{j^{2t}}} = \frac{4}{\zeta(2t)-1}<1,
		\]
		because  $\zeta(1.226)> 5$, where $\zeta(x)$ is the Riemann zeta function.
		For $T_\omega=T_1$, $g_1(x)=g_0(1-x)$ so the bound is the same.
		
		\,

		\item 		
		In light of Remark~\ref{rem: cond M5' and M6' and tilde N*} we check that conditions \eqref{cond M5'}, \eqref{cond M6'} hold for $N_*=1$.
		For each $\om\in \Om$, the elements of  $\mathcal P_\om$ consist of (unions of) full branches for $T_\om$. Hence, $N_{\omega, 1}(P)=1$ (coming from \eqref{eq: def of J(P)}) and $M_\om(J(P))=1$ (coming from \eqref{cond RC}) for every  $P\in \mathcal P_\om$. 
		In particular, for every $\om\in \Omega$, 
		$$
		\zeta(2t)-1\leq \inf \cL_\om\ind \leq \| \cL_\om\ind\|_\infty \leq \zeta(2t)<\infty
		$$ 
		for $t>\tfrac12$, so the  integrability condition \eqref{cond M6'} is satisfied.
		Also, by choosing $J(P)$ to be the largest full branch in $P$, we get that
		$$
		\log \essinf_{J(P)} g_\om \geq - 2t \log (k+2)
		$$ 
		for each $\om\in \Om$.
		Hence the  integrability condition \eqref{cond M5'} is satisfied.
		
	\end{enumerate}
\end{proof}
\begin{remark}
	Note that we have only chosen $t\in(0.5,0.613]$ to ensure that we have $N_*=1$. One could treat larger values of $t$, in particular $t\geq1$, by considering $N_*>1$. 
\end{remark}

\subsection{Random Non-Uniformly Expanding Maps}\label{sec:non-unifexp}
In this section we consider random compositions of a  $\bt$-transformation $T_0$ and the Pomeau-Manneville intermittent map $T_1$, defined by 
\begin{equation*}
T_0(x)=\bt x \mod 1 
\quad\text{ and }\quad
T_1(x)=
\begin{cases}
x+2^{a}x^{1+a} &  \text{ for } x\in[0,1/2]\\
2x-1 & \text{ for } x\in(1/2,1]
\end{cases}
\end{equation*}
where $a>0$ and $\bt\geq 5$. 
The map $T_1$, in this form, was first introduced in  \cite{liverani_probabilistic_1999}, but dates back to the works of \cite{manneville_intermittency_1980,pomeau_intermittent_1980}.
Sequential and random compositions of $T_1$, and maps similar to $T_1$, taken with varying values for $a$, have been  investigated previously, see e.g. \cite{aimino_polynomial_2015,nicol_central_2018}.

Let $T:\Om \to \{0,1\}$ be a measurable function. Let $T_\om=T_0$ if $T(\om)=0$ and $T_\om=T_1$ otherwise. 
We denote by $\cL_i$ the transfer operator associated to $T_i$, $i=0,1$.
\begin{lemma}\label{lem: ex random beta pm map}
	Consider a potential of the form $\phi(\om,x)= - t \log |T_\om'(x)|$, for $t\geq 0$. If
	\begin{align*}
	0<p=m(T^{-1}(1))< \frac{\log\frac{\floor\bt}{4}}{\log\frac{\floor{\bt}(2a+4)^t}{2^t+(a+2)^t}},
	\end{align*}
	then the hypotheses of Theorems~\ref{main thm: summary existence of measures and density} -- \ref{main thm: eq states} 
	hold.
\end{lemma}
\begin{proof}
	We begin by noting that our random covering assumption \eqref{cond RC} follows from the Birkhoff Ergodic Theorem as for $n$ sufficiently large we will see the map $T_0$ with frequency $(1-p)n$, and thus any small interval will eventually cover.
	As each map has only finitely many branches \eqref{cond SP1}--\eqref{cond SP3} are satisfied.  
	Now, we note that  
	\begin{align*}
	\inf g_0=\norm{g_0}_{\infty}=\frac{1}{\bt^t},
	\qquad
	\inf g_1=\frac{1}{\lt(a+2\rt)^t},
	\quad\text{ and }\quad
	\norm{g_1}_\infty=1.
	\end{align*}
	It then follows that 
	\begin{align}
	\frac{\floor{\bt}}{\bt^t}\leq \inf\cL_0\ind\leq \norm{\cL_0\ind}_\infty\leq \frac{\ceil{\bt}}{\bt^t}\leq \ceil{\bt},
	\label{eq: ex check M6' T0}
	\end{align}
	and 
	\begin{align}
	\frac{2^t+(a+2)^t}{(2a+4)^t}
	=
	\frac{1}{\lt(a+2\rt)^t}+\frac{1}{2^t}
	\leq
	\inf \cL_1\ind 
	\leq 
	\norm{\cL_1\ind}_\infty 
	\leq 
	1+\frac{1}{2^t}
	\leq 2.
	\label{eq: ex check M6' T1}
	\end{align}
	Thus conditions \eqref{cond M1}--\eqref{cond M4} are satisfied.
	Concerning the variation, we note that $\var_Z(g_0)=0$ for each $Z\in\cZ_0^{(1)}$, $\var_{[0,1/2]}(g_1)= 1$, $\var_{[1/2,1]}(g_1)=0$.
	Thus, in view of Remark~\ref{rem: when partitions Z_om^n=P_om,n}, the partition $\cP_{i,1}(1,1)=\cZ_{i}^{(1)}$ satisfies \eqref{cond P1} and \eqref{cond P2} with $\hat\al=\hgm=1$ for each $i\in\set{0,1}$.  
	
	We claim now that we can take the number $N_*$, coming from \eqref{eq: def of N}, to be 1.
	Indeed, since 
	\begin{align}
	\int_\Om\log\frac{4\norm{g_\om}_\infty}{\inf\cL_\om\ind_\om}\, dm(\om)
	&=
	(1-p)\log\frac{4\norm{g_0}_\infty}{\inf\cL_0\ind}
	+
	p\log\frac{4\norm{g_1}_\infty}{\inf\cL_1\ind}
	\nonumber\\
	&\leq 
	(1-p)\log\frac{4/\bt^t}{\floor{\bt}/\bt^t}
	+
	p\log\frac{4}{(2^t+(a+2)^t)/(2a+4)^t}
	\nonumber\\
	&=
	(1-p)\log\frac{4}{\floor{\bt}}
	+
	p\log\frac{4(2a+4)^t}{2^t+(a+2)^t},\label{eq: ex p inequal}
	\end{align}
	\eqref{eq: def of N} holds whenever
	\begin{align}\label{eq: ex p value}
	p<\frac{\log\frac{\floor\bt}{4}}{\log\frac{\floor{\bt} (2a+4)^t}{2^t+(a+2)^t}}.
	\end{align}
	Thus, as we have chosen $p>0$ satisfying \eqref{eq: ex p value}, we see that 
	\begin{align}
	\int_\Om\log\frac{4\norm{g_\om}_\infty}{\inf\cL_\om\ind_\om}\, dm(\om)<0.
	\label{eq: ex check CP1 CP2}
	\end{align}
	In light of Remark~\ref{rem: cond M5' and M6' and tilde N*}, we now check conditions \eqref{cond M5'} and \eqref{cond M6'}. To 
	that end, we note that \eqref{cond M5'} holds since $\inf_Z g_0=1/\bt^t$ for each $Z\in\cZ_0^{(1)}$, $\inf_{[0,1/2]}g_1=1/(a+2)^t$, and $\inf_{[1/2,1]}g_1=1/2^t$, and \eqref{cond M6'} follows from \eqref{eq: ex check M6' T0} and \eqref{eq: ex check M6' T1}. Finally, we see that \eqref{cond CP1} and \eqref{cond CP2} hold in light of \eqref{eq: ex check CP1 CP2} and Remark~\ref{rmk:checkCP}.
\end{proof}
\begin{remark}\label{rem: betq geq 5}
	Notice that our choice of $\bt\geq 5$ was instrumental in obtaining $N_*=1$. Using similar arguments one could yield the conclusions of Lemma~\ref{lem: ex random beta pm map} with any $\bt\geq 2$ and $N_*=3$ for $p>0$ sufficiently small. 
	
	Furthermore, it is interesting to point out that our results hold for any positive value of the parameter $a$ defining the intermittent map $T_1$. This is in contrast to the deterministic setting (of the map $T_1$ together with potential $\phi=-\log |T_1'|$), where the density with respect to the conformal measure, which is the Lebesgue measure, for the map $T_1$ is Lebesgue summable only for $a<1$. 
		
	Also note that the upper bound on $p$ tends to zero as $t$ tends to infinity. 
\end{remark}
\begin{remark}
	Here we have chosen, for the ease of computation and exposition, to present Example~\ref{sec:non-unifexp} as the random iteration of two maps each with positive probability, however, assuming the $m$-continuity conditions of Lemma \ref{lem:excontPot}, one could easily apply similar arguments to the setting where one considers maps of the form 
	\begin{equation*}
	T_\om(x)=
	\begin{cases}
	\bt_\om x \mod 1 & \text{ if } \om\in\Om_0\\
	\, & \\
	\begin{cases}
	x+2^{a_\om}x^{1+a_\om} &  \text{ for } x\in[0,1/2]\\
	2x-1 & \text{ for } x\in(1/2,1]
	\end{cases}
	& \text{ if } \om\in\Om_1	
	\end{cases}
	\end{equation*}
	where $\Om=\Om_0\cup\Om_1$ is infinite with $\Om_0\cap\Om_1=\emptyset$ such that $m(\Om_0),m(\Om_1)>0$, and where the maps $\om\mapsto a_\om\in(0,a^*]$, for some $a^*>0$, $\om\mapsto\bt_\om\in [5,\infty)$. 
\end{remark}

Note that the ``intermittency'' of the map $T_1$ in the previous lemma is a bit misleading. We chose $p$ so that $T_1$ was applied sufficiently infrequently to achieve a map cocycle that was expanding on average. This same technique allows us to deal with maps with contracting branches\footnote{The maps are still required to be onto, however some branches may be contracting.}. Indeed, if we consider the maps 
\begin{equation*}
T_0(x)=\bt x \mod 1 
\quad\text{ and }\quad
T_1(x)=
\begin{cases}
ax &  \text{ for } x\in[0,1/2]\\
(2-a)x-(1-a) & \text{ for } x\in(1/2,1],
\end{cases}
\end{equation*}
with $0<a<1$ and $\bt\geq 5$\footnote{The same argument given in Remark~\ref{rem: betq geq 5} allows us to replace $\bt\geq 5$ with $\bt\geq 2$ at the expense of taking $N_*=3$.}, in the same setup as the previous example, then using similar reasoning we arrive at the following lemma.
\begin{lemma}
	Consider a potential of the form $\phi(\om,x)= - t \log |T_\om'(x)|$, for $t\geq 0$. If 
	\begin{align*}
	0<p=m(T^{-1}(1))<\frac{\log\frac{\floor\bt}{4}}{\log\frac{\floor\bt(2-a)^t}{a^t}},
	\end{align*}
	then the hypotheses of Theorems~\ref{main thm: summary existence of measures and density} -- \ref{main thm: eq states} 
	hold.
\end{lemma}
\begin{remark}
	As with Lemma~\ref{lem: ex random beta pm map}, we note again that $p$ must go to zero as $t$ tends toward infinity.
\end{remark}

\subsection{Random Lasota-Yorke maps}

In this final example we discuss the very broad class of random piecewise-monotonic maps, where each fiber map $T_\omega$ is a finite-branched map which takes the interval onto itself.
We consider a general potential $\varphi\in\BV_\Om(I)$.
Recall that $\#\cZ_\om^{(1)}$ denotes the number of branches of monotonicity of $T_\omega$, and for each $n\in\NN$ define $\cI_\om^{(n)}:=\min_{y\in [0,1]}\#\{T^{-n}_\omega (y)\}$.
\begin{lemma}
	\label{lem:RLY1}
	Under the $m$-continuity conditions of Lemma \ref{lem:excontPot}, we assume the integrability conditions $\log \cI_\om^{(1)}$, $\log\#\cZ_\om^{(1)}$, $\inf \phi_\om$, $\sup \phi_\om\in L^1_m(\Omega)$, and that
	\begin{equation}
	\label{RLYcontract}
	\int_\Om (\sup \phi_\om - \inf \phi_\om) \, dm(\om) < \int_\Om \log\cI_\om^{(1)} \, dm(\om).
	\end{equation}
	Then conditions \eqref{cond M1}--\eqref{cond M4} hold and the potential $\phi$ is summable and contracting, i.e. conditions \eqref{cond SP1}--\eqref{cond SP3} and \eqref{cond CP1}--\eqref{cond CP2} hold.
\end{lemma}
\begin{proof}
	In light of Remark~\ref{rem: SP1 - SP3 hold when finite} we see that (\ref{cond SP1})--(\ref{cond SP3}) hold since $\#\cZ_\om^{(1)}<\infty$ and $\phi\in\BV_\Om(I)$.
	Conditions (\ref{cond M1}) and (\ref{cond M2}) hold by $m$-continuity, see \cite{FLQ2}.
	For conditions (\ref{cond M3})--(\ref{cond M4}), we note that because
	$$
		\log \cI_\om^{(1)}+\inf \phi_\om\le \log\inf\mathcal{L}_\omega\ind\le \log S^{(1)}_\omega\le \log\#\cZ_\om^{(1)}+\sup \phi_\om,
	$$ 
	these conditions hold by the integrability hypotheses on the terms in the above inequalities.
	Using \eqref{RLYcontract}, we have
	\begin{align*}
		\int_\Om\log\norm{g_\omega}_\infty\, dm(\om)
		&= 
		\int_\Om\sup\phi_\om\, dm(\om)
		\\
		&<
		\int_\Om\log \cI_\om^{(1)}\, dm(\om)+\int_\Om\inf\phi_\om\, dm(\om)
		\le
		\int_\Om\log\inf\mathcal{L}_\omega\ind\, dm(\om),
	\end{align*}
	and thus, applying \eqref{eq: on avg CPN1N2} with $N_1=N_2=1$, we see that $\varphi$ is a contracting potential.

\end{proof}

\begin{remark}\label{rem: LY example remark N1N2}
Using similar arguments, we could replace \eqref{RLYcontract} in Lemma~\ref{lem:RLY1} with the following: there exists $N\in\NN$ such that $\log\cI_\om^{(N)}\in L^1_m(\Om)$ and 
\begin{align*}
	\int_\Om \sup S_{N, T}(\phi_\om)-\inf S_{N, T}(\phi_\om)\, dm(\om) 
	<
	\int_\Om \log \cI_\om^{(N)}\, dm(\om).
\end{align*}
\end{remark}

\begin{corollary}
	Assuming the hypotheses of Lemma~\ref{lem:RLY1}, if, in addition, we assume the random covering condition \eqref{cond RC} and the general integrable covering conditions (\ref{cond M5}), (\ref{cond M6}), then the hypotheses of Theorems~\ref{main thm: summary existence of measures and density} -- \ref{main thm: eq states}	hold.
\end{corollary}
We now present a final example of the class of random Lasota-Yorke maps focusing on geometric potentials, i.e. potentials of the form $\phi_\om=-t\log\absval{T_\om'}$ for $t\geq 0$. 
We will require the following strengthening of our random covering assumption \eqref{cond RC} which has appeared before in \cite{aimino_concentration_2016} and \cite{haydn_almost_2017}.
\begin{definition}
	We say that the family $(T_\om)_{\om\in\Om}$ satisfies \textit{strong random covering} if there exists $M(n)\in\NN$ such that for any $\om\in\Om$ and any $Z\in\cZ_\om^{(n)}$ we have that $T_\om^{M(n)}(Z)=I$.
\end{definition}
\begin{remark}
	Note that this strong random covering condition is satisfied by the maps considered in Section~\ref{sec:betadelta}. 
\end{remark}
\begin{lemma}\label{lem: LY example}
	Let $\phi_\om=-t\log|T_\om'|$ for $t\geq 0$, and assume the $m$-continuity conditions of Lemma \ref{lem:excontPot}. We further suppose the system satisfies strong random covering as well as the following: 
	\begin{enumerate} 
		\item[\mylabel{1}{hyp 1}] for each $\om\in\Om$, $Z\in\cZ_\om^*$, and $x\in Z$ 
		\begin{enumerate}
			\item $T_\om\rvert_Z\in C^2$, 
			\item there exists $K\geq 1$ such that
			\begin{align*}
			\frac{|T_\om''(x)|}{|T_\om'(x)|}\leq K,
			\end{align*}
		\end{enumerate}	
		\item[\mylabel{2}{hyp 2}] there exists $n_0\in\NN$ and $1< \lm\leq \Lm<\infty$ such that 
		\begin{enumerate}
			\item[\mylabel{a}{hyp 2a}] $|T_\om'|\leq \Lm$ for $m$-a.e. $\om\in\Om$,
			\item[\mylabel{b}{hyp 2b}] $|(T_\om^{n_0})'|\geq \lm^{n_0}$ for $m$-a.e. $\om\in\Om$,
			\item[\mylabel{c}{hyp 2c}] $\frac{1}{n_0}\int_\Om \log \cI_\om^{(n_0)} \,dm(\om)>t\log\frac{\Lm}{\lm}$,
		\end{enumerate}
		\item[\mylabel{3}{hyp 3}] for each $n\in\NN$ there exists 
		$$
		\dl_n:=\inf_{\om\in\Om}\min_{Z\in\cZ_\om^{(n)}}\diam(Z) >0.
		$$
	\end{enumerate}
	Then the hypotheses of Theorems~\ref{main thm: summary existence of measures and density} -- \ref{main thm: eq states} hold.
\end{lemma}
\begin{proof}
	First we note that hypothesis \eqref{hyp 2} implies that we have
	\begin{align}\label{eq: sup inf est of g}
		\Lm^{-kn_0t}\leq g_\om^{(kn_0)}\leq \lm^{-kn_0t}
	\end{align}
	for each $\om\in\Om$ and $k\in\NN$, and therefore $\inf S_{kn_0,T}(\phi_\om)$, $\sup S_{kn_0,T}(\phi_\om)\in L^1_m(\Om)$ for each $k\in\NN$. Furthermore, \eqref{eq: sup inf est of g}, together with hypothesis \eqref{hyp 2}, gives that 
	\begin{align*}
		\frac{1}{n_0}\int_\Om\sup S_{n_0,T}(\phi_\om)-\inf S_{n_0,T}(\phi_\om)\, dm(\om)\leq t\log\frac{\Lm}{\lm}<\frac{1}{n_0}\int_\Om\log\cI_\om^{(n_0)}\,dm(\om).
	\end{align*}
	To see that the last of the hypotheses of Lemma~\ref{lem:RLY1} is satisfied, we note that hypothesis \eqref{hyp 3} implies that there exists $D\in\NN$ such that 
	$$
		\#\cZ_\om^{(1)}\leq D
	$$ 
	for each $\om\in\Om$, which of course implies that $\log\#\cZ_\om^{(1)}\in L^1_m(\Om)$. Thus, Lemma~\ref{lem:RLY1} and Remark~\ref{rem: LY example remark N1N2} imply that $\phi$ is contracting and summable and that conditions \eqref{cond M1}--\eqref{cond M4} hold. Thus, given our strong random covering assumption, we have only to find an appropriate partition $\cP_{\om,n}(\hal,\hgm)$ (satisfying \eqref{cond P1} and \eqref{cond P2}) for which we can check conditions \eqref{cond M5'} and \eqref{cond M6'} for $N_*$. 
	
	To that end, using hypotheses \eqref{hyp 1} and \eqref{hyp 2}, we note that for any $\om\in\Om$, $k\in\NN$, $Z\in\cZ_\om^{(kn_0)}$, and any interval $J\sub Z$ we have that 
	\begin{align}
		\var_J(g_\om^{(kn_0)})
		&\leq 
		2\norm{g_\om^{(kn_0)}}_\infty + \int_J \absval{\left(g_\om^{(kn_0)}(x)\right)'} dx
		\nonumber\\
		&\leq
		2\norm{g_\om^{(kn_0)}}_\infty + t\int_J \sum_{j=0}^{kn_0-1}
		\lt(
		\frac{\absval{T_{\sg^{j}(\om)}''(T_\om^j(x))}}{\absval{T_{\sg^{j}(\om)}'(T_\om^j(x))}}
		\cdot
		\frac
			{\prod_{i=0}^{j-1}\absval{T_{\sg^{i}(\om)}'(T_\om^i(x))}}
			{\prod_{i=0}^{kn_0-1}\absval{T_{\sg^{i}(\om)}'(T_\om^i(x))}^t}
			\rt)\, dx
		\nonumber\\
		&\leq 
		2\norm{g_\om^{(kn_0)}}_\infty + tK\cdot\frac{\sum_{j=0}^{kn_0-1}\Lm^j}{\lm^{nt}}\cdot\diam(J)
		\nonumber\\
		&\leq
		2\norm{g_\om^{(kn_0)}}_\infty + \frac{tK}{\Lm-1}\cdot\lt(\frac{\Lm}{\lm^t}\rt)^{kn_0}\cdot \diam(J).
		\label{eq: var interval est}
	\end{align}
	Now we for each $Z\in\cZ_\om^{(n)}$, we subdivide $Z$ into 
	$$
		v_n:=\ceil{\frac{\Lm}{\lm^t}}^n\cdot\ceil{\frac{\Lm^t}{\lm}}^n
	$$ 
	many pieces $P_{Z,1}, \cdots, P_{Z,v_n}$ of equal length. Since hypotheses \eqref{hyp 2}\eqref{hyp 2b} and \eqref{hyp 3} imply that for each $\om\in\Om$, $k\in\NN$, and $Z\in\cZ_\om^{(kn_0)}$ we have that
	\begin{align*}
		\dl_{kn_0}\leq\diam(Z)\leq \lm^{-kn_0},
	\end{align*}
	then for any $P_{Z,j}\sub Z$ ($1\leq j\leq v_{kn_0}$) we have 
	\begin{align}
		\dl_{kn_0}v_{kn_0}^{-1}\leq
		\diam(P_{Z,j})
		&= 
		\ceil{\frac{\Lm}{\lm^t}}^{-kn_0}\cdot\ceil{\frac{\Lm^t}{\lm}}^{-kn_0}\cdot \diam(Z)
		\leq 
		\ceil{\frac{\Lm}{\lm^t}}^{-kn_0}\cdot\frac{1}{\lm^{kn_0}\lt(\frac{\Lm^t}{\lm}\rt)^{kn_0}}
		\nonumber\\
		&\leq 
		\ceil{\frac{\Lm}{\lm^t}}^{-kn_0}\Lm^{-kn_0t}
		\leq 
		\ceil{\frac{\Lm}{\lm^t}}^{-kn_0}\norm{g_\om^{(kn_0)}}_\infty.
		\label{eq: diam P est}
	\end{align} 
	Thus, inserting \eqref{eq: diam P est} into \eqref{eq: var interval est} gives
	\begin{align}
		\var_{P_{Z,j}}(g_\om^{(kn_0)})
		&\leq 
		2\norm{g_\om^{(kn_0)}}_\infty + \frac{tK}{\Lm-1}\cdot\lt(\frac{\Lm}{\lm^t}\rt)^{kn_0}\cdot \diam(P_{Z,j})
		\nonumber\\
		&\leq 
		2\norm{g_\om^{(kn_0)}}_\infty + \frac{tK}{\Lm-1}\cdot\lt(\frac{\Lm}{\lm^t}\rt)^{kn_0}\cdot \ceil{\frac{\Lm}{\lm^t}}^{-kn_0}\norm{g_\om^{(kn_0)}}_\infty
		\nonumber\\
		&\leq 
		\lt(2+\frac{tK}{\Lm-1}\rt)\norm{g_\om^{(kn_0)}}_\infty.
		\label{eq: var interval est 2}
	\end{align}
	In light of \eqref{eq: var interval est 2} and the fact that $P_{Z,j}\sub Z$ for each $1\leq j\leq v_{kn_0}$, we note that conditions \eqref{cond P1} and \eqref{cond P2} are satisfied for the partition 
	\begin{align*}
		\cP_{\om,kn_0}=\cP_{\om,kn_0}\lt(2+\frac{tK}{\Lm-1}, 1\rt):=\set{P_{Z,j}: 1\leq j\leq v_{kn_0}, \, Z\in\cZ_\om^{(kn_0)}}.
	\end{align*}
	Let $k_*$ be the minimum integer such that
	\begin{align*}
		\int_\Om \log \frac{\lt(5+\frac{tK}{\Lm-1}\rt)\norm{g_\om^{(k_*n_0)}}_\infty}{\inf\cL_\om^{k_*n_0}\ind} \, dm(\om) <0
	\end{align*}
	and let $N_*=k_*n_0$. Now let $k_1\in\NN$ be the first integer such that for each $P\in\cP_{\om,N_*}$ we have
	\begin{align*}
		\lm^{-k_1n_0}\leq \frac{\dl_{N_*}}{2v_{N_*}}\leq\frac{1}{2}\diam(P).
	\end{align*} 
	Thus for each $P\in\cP_{\om,N_*}$ there exists $J\in\cZ_\om^{(k_1n_0)}$ such that $J\sub P$, and by strong random covering we have that $T_\om^{M(k_1n_0)}(J)=I$. 
	Finally, we see that \eqref{cond M5'} and \eqref{cond M6'} are satisfied since for each $\om\in\Om$, each $P\in\cP_{\om,N_*}$, and each $J\sub P$ we have  
	\begin{align*}
		\inf_Jg_\om^{(M(k_1n_0))}\geq \Lm^{-M(k_1n_0)t}
	\end{align*} 	
	and 
	\begin{align*}
		\norm{\cL_\om^{M(k_1n_0)}\ind}_\infty\leq \norm{g_\om^{(M(k_1n_0))}}_\infty D^{M(k_1n_0)}\leq \lm^{-M(k_1n_0)t}D^{M(k_1n_0)}.
	\end{align*}	
\end{proof}
\begin{remark}
	Note that if the set $\set{T_\om:\om \in \Om}$ is finite, i.e. we have only finitely many maps, then hypothesis \eqref{hyp 3} of Lemma~\ref{lem: LY example} always holds.
\end{remark}

\section*{Acknowledgments}
J.A., G.F., and C.G.-T. thank the Centro de Giorgi in Pisa and CIRM in Luminy for their support and hospitality.\\
J.A. is supported by an ARC Discovery project and thanks the School of Mathematics and Physics at the University of Queensland  for their hospitality.\\
G.F., C.G.-T., and S.V. are partially supported by an ARC Discovery Project.\\
S.V. thanks the Laboratoire International Associé LIA LYSM, the INdAM (Italy), the UMI-CNRS 3483, Laboratoire Fibonacci (Pisa) where this work has been completed under a CNRS delegation and the Centro de Giorgi in Pisa for various supports. 

\bibliographystyle{plain}
\bibliography{mybib}

\end{document}